\documentclass{amsart}
\usepackage{amssymb}
\usepackage{amsfonts}
\usepackage{amssymb}
\usepackage{amsmath}
\usepackage{amsthm}
\usepackage{enumerate}
\usepackage{tabularx}
\usepackage{centernot}
\usepackage{mathtools}
\usepackage{stmaryrd}
\usepackage{amsthm,amssymb}
\usepackage{etoolbox}
\usepackage{url}
\usepackage{tikz}
\usepackage{amssymb}
\usetikzlibrary{matrix}
\usepackage{tikz-cd}
\usepackage{tikz}
\usepackage{marginnote}
\definecolor{mygray}{gray}{0.85}
\usepackage[backgroundcolor=mygray,colorinlistoftodos,prependcaption,textsize=small]{todonotes}
\usepackage{xargs}                      

\usepackage[nocompress]{cite}
\usepackage[colorlinks,citecolor=blue,urlcolor=blue, linkcolor=blue]{hyperref}

\usepackage{float}
\newcommand{\pureindep}[1][]{%
	\mathrel{
		\mathop{
			\vcenter{
				\hbox{\oalign{\noalign{\kern-.3ex}\hfil$\vert$\hfil\cr
						\noalign{\kern-.7ex}
						$\smile$\cr\noalign{\kern-.3ex}}}
			}
		}\displaylimits_{#1}
	}
}

\newcommand{\indep}[2]{%
	\mathrel{
		\mathop{
			\vcenter{
				\hbox{%
					\oalign{
						\noalign{\kern-.3ex}\hfil$\vert$\hfil\cr
						\noalign{\kern-.7ex}
						$\smile$\cr\noalign{\kern-.3ex}
					}
				}
			}
		}^{\!\!\!\!\!#2}_{\!\!\hspace{-0.1em}#1}
	}
}

\newcommand{\displayindep}[2]{%
	\mathrel{
		\mathop{
			\vcenter{
				\hbox{%
					\oalign{
						\noalign{\kern-.3ex}\hfil$\vert$\hfil\cr
						\noalign{\kern-.7ex}
						$\smile$\cr\noalign{\kern-.3ex}
					}
				}
			}
		}^{\!\!\hspace{-0.1em}#2}_{\!\!\hspace{-0.1em}#1}
	}
}

\newcommand{\displayfindep}[2]{%
	\mathrel{
		\mathop{
			\vcenter{
				\hbox{%
					\oalign{
						\noalign{\kern-.3ex}\hfil$\vert$\hfil\cr
						\noalign{\kern-.7ex}
						$\smile$\cr\noalign{\kern-.3ex} 
					}
				}
			}
		}^{\!\hspace{-0.14em}#2}_{\!\!\hspace{-0.05em}#1}
	}
}

\newcommand{\mrm}[1]{\mathrm{#1}}

\renewcommand{\leq}{\leqslant}
\renewcommand{\geq}{\geqslant}

\newcommand{\specialonleq}{\leqslant_{\mrm{o}}^{\mrm{n}}}

\newcommand{\onleq}{\leqslant_{\mrm{o}}}
\newcommand{\nonleq}{\nleqslant_{\mrm{o}}}
\newcommand{\oleq}{\leqslant_{\mrm{o}}}

\newcommand{\noleq}{\nleqslant_{\mrm{o}}}
\newcommand{\hleq}{\leq_{\mrm{HF}}}

\newcommand{\type}{\mathrm{tp}}
\newcommand{\atype}{\mathrm{tp}^\mrm{qf}}
\newcommand{\ptype}{\mathrm{tp}^\mrm{p}}
\newcommand{\itype}{\mathrm{tp}^\mrm{i}}

\def\Ind{\setbox0=\hbox{$x$}\kern\wd0\hbox to 0pt{\hss$\mid$\hss}
	\lower.9\ht0\hbox to 0pt{\hss$\smile$\hss}\kern\wd0}
\def\Notind{\setbox0=\hbox{$x$}\kern\wd0\hbox to 0pt{\mathchardef
		\nn=12854\hss$\nn$\kern1.4\wd0\hss}\hbox to
	0pt{\hss$\mid$\hss}\lower.9\ht0 \hbox to 0pt{\hss$\smile$\hss}\kern\wd0}

\def\ind{\mathop{\mathpalette\Ind{}}}
\def\nind{\mathop{\mathpalette\Notind{}}}

\makeatletter
\def\subsection{\@startsection{subsection}{3}%
  \z@{.5\linespacing\@plus.7\linespacing}{.3\linespacing}%
  {\bfseries\centering}}
\makeatother

\makeatletter
\def\subsubsection{\@startsection{subsubsection}{3}%
  \z@{.5\linespacing\@plus.7\linespacing}{.3\linespacing}%
  {\centering}}
\makeatother

\makeatletter
\def\myfnt{\ifx\protect\@typeset@protect\expandafter\footnote\else\expandafter\@gobble\fi}
\makeatother

\renewcommand{\restriction}{ {\upharpoonright} }
\newtheorem{theorem}{Theorem}[section]
\newtheorem*{theorem*}{Main Theorem}
\theoremstyle{definition}
\newtheorem{corollary}[theorem]{Corollary}
\newtheorem{definition}[theorem]{Definition}
\newtheorem{lemma}[theorem]{Lemma}
\newtheorem{proposition}[theorem]{Proposition}
\newtheorem{example}[theorem]{Example}

\newtheorem{fact}[theorem]{Fact}

\newtheorem{remark}[theorem]{Remark}
\newtheorem{notation}[theorem]{Notation}
\newtheorem{context}[theorem]{Context}
\newtheorem{assumption}[theorem]{Assumption}
\newtheorem{technical}[theorem]{Technical assumption}

\usepackage{xcolor}

\newcounter{claimcounter}
\numberwithin{claimcounter}{theorem}

\begin{document}

\begin{abstract} Taking inspiration from \cite{tent, funk2,paolini&hyttinen}, we develop a general framework to deal with the model theory of open incidence structures. In this first paper we focus on the study of systems of points and lines (rank $2$). This has a number of applications, in particular we show \mbox{that for any of the following $\mathcal{C}$:}
\newline (1) $(k, n)$-Steiner systems (for $2 \leq k < n$);
\newline (2) generalised $n$-gons (for $n \geq 3$);
\newline (3) $k$-nets (for $k \geq 3$);
\newline (4) affine planes;
\newline (5) projective M\"obius, Laguerre and Minkowski planes;
\newline all the non-degenerate free structures in $\mathcal{C}$ are elementarily  equivalent and their common theory is decidable, strictly stable \mbox{and with no prime model.}
\end{abstract}

\title[{On the Model Theory of Open Incidence Structures}]{On the Model Theory of Open Incidence Structures: The Rank 2 Case}

\thanks{The authors were supported by project PRIN 2022 ``Models, sets and classifications", prot. 2022TECZJA. The first author was also supported by INdAM Project 2024 (Consolidator grant) ``Groups, Crystals and Classifications''. The second author was also supported by an INdAM post-doc grant. We wholeheartedly thank John Baldwin and Tapani Hyttinen for  useful discussions related to this paper. We are also thankful to Matteo Bisi, Domenico Zambella and one anonymous reviewer for comments on a prior draft of this paper.}

\author{Gianluca Paolini}
\author{Davide Emilio Quadrellaro}
\address{Department of Mathematics “Giuseppe Peano”,
	University of Torino,
	Via Carlo Alberto 10,
	10123 Torino, Italy.}
\email{gianluca.paolini@unito.it}
\email{davideemilio.quadrellaro@unito.it}

\address{Istituto Nazionale di Alta Matematica ``Francesco Severi'',
	Piazzale Aldo Moro 5
	00185 Roma, Italy.}
\email{quadrellaro@altamatematica.it}


\date{\today}
\maketitle


\section{Introduction}

As witnessed by the 1420 page book ``Handbook of Incidence Geometry" \cite{handbook}, the subject of incidence geometry is a vast one, with interactions with various other parts of mathematics, such as algebra, combinatorics, geometry and topology. In this article we develop an abstract framework to deal with the model theory of open incidence structures, and we use it to give model theoretic applications to cases of incidence geometries of \mbox{interest, focusing here on systems of points and lines.}

Originally, the notion of open projective plane was introduced by Hall in \cite{hall_proj} as a tool to study free projective planes, which he described as the analogue of free groups in the context of projective geometry. He defined \emph{the free projective plane of rank $k$} as the inductive closure of the finite configuration $\pi^k_0$ consisting of a line $\ell$, $k-2$ points on $\ell$ and two points off of $\ell$, under the following operations:
\begin{enumerate}[(a)]
	\item if two points are not joined by a line, then add a new line that joins them;
	\item if two lines are parallel, then add a new point intersecting them.
\end{enumerate}
To study the free projective planes, Hall defined the notion of \emph{open plane}: a plane $P$ is open if every finite configuration $C\subseteq P$ contains at least one line $\ell$ incident with at most two points in $C$, or it contains at least one point $p$ incident with at most two lines in $C$. Hall then proved that free projective planes are always open and, conversely, that finitely generated open projective planes are free (the assumption of finite generation is necessary, cf.~\cite{kope}). In \cite{paolini&hyttinen} Hyttinen and Paolini used Hall's notion of open projective plane (which is a first-order notion) to axiomatise the theory of non-degenerate (i.e., infinite) free projective planes, proving various model theoretic results, in particular stability and non-superstability of the theory.

The idea of free constructions does not apply only to projective planes, but to a vast number of cases of incidence geometries. The first generalisation of Hall's construction is probably due to Tits \cite{tits}, who introduced a notion of free completion in the context of \emph{generalised $n$-gons}, combinatorial structures which Tits introduced to give a geometric interpretation of groups of Lie type in rank $2$ (in fact, generalised $n$-gons are simply the rank $2$ case of spherical buildings). As in the case of projective planes (i.e., generalised $3$-gons), one can introduce a notion of \emph{open} generalised $n$-gon. Relying on this, in \cite{tent} Ammer and Tent used the theory of non-degenerate open generalised $n$-gons to axiomatise the theory of free generalised $n$-gons, generalising all the results of Hyttinen and Paolini \cite{paolini&hyttinen} from $n = 3$ to any $3 \leq n < \omega$. Apart from the increased level of generality, the methods from \cite{tent} were decidedly more elegant than those from \cite{paolini&hyttinen}, and made clearer the crucial role played by {\em Hrushovski constructions} in this setting. Our paper rests on the fundamental insights of \cite{tent}.

Given the previous work from \cite{paolini&hyttinen,tent}, it is natural to ask to what extent these results indicate a pattern, and whether, whenever in a given geometric context a notion of open incidence structure is well-defined, then it can be used to axiomatise the free incidence structures in that class. The starting point of this investigation is of course the analysis of systems of points and lines (as in the case of projective planes and generalised $n$-gons), which we will refer to as {\em the rank $2$ case}. To our knowledge, the most detailed and general study of rank $2$ free constructions in incidence geometry is contained in Chapter 13 ``Free Constructions'' by Funk and Strambach in the already mentioned ``Handbook of Incidence Geometry" \cite{handbook}. There it is shown that, in all naturally occurring theories of incidence structures $T$, there is a notion of free construction and a corresponding notion of free object of rank $k$ (in the sense of number of generators). In \cite[Ch.~13]{handbook} a general theory of free extensions is proposed for $T$ an arbitrary $\forall\exists$-theory of incidence structures, and this theory is tested on the following cases:
\begin{enumerate}[(1)]
	\item $(k, n)$-Steiner systems (for $2 \leq k < n$);
	\item generalised $n$-gons (for $n \geq 3$);
	\item $k$-nets (for $k \geq 3$);
	\item affine planes;
	\item projective M\"obius, Laguerre and Minkowski planes.
\end{enumerate}
\noindent In particular, Funk and Strambach provide a uniform definition of free completions in incidence geometry, and they also show that one can always define a notion of open configuration (generalising the definition of open projective plane considered above). Commenting on their results, Funk and Strambach  \cite[p.~741]{funk2} remark:

\begin{quote}
	{\small Besides a survey on significant results obtained thus far, the main purpose of this article is the development of a unifying treatment including all classes of incidence geometries which can be characterised by a set $\Sigma$ of axioms formulated in a first-order language $\mathcal{L}$ (e.g., projective planes, affine planes, generalised $n$-gons, Benz planes, etc.). By using rather simple model-theoretic tools, we can define the notions of \emph{(hyper-)free}, \emph{open}, \emph{confined}, \emph{closed}, \emph{{(hyper-)} free extensions}, and \emph{degenerate geometries} without knowing $\Sigma$ and $\mathcal{L}$ explicitly. To a very vast extent, we succeed in reformulating and proving the main results concerning free extensions within this general frame. The application of these results reduces to an easy verification of some model-theoretic conditions on the axioms. Thus it becomes clear that the real nature of free extensions in fact lies beyond geometry. We hope that this insight will contribute to stop splitting research on that subject.}
\end{quote}

\noindent Our present contribution is fully in this spirit. In this paper we develop an abstract model-theoretic framework that we believe will cover all concrete instances of free incidence structures and, using that, we extend the results from \cite{tent, paolini&hyttinen} to \emph{all examples of incidence geometries of rank $2$} known to us. In particular, this covers all concrete cases considered in \cite{funk2} and listed above. We take as a special sign of strength of our approach the fact that we were also able to apply our methods to cases of incidence geometries with notions of parallelism, namely affine planes and M\"obius planes (cf.~Sections~\ref{section:affine_planes} and \ref{last_section}). To the best of our knowledge this has not been considered elsewhere (apart from \cite{funk2}). Our main results can be summarised in the following theorem, which vastly generalises  \cite{tent, paolini&hyttinen}:

\begin{theorem*}\label{main_application} The theory $T^+$ of open non-degenerate models of $T$ is complete, decidable and strictly stable for any of the following theories $T$:
	\begin{enumerate}[(1)]
		\item $(k, n)$-Steiner systems (for $2 \leq k < n$);
		\item generalised $n$-gons (for $n \geq 3$);
		\item $k$-nets (for $k \geq 3$);
		\item affine planes;
		\item projective M\"obius, Laguerre and Minkowski planes.
	\end{enumerate}
Thus, for any of the above classes, all the non-degenerate free incidence structures in that class are elementarily equivalent. Further, the theory $T^+$ does not have a prime model, it is not model complete, and it does not admit quantifier elimination. Finally, for any $A, B, C \subseteq \mathfrak{M}$ we have that $B \ind_A C$ if and only if $\mrm{icl}(ABC) = \mrm{icl}(AB) \otimes_{\mrm{icl}(A)} \mrm{icl}(AC))$ (where $\mathfrak{M}$ is the monster model of $T^+$ and the right-hand part of the equivalence is defined in Section~\ref{sec_general}).
\end{theorem*}

Our paper has two parts, an abstract one (Section~\ref{sec_general}) and one that deals with concrete applications (Section~\ref{sec:application}). The previous theorem puts together our general results from Section~\ref{sec_general} with the applications from Section~\ref{sec:application}.  We stress here however that all our results in Section \ref{sec_general} are stated in a purely axiomatic fashion (following the example of Baldwin and Shi's seminal paper on stable generic structures in Hrushovski constructions \cite{baldwin_generic}), and we expect to find applications of our methods beyond the scope of incidence geometries of rank $2$. In a work in preparation we intend to explore the general rank $k$ case, with particular focus on connections to buildings, and thus naturally to the fundamental references \cite{ziegler_free, tent2}. But as the present paper is already of considerable length and covers many case studies we decided to split the treatment of the subject into two separate papers.

In addition to the case studies listed above, we decided to include here a further one (cf.~Section~\ref{sec:n-open-grpahs}) with a very different motivation and history, i.e., the case of {\em $n$-open graphs}. These structures were recently introduced by Baldwin, Freitag and Mutchnik in \cite{baldwin_generic} to give the first example, for every $n < \omega$, of a stable theory $T_n$ admitting a type $p$ such that $F_{Mb}(p) = n$, and so that $F_{Mb}(p) \leq n$ for all types $p$. This case fits perfectly in our setting and shows that, apart from the intrinsic geometric interest, our approach could be exploited by model theorists to give counterexamples to conjectures in classification theory. From our point of view, an additional key element of interest is that these graphs make for an example where the notion of openness applies, but where strictly speaking there is no notion of free extension (we elaborate this further in Section \ref{sec_general}). Additionally, $n$-open graphs are one-sorted structures and, with the exception of the case $n=1$ of the free pseudoplane, they do not have a natural interpretation in terms of geometries of rank $2$. We show that our framework also covers these structures and we prove a version of the theorem above (cf.~Corollary \ref{corollary:ngraphs}). We stress that our results in this case are not novel, but provide nonetheless an interesting application of our techniques beyond the context of free incidence geometries of rank $2$.

We now come to the technical side of the story, which is intended mainly for the model theorists, but also explains to the general reader the techniques used to establish our results. The fundamental technical tool behind this study is the use of {\em Hrushovski constructions}. The relevance of this was already realised by the first author and Hyttinen in \cite[Section~8]{paolini&hyttinen}, but the idea was fully developed only in \cite{tent}. In this study we owe a great deal to the elegant treatment of the subject presented in \cite{tent} in the case of generalised $n$-gons. At the same time we also make a crucial use of the notion of $\mrm{HF}$-order, which was introduced in \cite{paolini&hyttinen} as a generalisation of the existing notion from the incidence geometry literature (cf.~e.g.~\cite{sieben}). These $\mrm{HF}$-orders are an important technical tool available in open incidence structures, which make for a very handy instrument when dealing with these kinds of Hrushovski constructions and, moreover, also play a fundamental role in the proof of stability.

What we feel to be the most interesting contribution of this study from the technical point of view is the isolation of some axiomatic conditions which make the whole machinery work. The axioms are actually split into different sets, namely the finitary Conditions \ref{context} and the assumptions on the theory $T^+$ from \ref{main_th}(\hyperref[completeness_axiom]{C1})-(\hyperref[the_K_homogeneous_lemma]{C2}), \ref{ass:hf_closure}(\hyperref[the_hf_axiom]{C3}), \ref{assumption:no-superstability_2}(\hyperref[CP]{C4}) and \ref{assumptions:no_prime}(\hyperref[F=C]{C5})-(\hyperref[delta-rank]{C7}). The set of axioms from \ref{context} isolates some conditions on the Hrushovski class of finite structures $(\mathcal{K}, \oleq)$ under consideration, expanding the axiomatic framework of Baldwin and Shi in \cite{baldwin_generic} to the case of open incidence structures. Crucially, some of the axioms from \ref{context} (most notably \ref{context}(\ref{condition:I})) are {\em false} in most Hrushovski constructions arising from a predimension. We believe that this is not a defect of our approach but rather one of its main strengths, as our axioms are exactly what make things work in the more specific context of open incidence structures (a context which is specific but general enough to cover all the known cases of free incidence structures in rank $2$). The second set of axioms, namely \ref{main_th}(\hyperref[completeness_axiom]{C1})-(\hyperref[the_K_homogeneous_lemma]{C2}), isolates some conditions on infinite models of the intended axiomatisation of the generic for $(\mathcal{K}, \oleq)$, namely, the theory $T^+$ defined in \ref{the_theory}, which in all concrete cases will correspond to the theory of  open incidence structures of the given type (generalising \cite{tent, paolini&hyttinen}). The assumptions from \ref{context} and \ref{main_th}(\hyperref[completeness_axiom]{C1})-(\hyperref[the_K_homogeneous_lemma]{C2}) are enough by themselves to prove that the theory $T^+$ is complete, and thus they already suffice to show that in all intended applications the free incidence structures are all elementarily equivalent. Assumption \ref{ass:hf_closure}(\hyperref[the_hf_axiom]{C3}) is the key ingredient in the proof of stability and in the characterisation of forking independence, while Assumption \ref{assumption:no-superstability_2}(\hyperref[CP]{C4}) isolates some combinatorial conditions determining the failure of superstability. The additional Assumptions \ref{assumptions:no_prime}(\hyperref[F=C]{C5})-(\hyperref[delta-rank]{C7}) are needed for our additional results: we use them to establish that $T^+$ does not have a prime model and it is not model complete. We crucially notice that, while the assumptions from \ref{context}, and \ref{main_th}(\hyperref[completeness_axiom]{C1})-(\hyperref[the_K_homogeneous_lemma]{C2}), \ref{ass:hf_closure}(\hyperref[the_hf_axiom]{C3}) hold in all the examples that we consider in Section \ref{sec:application}, the assumptions from \ref{assumption:no-superstability_2}(\hyperref[CP]{C4}) and \ref{assumptions:no_prime}(\hyperref[F=C]{C5})-(\hyperref[delta-rank]{C7}) are not always valid, and they thus provide us with a {\em dividing line} for our applications. On the one hand, we have the incidence geometries from our Main Theorem, which are all not superstable and with no prime model; on the other hand, there are structures such as the free pseudoplane which, as well-known, is $\omega$-stable and prime. Our abstract framework covers both cases.

The structure of the paper reflects our axiomatic treatment of the subject. After recalling some preliminary definitions and notational conventions in Section \ref{section:preliminaries}, in Section~\ref{sec_general} we present our general framework and we prove all our main results about $T^+$ from the corresponding set of assumptions. We provide  a general proof of completeness and stability of $T^+$, we characterise forking independence,  and we show that $T^+$ is not superstable and does not have a prime model. In the subsequent Section \ref{sec:application} we first consider the simpler case of $n$-open graphs (introduced in \cite{baldwin_generic}), and then we deal separately with each case of incidence structures from (1)-(5) of our Main Theorem above, verifying each of the axiomatic conditions isolated in Section~\ref{sec_general}. We believe that our abstract approach confirms Funk and Strambach's fundamental verdict that ``the real nature of free extensions lies beyond geometry''.

\section{Preliminaries and notations}\label{section:preliminaries}

We introduce in this section some general notions that we use to define the framework at the heart of this article. We also fix some notation and terminology that we will follow throughout the paper. Firstly, we introduce the general context of the paper and some basic notational convention.

\begin{context}
	In this article we shall always work in a finite multi-sorted relational language $L$. We denote the substructure relation between $L$-structures as $A \leq B$. We let $T$ be an $L$-theory and $T_\forall$ its universal fragment. We let \mbox{$(\mathcal{K}, \oleq)$} be a hereditary class of finite models of $T_\forall$ and $\oleq$ be a strong substructure relation---meaning that $A \oleq B$ implies $A \leq B$ and $A \oleq B \in \mathcal{K}$ implies $A \in \mathcal{K}$.
\end{context}

\begin{notation}\label{atomic_types}\label{notational_conventions}
	Let $\kappa$ be a cardinal. We write $A\subseteq_\kappa B$ if $A\subseteq B$ and $|A|<\kappa$. Thus $A\subseteq_\omega B$ means that $A$ is a finite subset of $B$. We write $A^\kappa$ for the set of all sequences of elements from $A$ of length $\kappa$, and we let $A^{<\kappa} = \bigcup_{\alpha<\kappa} A^\alpha$.  Let $M$ be an $L$-structure, $\kappa$ a cardinal, $\bar{a}\in M^{<\kappa}$ and $A\subseteq M$; we write $\mathrm{tp}_M(\bar{a}/A)$ for the type of $\bar{a}$ in $M$ with parameters from $A$ and  $\mathrm{tp}^{\mrm{qf}}_M(\bar{a}/A)$ for the quantifier-free type of $\bar{a}$ in $M$ with parameters from $A$. Given a model $M$ and a subset $A\subseteq M$ we write $\mrm{acl}_M(A)$ for the algebraic closure of $A$ in $M$ and  $\mrm{acl}^{\mrm{qf}}_M(A)$ for the algebraic closure of $A$ in $M$ computed in the language with only quantifier-free formulas. We often omit the subscript $M$ when it is clear from the context. Finally, if $A\subseteq B,C$ and there is an isomorphism $f:B\to C$ fixing $A$, then we often just write $B\cong_A C$, and we write $f:B\cong_A C$ when we want to specify the underlying map.
\end{notation}

As we spelled out in the introduction, our approach in this paper is mainly graph-theoretical. We thus recall some standard notions from graph theory, which we will often employ in this work. In this paper by a graph we mean a structure of the form $(G, E)$ with $E$ a binary, irreflexive and symmetric relation on $G$.

\begin{definition}\label{def:graphs}
	Let $(G,E)$ be a graph, then we define the following:
	\begin{enumerate}[(1)]
		\item the \emph{valency} of an element $a\in G$ is the size of the set $\{b\in G : aEb \}$ of neighbours of $a$;
		\item  the \emph{distance} $d(a,b) = d(a, b/G)$ between two elements $a,b\in G$ is the length $n$ of the shortest path $a=c_0Ec_1E\dots E c_{n-1}Ec_n=b$ in $G$, and it is $\infty$ if there is no such path;
		\item the \emph{girth} of a graph is the length of its shortest cycle;
		\item the \emph{diameter} of a graph is the maximal distance between any two elements;
		\item the graph $G$ is \emph{bipartite} if there are $A,B\subseteq G$ such that $A\cap B=\emptyset$, $A\cup B= G$, and every edge has one vertex in $A$ and one in $B$.
	\end{enumerate}
\end{definition}

We introduce some less standard notions, mostly from \cite{funk2}, which are motivated by the geometric nature of the theories considered in this work. The distinction between incidence and parallelism symbols is essentially from \cite[pp.~743,~749]{funk2}.

\begin{definition}\label{local_equivalence_relations} \;
	\begin{enumerate}[(1)]
		\item Let $R\in L\cup \{= \}$ be a relation symbol of arity $n\geq 2$. We say that $R$ is a \emph{local equivalence relation in the theory $T$} if for any $L$-structure $A\models T$ and tuple $a_1,\dots,a_{n-2}\in A$  we have that $R(x,y,a_1,\dots,a_{n-2})$ induces an equivalence relation on $A$ (notice that in the case $n = 2$ the quantifier ``for every tuple $a_1,\dots,a_{n-2}\in A$'' is vacuous). If $R\in L\cup \{= \}$ is not a local equivalence relation, then we say that it is an \emph{incidence relation in the theory $T$}. We omit the reference to $T$ when it is clear from the context. 
		\item We denote  by $L_{\mrm{p}}$ the subset of $L$ consisting of all local equivalence relations in $T$, and by $L_{\mrm{i}}$ the the subset of $L$ consisting of  all incidence relations in $T$. 		
		\item  Let $M$ be an $L$-structure, $\bar{a}\in M^{<\kappa}$ and $B\subseteq M$.  The \emph{parallelism type} $\ptype_M(\bar{a}/B)$  of $\bar{a}$ in  $M$ with parameters in $B$ is the quantifier-free type of $\bar{a}\in M^{<\kappa}$ with parameters in $B$ computed in the restricted signature $L_{\mrm{p}}$. The \emph{incidence type} $\itype_M(\bar{a}/B)$  of $\bar{a}$ in $M$ with parameters in $B$ is the quantifier-free type of $\bar{a}\in M^{<\kappa}$ with parameters in $B$ computed in the restricted signature  $L_{\mrm{i}}$. We often omit the subscript $M$ when it is clear from the context.
	\end{enumerate}
\end{definition}

\begin{remark}
	For simplicity, notice that we always assumed that in a local equivalence relation $R$ the first two variables represent the arguments of the equivalence relation, and the rest are the parameters. Notice also that the equality $=$ is a (local) equivalence relation in all theories $T$. We also have that $\mathrm{tp}_M^{\mrm{qf}}(\bar{a}/B)$ is uniquely determined by  $\itype_M(\bar{a}/B) \cup \ptype_M(\bar{a}/B)$, and that quantifier-free formulas with equalities are contained in $\ptype_M(\bar{a}/B)$.  We stress that the sentence stating that a relation $R\in L$ is a local equivalence relation is universal, and so $T$ and $T_\forall$ agree on the distinction between $L_{\mrm{i}}$ and $L_{\mrm{p}}$.  We will often use this fact implicitly.
\end{remark}

We will use the notions of parallelism type $\ptype(\bar{a}/B)$ and incidence type $\itype(\bar{a}/B)$ especially in the definition of the Gaifman closure \ref{def:interior_closure}. We adapt to our current setting the usual notion of Gaifman graph (cf.~\cite[p.~26]{ebbinghaus}).

\begin{definition}\label{gaifman_graph} Let  $M\models T_\forall$ and  let $L_{\mrm{i}}$,  $L_{\mrm{p}}$ be the sublanguages of $L$ consisting of incidence relations, and of local equivalence relations, respectively.  
	\begin{enumerate}[(1)]
		\item The \emph{incidence Gaifman graph} $G_{L_{\mrm{i}}}(M)$ of $M$ is the graph whose set of vertices is $M$ and which has an edge $E(a,b)$ between $a,b\in M$  whenever $M\models R(\sigma(a,b, \bar{c}))$ for some incidence relation $R \in L_{\mrm{i}}$, for some tuple $\bar{c}\in M^{<\omega}$ of the suitable arity, and for some permutation $\sigma$.
		\item The \emph{parallelism Gaifman graph} $G_{L_{\mrm{p}}}(M)$ of $M$ is the graph whose set of vertices is $M$ and which has an edge $E(a,b)$ between $a,b\in M$  whenever $M\models R(a,b, \bar{c})$ for some parallelism relation $R \in L_{\mrm{p}}$ and for some tuple $\bar{c}\in M^{<\omega}$.
		\item The \emph{(full) Gaifman graph} $G(M)\coloneqq G_{L}(M)$ of $M$ is the graph whose set of vertices is $M$ and which has an edge $E(a,b)$ between $a,b\in M$  whenever there is an edge between $a$ and $b$ in either  $G_{L_{\mrm{i}}}(M)$ or  $G_{L_{\mrm{p}}}(M)$.
		\item For $L'\in \{L_{\mrm{i}},L_{\mrm{p}}, L\}$ The \emph{distance} $d_{L'}(a,b/M)$ between two elements $a,b\in M$ is their distance in the Gaifman graph $G_{L'}(M)$.  We write $d(a,b/M)$ for their distance in $G(M)$,  $d_{\mrm{i}}(a,b/M)$ for their distance in $G_{\mrm{i}}(M)$ and $d_{\mrm{p}}(a,b/M)$ for their distance in $G_{\mrm{p}}(M)$. We simply write $d(a,b)$, $d_{\mrm{i}}(a,b)$ and $d_{\mrm{p}}(a,b)$ when the structure $M$ is clear from the context.
		\item For $L'\in \{L_{\mrm{i}},L_{\mrm{p}}, L\}$, we extend the notion of distance to arbitrary finite subsets $A,B\subseteq M$ by letting
		\[  d_{L'}(A,B)\coloneqq   \mrm{min}_{a\in A} \Bigl( \mrm{min}_{b\in B} \bigl(d_{L'}(a,b)  \bigr)    \Bigr),    \]
		and we let the distance $d_{L'}(\bar{a},\bar{b})$ between two finite tuples $\bar{a},\bar{b}\in M^{<\omega}$  be the same as their distance as subsets.
	\end{enumerate}
\end{definition}

We use the previous definitions to define the notion of Gaifman closure. This was introduced in \cite[Def. 2]{funk2} under the name of (restricted) interior construction.

\begin{definition}[Gaifman Closure]\label{def:interior_closure}
	Let $ A\subseteq M\models T_\forall$,  the \emph{Gaifman closure $\mrm{gcl}_M(A)$ of $A$ in $M$} is defined as $\mrm{gcl}_M(A)=\mrm{gcl}^{\mrm{i}}_M(A)\cup \mrm{gcl}^{\mrm{p}}_M(A)$, where:
	\begin{enumerate}[(1)]
		\item  $\mrm{gcl}^{\mrm{i}}_M(A) = \{b\in M : d_{\mrm{i}}(A,b)\leq 1    \}$;
		\item $\mrm{gcl}^{\mrm{p}}_M(A)$ is any subset of $M$ satisfying the two following constraints:
		\begin{enumerate}[(a)]
			\item for any $b\in M$ with $d_{\mrm{p}}(A,b)\leq 1$, there is $b'\in \mrm{gcl}^{\mrm{p}}_M(A)$ such that
			\[ M\models R(b,a,\bar{c}) \; \Longleftrightarrow  M\models R(b',a,\bar{c}) \]
			for all $R\in L_{\mrm{p}}$, $a\in A$ and $\bar{c}\in M^{<\omega}$;
			\item for all $b,b'\in \mrm{gcl}^{\mrm{p}}_M(A)$ there are some $R\in L_{\mrm{p}}$, $a\in A$, and $\bar{c}\in M^{<\omega}$ s.t.
			\[ M\models R(b,a,\bar{c}) \; \text{ and } \;  M\models \neg R(b',a,\bar{c}). \]
		\end{enumerate}
		In other words, $\mrm{gcl}^{\mrm{p}}_M(A)$ is any subset of  $ \{b\in M : d_{\mrm{p}}(A,b)\leq 1 \} $ containing exactly one element for every distinct parallelism type with parameters in $M$.
	\end{enumerate}
\end{definition}

\begin{remark}\label{non-uniqueness:interior-closure}
	Clearly, we have that  $A\subseteq \mrm{gcl}^{\mrm{i}}_M(A)$. Moreover,  since the equality symbol ``='' belongs to $L_{\mrm{p}}$, we also have that $A\subseteq \mrm{gcl}^{\mrm{p}}_M(A)$, since every $a\in A$ is the unique element satisfying the formula $x=a$. We stress that by the definition of $\mrm{gcl}^{\mrm{p}}_M(A)$ the Gaifman closure of a set is \emph{not unique}. However, we always work in a context $(\mathcal{K}, \oleq)$ where the choice of one Gaifman closure over another does not really make a difference. This is captured by Condition \ref{context}(\ref{condition:K}) below. In particular, statements like $\mrm{gcl}_A(B)=C$ are true only \emph{modulo} a specific choice for $\mrm{gcl}_A(B)$, and are otherwise true only up to isomorphism.
\end{remark}

We conclude this section by introducing the notion of \emph{free amalgam} in $(\mathcal{K},\oleq)$, which generalises the usual notion of free amalgamation of graphs for structures with also local equivalence relations.

\begin{definition}[Free amalgam]\label{general_free_amalgam}
	Let $A,B,C\models T_\forall$ with $A\subseteq B$, $A\subseteq C$ and $A=B\cap C$, the \emph{free amalgam} of $B$ and $C$ over $A$ is the structure $B\otimes_A C$ such that:
	\begin{enumerate}[(1)]
		\item the domain of $B\otimes_A C$ is $B\cup C$;
		\item every incidence symbols $R\in L_{\mrm{i}}$ is interpreted  by letting $R^{B\otimes_A C}= R^B\cup R^C$;
		\item every parallelism symbols $P\in L_{\mrm{p}}$ is interpreted by letting $ B\otimes_A C\models P(b,c,\bar{d})$, for $b,c\in B\cup C$ and $\bar{d}\in (B\cup C)^{<\omega}$ if one of the following conditions hold:
		\begin{enumerate}[(i)]
			\item $B\models P(b,c,\bar{d})$ (so in particular $b,c\in B$ and $\bar{d}\in B^{<\omega}$);
			\item $C\models P(b,c,\bar{d})$ (so in particular $b,c\in C$ and $\bar{d}\in C^{<\omega}$);
			\item there is some $a\in A$ such that $ B\models P(b,a,\bar{d}) $  and $ C\models P(a,c,\bar{d})$ (so in particular $\bar{d}\in A^{<\omega}$).
		\end{enumerate}
	\end{enumerate}
\end{definition}

\begin{remark}\label{remark_pushout}
	By  Definition \ref{general_free_amalgam} the amalgam $B\otimes_A C$ is always a well-defined $L$-structures but it is not necessarily a models in $\mathcal{K}$. However, the key property of the free amalgam is that, for $A,B,C\models T_\forall$ with $A\subseteq B$, $A\subseteq C$ and $A=B\cap C$, $B\otimes_{A} C$ satisfies $\mrm{gcl}_{B\otimes_A C}(B\setminus AC)=\mrm{gcl}_{B}(B\setminus A)$ and $\mrm{gcl}_{B\otimes_A C}(C\setminus AB)=\mrm{gcl}_{C}(C\setminus A)$. It will then follow from Condition (\ref{condition:K}) below that $B\oleq B\otimes_A C$ holds if and only if $A\oleq C$ and that $C\oleq B\otimes_A C$ holds if and only if $A\oleq B$.
\end{remark}

\section{The general framework}\label{sec_general}

In this paper we investigate the first-order theory of open incidence structures of rank $2$, such as open projective planes, open Steiner systems, and similar classes of geometric objects. To cover all these examples we  proceed in an axiomatic fashion, and we start in Section \ref{subsec:1} by introducing a class of finitary conditions (Context \ref{context}) that carves out a special subclass of Hrushovski constructions that covers all the geometric theories which we are interested in. In Section \ref{subsec:3} we introduce the theory $T^+$, which is the main object of investigation of the article, and prove that that generic $M_\star$ of the class $(\mathcal{K},\oleq)$ is always a model of $T^+$. In Section \ref{subsec:4}  we state the crucial assumptions \ref{main_th}(\hyperref[completeness_axiom]{C1})-(\hyperref[the_K_homogeneous_lemma]{C2}), which we later verify  for all our concrete examples, and we prove the first main result of the paper, i.e., that the theory $T^+$ is complete. In Section \ref{subsec:2} we introduce the main technical device of our article, namely that of $\mrm{HF}$-orders, and the key assumption \ref{ass:hf_closure}(\hyperref[the_hf_axiom]{C3}) on the existence of a $\mrm{HF}$-closure. In Section \ref{subsec:5} we derive from the former conditions the stability of the theory $T^+$, and we additionally characterise the relation of forking independence in its models. Finally, in Sections \ref{subsec:no_superstability} and \ref{subsec:no_prime} we identify the additional assumptions \ref{assumption:no-superstability_2}(\hyperref[CP]{C4}) and \ref{assumptions:no_prime}(\hyperref[F=C]{C5})-(\hyperref[delta-rank]{C7}), which entail respectively that $T^+$ is not superstable, and that it does not have a prime model, it is not model complete and does not have quantifier elimination.

\subsection{Finitary conditions}\label{subsec:1}

As we laid out in the previous section, we always work in a finite multi-sorted relational language $L$. We also fix a hereditary class  $\mathcal{K}$ of finite models of $T_\forall$ and a relation $\oleq$ between elements of $\mathcal{K}$. We refer to extensions $A \oleq B$ as \emph{strong extensions} (or also \emph{open extensions}) and we provide the following definitions.

\begin{definition}\label{def_extensions} For $A \subseteq B\in \mathcal{K}$, we define the following notions:
	\begin{enumerate}[(1)]
		\item  we say that $B$ is \emph{open over $A$} if $A\oleq B$, and that $B$ is \emph{open} if $\emptyset\oleq B$;
		\item  we say that $B$ is \emph{closed over $A$} if $A\noleq B$, and that $B$ is \emph{closed} if $\emptyset\noleq B$;
		\item we say that $A\leq B$ is a \emph{minimal strong extension} (also \emph{minimal open extension}) if $A\oleq B$ and there is no $A \subsetneq C \subsetneq B$ such that $A \oleq C \oleq B$;
		\item we say that $B$ is \emph{confined over $A$} if $A\noleq B$ and there is no $A \subsetneq C \subsetneq B$ such that $A \noleq C$, i.e., $A\leq B$ is a minimal closed extensions of $A$;
		\item we say that $A\leq B$ is a \emph{$\mathcal{K}$-algebraic extension} if there is some $n < \omega$ such that, if $A \leq  C\in \mathcal{K}$ then $C$ contains at most $n$ many disjoint copies of $B$ over $A$;
		\item we say that $(A,B)$ is a \emph{one-element extension} if $|B\setminus A|=1$;
		\item\label{def_trivial} we say that $(A,Ab)$ is a \emph{trivial extension} if $d_{Ab}(A,b)=\infty$, i.e., $b$ has no edge with elements from $A$ in the full Gaifman graph $G(Ab)$.
	\end{enumerate}
\end{definition}

If $f:A\to B$ is an embedding between structures in $\mathcal{K}$ satisfying $f(A)\oleq B$ then we say that $f$ is a \emph{strong embedding} (also \emph{open embedding}).  We say that $(\mathcal{K}, \oleq)$ has the \emph{amalgamation property} if for all $A,B, C\in \mathcal{K}$ with $A\oleq B$ and $A\oleq C$, there is a structure $D\in \mathcal{K}$ and two strong embeddings $f:B\to D$ and $g:C\to D$ such that $f\restriction A=g\restriction A=\mrm{id}_A$. In this work, we shall work under some more specific form of amalgamation, which is a weaker variant of the \emph{sharp amalgamation property} defined in \cite[Def.~2.31]{baldwin_generic}. We first introduce the notion of $n$-strong substructure, again from  \cite[Def.~2.26]{baldwin_generic}.

\begin{definition}\label{def:sharp_amalgamation}
For  $A,B\in \mathcal{K}$ we say that $(A,B)$ is a \emph{$n$-strong extension} and write $A\leq_n B$ if for every $C\subseteq B$ with $|C\setminus A|\leq n$ we have $A\oleq C$.
\end{definition} 

\noindent Since in this article we mostly deal with classes $(\mathcal{K},\oleq)$ which admits both non-algebraic and also algebraic strong extensions $A\oleq B$, we introduce the following version of the amalgamation property. By Remark~\ref{remark_pushout} it is immediate to see that, in the setting from \ref{context}, the following algebraic amalgamation property entails that $(\mathcal{K}, \oleq)$ has the amalgamation property.

\begin{definition}\label{algebraic_number}
	Let $A\oleq  B \in \mathcal{K}$ be a minimal strong extension. If $A\oleq B$ is $\mathcal{K}$-algebraic (cf.~\ref{def_extensions}), then the \emph{algebraic degree} of $(A,B)$ is the greatest number $n<\omega$ such that every $C\in \mathcal{K}$ contains at most $n$ many disjoint copies of $B$ over $A$. If $A\oleq B$ is not $\mathcal{K}$-algebraic, then the \emph{algebraic degree} of $(A,B)$ is $\infty$.
\end{definition}

\begin{definition}\label{condition:amalgam}
	We say that  $(\mathcal{K}, \oleq)$ has the \emph{algebraic amalgamation property} if it satisfies the following condition. Let $A,B,C\in \mathcal{K}$ with  $A=B\cap C$, $A\oleq B$ a minimal strong extension and $A\leq_{|B\setminus A|} C$, then the following hold.
		\begin{enumerate}[(1)]
		\item If $A\oleq B$ has algebraic degree $n\in [1,\infty]$ and $C$ contains $<n$ disjoint copies of $B$ over $A$, then $B\otimes_A C\in \mathcal{K}$.
		\item If $A\oleq B$ is $\mathcal{K}$-algebraic, $C$ contains a copy of $B$ over $A$ and $A\oleq C$ then there is a strong embedding $f:B\to C$.
		\end{enumerate}
\end{definition} 

We next identify the following axiomatic setting, which abstracts the fundamental properties of (partial) open incidence structures. In particular, in this article we shall always assume that $(\mathcal{K}, \oleq)$ satisfies the following finitary conditions. These extend the axiomatic approach for Hrushovski constructions laid out in \cite{baldwin_generic}. 

\begin{context}\label{context} We let  $(\mathcal{K}, \oleq)$ be a hereditary class $\mathcal{K}$ of finite models of $T_\forall$ such that $\oleq$ is a relation with the following properties:
	\begin{enumerate}[(A)]
		\item\label{condition:A} if $A \in \mathcal{K}$, then $A \oleq A$;
		\item\label{condition:B} if $A \oleq B$, then $A$ is a substructure of $B$ (written $A\leq B$);
		\item\label{condition:C} if $A, B, C \in \mathcal{K}$ and $A \oleq B \oleq C$, then $A \oleq C$;
		\item\label{condition:D} $\emptyset \in \mathcal{K}$ and $\emptyset \oleq A$, for all $A \in \mathcal{K}$;
		\item\label{condition:E} if $A, B, C \in \mathcal{K}$, $A \oleq C$, and $B$ is a substructure of $C$, then $A \cap B \oleq B$;
		\item\label{condition:F}\label{preservation_iso} if $A \leq B\in \mathcal{K}$, $B' \in \mathcal{K}$, $A \oleq B$ and $f: B \cong_A B'$, then $A \oleq B'$;
		\item\label{condition:G} $(\mathcal{K}, \oleq)$ has the algebraic amalgamation property (cf.~\ref{def:sharp_amalgamation});
		\item\label{condition:H} if $A \leq B \in \mathcal{K}$ and $B$ is confined over $A$, then $(A, B)$ is $\mathcal{K}$-algebraic (cf.~\ref{def_extensions});
		\item\label{condition:I} there is a  non-empty  finite set $\mathbf{X}_\mathcal{K}\subseteq \omega$ whose elements are different from 0 and such that, if $A \oleq B \in \mathcal{K}$ and $A\neq B$, then there are $n\in \mathbf{X}_\mathcal{K}$ and $(b_1,\dots,b_n)\in (B\setminus A)^n$ such that $B\setminus\{b_1,\dots,b_n\}\oleq B$;
		\item\label{condition:J} there is $\mathbf{n}_\mathcal{K} < \omega$ such that, if  $A \oleq A b\in \mathcal{K}$ is a one-element extension, then $|\mrm{gcl}_{Ab}(b)\setminus \{b\}|\leq \mathbf{n}_\mathcal{K}$ (recall the definition of $\mrm{gcl}$ from \ref{def:interior_closure});
		\item\label{condition:K} for all $A\leq B\in \mathcal{K}$, if  $A\cap \mrm{gcl}_{B}(B\setminus A)\oleq \mrm{gcl}_{B}(B\setminus A)$ then $A \oleq B$;
		\item\label{condition:L} for all $A\in \mathcal{K}$, if  $(A,Ab)$ is trivial (in the sense of \ref{def_extensions}(\ref{def_trivial})), then $A\oleq Ab\in \mathcal{K}$.
	\end{enumerate}
\end{context}

\begin{remark} \label{explanation:finitary_conditions}
	We stress that the conditions from \ref{context} can be split in two groups.  Conditions \ref{context}(\ref{condition:A})-(\ref{condition:F}) are standard and hold also in Hrushovski constructions arising from a \emph{predimension} $\delta$ (cf.~\cite{baldwin_generic}).  As already remarked, Condition~\ref{context}(\ref{condition:G}) is a strengthening of amalgamation that builds on \cite{baldwin_generic}, and that holds in all incidence geometries that we consider later in Section~\ref{sec:application}. Condition \ref{context}(\ref{condition:H}) says that confined configurations are always algebraic, a fact which usually holds also in the setting of Hrushovski constructions arising from a predimension considered in \cite{baldwin_generic}, cf.~ \cite[Corollary 3.20]{baldwin_generic}). As we show later in \ref{acl_open_rk}, this has the important consequence that every extension of an algebraically closed set (in a model of $T^+$) is strong.  On the contrary, Conditions \ref{context}(\ref{condition:I})-(\ref{condition:K}) are quite specific to our current combinatorial setting.  Condition \ref{context}(\ref{condition:I}) means that every strong extension $A\oleq B$ in $\mathcal{K}$ can be build \textit{via} a series of strong extensions which are bounded in size. Conditions \ref{context}(\ref{condition:J})-(\ref{condition:K}) capture the essential properties of the Gaifman closure in the class of geometries we are investigating. Condition \ref{context}(\ref{condition:J}) essentially says that the Gaifman closure of any strong extension is bounded:  if $A\oleq B$, then clearly $\mrm{gcl}_{B}(B\setminus A)= \bigcup_{b\in B\setminus A}\mrm{gcl}_{B}(b)$, and thus by (\ref{condition:J}) we have that $|\mrm{gcl}_{B}(B\setminus A)|\leq (\mathbf{n}_\mathcal{K} \cdot |B\setminus A|) + |B\setminus A|$. Together, Conditions \ref{context}(\ref{condition:I}) and \ref{context}(\ref{condition:J}) have the important consequence that the notion of  $\mrm{HF}$-order becomes first-order expressible (cf.~\ref{def_HF_order} and \ref{hf_first_order}). Finally, Condition \ref{context}(\ref{condition:K}) essentially says that, in order to determine if an extension $(A,B)$ is strong, it is enough to look at the subextension induced by the Gaifman closure of $B\setminus A$. Together with Condition \ref{context}(\ref{condition:J}), this has the important consequence that whether an extension $A\oleq B$ is strong depends on a subset of $B$ of size $\leq (\mathbf{n}_\mathcal{K} \cdot |B\setminus A|) + |B\setminus A| $. Finally, Condition (\ref{condition:L}) makes sure that trivial extensions of any sort are strong.
\end{remark}

\begin{example}\label{example:projective_planes}
	To provide some intuition about the abstract setting from \ref{context}, we explain how the case of open projective planes from \cite{paolini&hyttinen} fits into this framework. Recall that a \emph{partial projective plane} is a set of lines and points such that any two points are incident with at most one line, and any two lines are incident with at most one point. A partial projective plane $P$ is \emph{open} if every finite subset $A\subseteq P$ contains either a point incident with at most two lines, or a line incident with at most two points. We let $\mathcal{K}$ be the class of finite, partial, open projective planes, and we write $A\oleq B$ if every non-empty $C\subseteq B\setminus A$ contains at least one element incident to at most two elements from $AC$. Then, Condition \ref{context}(\ref{condition:H}) boils down to the fact that minimal extensions $B\supsetneq A$ where each element $b\in B\setminus A$ has at least three incidences are algebraic over $A$, which can be easily verified. Condition \ref{context}(\ref{condition:I})  holds immediately for $\mathbf{X}_\mathcal{K}=\{1\}$, since every finite open partial projective plane can be obtained via a sequence of one-element extensions, by adding points incident with at most two lines and lines incident to at most two points (this was shown already by Siebenmann in \cite[Lem.~1]{sieben}). Similarly, Condition \ref{context}(\ref{condition:J}) is true by letting $\mathbf{n}_\mathcal{K}=2$, since whenever $A\oleq Ab$ we have that $b$ is incident to at most two elements from $A$, i.e., $|\mrm{gcl}_{Ab}(b)\setminus \{b\}|\leq 2$.  Condition \ref{context}(\ref{condition:K}) holds because whether $A\oleq B$ depends exclusively on the number of incidences of elements in $B\setminus A$, which is clearly preserved when restricting to $ \mrm{gcl}_{B}(B\setminus A)$. Finally, Condition \ref{context}(\ref{condition:L}) holds because we can always extend a plane $A\in \mathcal{K}$ with either a point or a line which bears no incidence with elements in $A$.
\end{example}

\begin{remark}\label{remark:Tforall}
	Recall that all structures in $\mathcal{K}$ are models of the universal fragment $T_\forall$ of the theory $T$ that we fixed at the beginning. Since  $\mathcal{K}$ is hereditary we will henceforth always assume without loss of generality that $T_\forall$ is exactly $\mrm{Th}_\forall(\mathcal{K})$, namely the universal theory of the class $(\mathcal{K},\oleq)$.
\end{remark}

We extend the relation $\oleq$ from finite structures in $\mathcal{K}$ to arbitrary models of $T_\forall$ as follows. We notice that the following definition of $A \oleq B$ for infinite structures is equivalent to the extension of the definition of $\oleq$ from $\mathcal{K}$ to $T_\forall$ from \cite[Def.~2.17]{baldwin_generic}.

\begin{definition}\label{infinite_strong extensions}
	Let $A \subseteq B \models T_\forall$, then we let $A \oleq B$ if and only if $ A\cap B_0 \oleq B_0 $ for all finite $ B_0 \subseteq B$. 
\end{definition}

\noindent Notice that all definitions from this section and the previous one extend naturally also to the setting with arbitrary structures $A \subseteq B \models T_\forall$. Also, if $B$ is finite, then Condition (\ref{condition:E}) makes sure that  $A \oleq B$ holds if and only if $A \oleq B_0$ for all $B_0\subseteq B$, thus  the previous definition agrees with the original relation $\oleq$ over the finite models of $T_\forall$ (cf.~also \cite[Def.~2.17]{baldwin_generic}). It is then straightforward to verify that arbitrary models of $T_\forall$ satisfy the conditions from \ref{context}. Also, the free amalgam ${B\otimes_A C}$ is well-defined also for infinite models of $T_\forall$, and it is straightforward to verify that under the assumptions from \ref{context} these also satisfy the algebraic amalgamation property. Finally, we remark that, exactly as observed in \cite[Lem.~2.8]{baldwin_generic} the property of being a strong extension is elementary.

\begin{remark}\label{the_type_remark} If $A$ is finite and $A \subseteq B \models T_\forall$, then it follows as in \cite[Lem~2.8]{baldwin_generic} that $A \oleq B$ is equivalent to saying that $B \models p(\bar{a})$, where $p(\bar{x})$ is the following type:
	\begin{align*}
	p(\bar{x})\coloneqq\{ \; \forall\bar{y}(\neg\bigwedge \delta_{(A, C)}(\bar{x},\bar{y})) : \; A \noleq C, \; C\in\mathcal{K}  \},
	\end{align*}
	where $\delta_{(A, C)}(\bar{x}, \bar{y}) = \bigwedge \mathrm{tp}^{\mrm{qf}}_C(\bar{a}, \bar{c})$ for $\bar{a}, \bar{c}$ enumerations of $A$ and $C\setminus A$, respectively. Notice that such formulas always exist because $L$ is finite and, moreover, if $A,B$ are of bounded size then the condition $A\oleq B$ can be already defined by a finite number of formulas, by modifying the definition of  $p(\bar{x})$ so to consider only those extensions $A\noleq C$ with $C$ of size $\leq |B|$. By Definition \ref{infinite_strong extensions} the same can be done also for infinite strong extensions $A\oleq B$. Let $\bar{a}=(a_i)_{i<\kappa}$ be an enumeration of $A$, then $A \oleq B$ is equivalent to $B \models p(\bar{a})$, where $p(\bar{x})$ is the following type:
	\begin{align*}
	p(\bar{x})\coloneqq\{ \; \forall\bar{y}(\neg\bigwedge \delta_{(A_0, C_0)}(\bar{x},\bar{y})) : \; A_0\subseteq_\omega A, \; A_0 \noleq C_0, \; C_0\in\mathcal{K}  \},
	\end{align*}
	where $\bar{x}=(x_i)_{i<\kappa}$, $\delta_{(A_0, C_0)}(\sigma(\bar{x}), \bar{y}) = \bigwedge \mathrm{tp}^{\mrm{qf}}_{C_0}(\sigma(\bar{a}), \bar{c})$, $\sigma(\bar{a})$ is the subenumeration of $A_0$ induced by $\bar{a}$, $\sigma(\bar{x})$ the corresponding subenumeration of $\bar{x}=(x_i)_{i<\kappa}$, and $ \bar{c}$ is an enumeration of $C_0\setminus A_0$.
\end{remark}

From the previous list of assumptions \ref{context}, it already follows that $(\mathcal{K},\oleq)$ admits a generic model $M_\star$. To this end, we first introduce the following important definition of intrinsic closure (cf.~\cite[Not.~2.22]{baldwin_generic}). Essentially, this is by construction the smallest extension of a subset $A$ of $B$ to a set $A'\oleq B$.

\begin{definition}\label{baldwin:intrinsic_closure}
	Let $A\subseteq B\models T_\forall$, the \emph{intrinsic closure} $\mrm{icl}_B(A)$ is the set $\mrm{icl}_B(A)\coloneqq\bigcup_{n\in\omega}\mrm{icl}_B^n(A)$ defined by letting $\mrm{icl}_B^0(A)=A$ and where $\mrm{icl}_B^{n+1}(A)$ is obtained from $\mrm{icl}_B^{n}(A)$ by adjoining all confined pairs $(C,D)$ such that $C\subseteq_\omega \mrm{icl}_B^{n}(A)$ and $D\subseteq_\omega  B$.
\end{definition}

We next recall some additional definitions from \cite{baldwin_generic}. We stress that, while the concept of $(\mathcal{K}, \oleq)$-saturation is standard in the literature, the notion of ${(\mathcal{K}, \oleq,\kappa)}$-saturation is quite specific of our setting. Implicitly, it also occurs in the context of Hrushovski constructions with irrational $\alpha$, as in \cite{shelah} and \cite{laskowski} (cf.~Remark \ref{remark:completeness_strategy}).

\begin{definition}\label{rich_def} A structure $M$ of $T_\forall$ is \emph{$(\mathcal{K}, \oleq, \kappa)$-saturated} if, for every $A \oleq B \models T_\forall$ with $|B|<\kappa$, whenever there is $f: A \rightarrow M$ such that $f(A) \oleq M$ then there is $\hat{f}: B \rightarrow M$ extending $f$ such that $\hat{f}(B) \oleq M$. A structure $M$  of $T_\forall$ is \emph{$(\mathcal{K}, \oleq, \kappa)$-homogeneous} if, for every isomorphism $f:A\cong B$ with $|A|=|B|<\kappa$ and $A,B\oleq M$, there is an automorphism $\hat{f}$ of $M$ extending $f$. We say that $M$ is \emph{$(\mathcal{K}, \oleq)$-saturated} if it is $(\mathcal{K}, \oleq,\aleph_0)$-saturated and that it is \emph{$(\mathcal{K}, \oleq)$-homogeneous} if it is $(\mathcal{K}, \oleq,\aleph_0)$-homogeneous. We say that $M\models T_\forall$ has \emph{finite closures} if for every finite $A\subseteq M$ there is a finite $B\oleq M$ with $A\subseteq B$, i.e., if $|\mrm{icl}_M(A)|<\aleph_0$. We say that a model $M\models T_\forall$ is \emph{$(\mathcal{K},\oleq)$-generic} if it is  $(\mathcal{K}, \oleq)$-saturated, $(\mathcal{K}, \oleq)$-homogeneous, and has finite closures.
\end{definition}

\noindent Whenever $(\mathcal{K},\oleq)$ is a Hrushovski class satisfying both the amalgamation property and the Conditions \ref{context}(\ref{condition:A})-(\ref{condition:F}), it follows from standard results from the literature that it has a $(\mathcal{K},\oleq)$-generic model (see  e.g. \cite[Thm.~2.12]{baldwin_generic}).

\begin{fact}\label{fact:existence_generic}
	Let $(\mathcal{K},\oleq)$ be a class satisfying the assumptions from Context~\ref{context} above, then $(\mathcal{K},\oleq)$ has a $(\mathcal{K},\oleq)$-generic model $M_\star$.
\end{fact}

\subsection{The theory $T^+$}\label{subsec:3}

We come to the main object of investigation of the present article, i.e., the theory $T^+\coloneqq T^+_\mathcal{K}$ associated to $(\mathcal{K},\oleq)$. Given a class $(\mathcal{K},\oleq)$ satisfying the conditions from \ref{context}, we define the theory $T^+$ as follows. Recall from \ref{remark:Tforall} that $T_\forall=\mrm{Th}_\forall(\mathcal{K})$.

\begin{definition}\label{the_theory} Let $T^+\coloneqq T^+_\mathcal{K}$ be such that $M \models T^+$ if the following happens.
	\begin{enumerate}[(1)]
		\item\label{the_theory:1} $M \models T_\forall$.
		\item\label{first_alg_ax}\label{the_theory:2} Let $A \subseteq_\omega M$ and let $A \oleq B\in \mathcal{K}$ be a minimal strong extension which is $\mathcal{K}$-algebraic. If $A\leq_{|B\setminus A|}M$ then $M$ contains a copy of $B$ over $A$.
		\item\label{the_theory:3} Let $A \subseteq_\omega M$ and let $A \oleq B\in \mathcal{K}$ be a minimal strong extension which is not $\mathcal{K}$-algebraic. If $A\leq_{|B\setminus A|}M$ then $M$ contains infinitely many disjoint copies of $B$ over $A$.
	\end{enumerate}
\end{definition}

\begin{remark}
	We notice that $T^+$ is indeed first-order. Since $M\models T^+$  entails $M \models T_\forall$, it follows that the definition of $\oleq$ for arbitrary structures from \ref{infinite_strong extensions} applies also to models of $T^+$ and their subsets. Also, it follows from Remark~\ref{the_type_remark} and the fact that $L$ is finite that, if $A, B\subseteq M$ are finite, then it is possible to express that $A\oleq B$ by a single formula. Therefore clause (1) is expressible by a set of first-order sentences in $L$. To see that also clauses (2) and (3) are elementary it suffices to notice that the relation $A\leq_{|B\setminus A|}M$ is also elementary.
\end{remark}

\begin{example}
	Let $(\mathcal{K},\oleq)$ be the class of finite open projective planes with the relation $\oleq$ defined  in \ref{example:projective_planes}. Then we have that $M\models T^+$ if the following hold:
	\begin{enumerate}[(1)]
		\item $M$ is open, i.e., every finite subset of $M$ contains either a point incident with at most two lines, or a line incident with at most two points;
		\item every two points in $M$ are incident to a unique common line, and every two lines in $M$ are incident to a unique common point;
		\item every point is incident to infinitely many lines, and every line is incident to infinitely many points.
	\end{enumerate}
	Strictly speaking, the three clauses above may appear different from those in \ref{the_theory}. However, clause \ref{the_theory}(\ref{the_theory:1}) easily follows from the axiomatisation above, while the requirement that $A\leq_{|B\setminus A|}M$  from \ref{the_theory}(\ref{the_theory:2})-(\ref{the_theory:3}) is always verified in the context of partial projective planes, thus making the above axiomatisation simpler. As a matter of fact, in the examples that we consider in this paper the constraint that $A\leq_{|B\setminus A|}M$ always holds automatically, with the exception of generalised $n$-gons (for which we expand on this issue in Section \ref{section:ngons}). 
\end{example}

Under the finitary conditions from \ref{context} we can prove that the theory $T^+$ from \ref{the_theory} is the most natural first-order theory associated to the class ${(\mathcal{K},\oleq)}$. In particular, $T^+$ is always satisfied by the generic $M_\star$, and is thus always consistent.

\begin{proposition}\label{proposition:generic_models_T}
	Let $(\mathcal{K},\oleq)$ be a class satisfying the assumptions from Context~\ref{context}, then its $(\mathcal{K},\oleq)$-generic $M_\star$ is a model of $T^+$.
\end{proposition}
\begin{proof}
	The fact that $M_\star$ is open follows immediately from the definition of the generic from \ref{rich_def}. To show that $M_\star\models T^+$ it suffices to show that $M_\star$ satisfies \ref{the_theory}(2) and \ref{the_theory}(3). 
	
	\smallskip 
	\noindent Consider \ref{the_theory}(2). Let $A\subseteq_\omega M_\star$ and consider a minimal strong extension $A \oleq B\in\mathcal{K}$ which is $\mathcal{K}$-algebraic. Suppose that $A\leq_n M$ for $n=|B\setminus A|$. Since $M_\star$ has finite closures, there is a finite $C\oleq M_\star$ such that $A\subseteq C$. Moreover, since $A\leq_n M$, it also follows that $A\leq_n C$. By the  algebraic amalgamation property it follows that either there is an embedding $f:B\to C$ or $B\otimes_A C\in \mathcal{K}$. If there is an embedding $f:B\to C$ it follows that $C\subseteq M$ contains a copy of $B$ over $A$, which proves our claim. Otherwise, suppose that  $B\otimes_A C\in \mathcal{K}$. Then since by Remark~\ref{remark_pushout} we know that $C\oleq B\otimes_A C$, it follows from the fact that $M_\star$ is $(\mathcal{K}, \oleq)$-saturated that there is a strong embedding $g: B\otimes_A C \to M$ over $C$. Again, this shows that $M$ contains a copy of $B$ over $A$, which verifies \ref{the_theory}(2).
	
	\smallskip 
	\noindent Consider \ref{the_theory}(3). Let $A\subseteq_\omega M_\star$ and consider a minimal strong extension $A \oleq B\in \mathcal{K}$  which is not $\mathcal{K}$-algebraic. Suppose that $A\leq_n M$ for $n=|B\setminus A|$, then since $M_\star$ has finite closures, there is a finite $C\oleq M_\star$ such that $A\subseteq C$. For every $k<\omega$ we let $B_k$ be the free amalgam of $k$ many copies of $B$ over $A$, which by the algebraic amalgamation property belongs to $\mathcal{K}$. Then, it follows again from the algebraic amalgamation property that $B_k\otimes_A C\in \mathcal{K}$ and, reasoning exactly as in the previous case, we obtain that $B_k\otimes_A C$ strongly embeds in $M$ over $C$. Since this holds for all $k<\omega$ it follows that $M$ contains infinitely many copies of $B$ over $A$. This verifies \ref{the_theory}(3) and shows that $M_\star\models T^+$.
\end{proof}

\begin{remark}\label{remark:finite_closure_first_order}
	We stress that Proposition~\ref{proposition:generic_models_T} only shows that $M_\star\models T^+$, and \emph{not} that $T^+=\mrm{Th}(M_\star)$. Suppose however (following \cite[Def.~2.10]{baldwin_generic}) that \emph{all models of $T_\forall$} have the finite closure property, i.e., if $A\subseteq_\omega N\models T_\forall$ then there is a finite $B\supseteq A$ such that $B\oleq N$. If this is the case then by Remark \ref{the_type_remark} the fact that $A\oleq N$ becomes first-order expressible across models of $T_\forall$. Let then $T_\star$ be the theory that says that, if $A\oleq N$ and $A\oleq C\in \mathcal{K}$ then $N$ contains an isomorphic copy of $C$ over $A$. By a back and forth argument it is possible to show that $T_\star$ is the complete first-order theory of $M_\star$. However, this argument fails when the finite closure property does not hold. The key goal of the next section is to show that, under Assumption \ref{main_th}(\hyperref[completeness_axiom]{C1}), $T^+$ is the complete first-order theory of $M_\star$.
\end{remark}

\subsection{Completeness}\label{subsec:4}
The main goal of this section is to show that, under the conditions from \ref{context} and the following assumptions in \ref{main_th}, the theory $T^+$ is complete. In the cases of the incidence geometries considered in \cite{funk2} (and that we address later in Sections \ref{sec_steiner}-\ref{last_section}), it will be straightforward to verify that non-degenerate free completions of open configurations are models of $T^+$. In particular, this entails that, whenever $A$ and $B$ are two structures in $\mathcal{K}$ with infinite free completions $F(A)$ and $F(B)$ (defined in \ref{general_free_amalgam}), then  $F(A)$ and $F(B)$  are elementarily equivalent.

We proceed as follows. We next isolate the abstract Assumptions \ref{main_th}(\hyperref[completeness_axiom]{C1})-(\hyperref[the_K_homogeneous_lemma]{C2}) from which, together with the assumptions from Context~\ref{context}, it is possible to prove that $T^+$ is complete. In the following Section~\ref{subsec:2} we specify the additional Assumption \ref{ass:hf_closure}(\hyperref[the_hf_axiom]{C3}) connected to the existence of $\mrm{HF}$-closures, and some more technical assumptions (the conditions \ref{technical_assumptions}(\hyperref[trivial_condition]{D1})-(\hyperref[extension]{D3})) which always guarantee that Assumptions \ref{main_th}(\hyperref[completeness_axiom]{C1})-(\hyperref[the_K_homogeneous_lemma]{C2}) hold. In our concrete applications from Sections \ref{sec_steiner}-\ref{last_section} we will always verify Assumption \ref{ass:hf_closure}(\hyperref[the_hf_axiom]{C3}) and \ref{technical_assumptions}(\hyperref[trivial_condition]{D1})-(\hyperref[extension]{D3}). However, we decided to isolate \ref{main_th}(\hyperref[completeness_axiom]{C1})-(\hyperref[the_K_homogeneous_lemma]{C2})  to stress that \ref{main_th}(\hyperref[completeness_axiom]{C1}) makes for the key ingredient in the completeness proof, while the conditions \ref{technical_assumptions}(\hyperref[trivial_condition]{D1})-(\hyperref[extension]{D3}) are rather the technical tools that we use to establish it. We also emphasise that the latter Assumption \ref{main_th}(\hyperref[the_K_homogeneous_lemma]{C2}) is not necessary to show completeness, but it plays an important role in the proof of stability and in the study of forking independence from Section~\ref{subsec:5}. We refer the reader to Definition~\ref{rich_def} for the notions of $(\mathcal{K}, \oleq,\kappa)$-saturation and $(\mathcal{K}, \oleq,\kappa)$-homogeneity.

\begin{assumption}\label{main_th} Let $(\mathcal{K}, \oleq)$ and $T^+$ be respectively as in \ref{context} and \ref{the_theory}. We assume that $T^+$ satisfies the following conditions:
	\begin{enumerate}[(C1)]
		\item\label{completeness_axiom} for all $\kappa\geq \aleph_1$ every $\kappa$-saturated model $M\models T^+$ is $(\mathcal{K}, \oleq,\kappa)$-saturated;
		\item\label{the_K_homogeneous_lemma}  for all $\kappa\geq \aleph_1$ every $\kappa$-saturated model $M\models T^+$ is $(\mathcal{K}, \oleq,\kappa)$-homogeneous.	
	\end{enumerate}
\end{assumption}

\begin{remark}\label{remark:completeness_strategy}
	We explain the reason behind the requirement that $\kappa$-saturated model $M\models T^+$ are $(\mathcal{K}, \oleq,\kappa)$-saturated for all $\kappa\geq \aleph_1$. As we recalled in \ref{remark:finite_closure_first_order}, one says that a model $M\models T^+$ has the \emph{finite closure property} if every finite subset $A\subseteq M$ is included in a finite set $A'\supseteq A$ such that $A'\oleq M$. In the context of Hrushovski constructions this is  usually the key element to prove completeness: if one can show that the $\aleph_0$-saturated models of a certain theory $T$ are also $(\mathcal{K}, \oleq)$-saturated, then it follows by back and forth between strong finite subsets of two $\aleph_0$-saturated models that $T$ is exactly the theory of the generic $M_\star$ of $(\mathcal{K}, \oleq)$ (cf. also Remark~\ref{remark:finite_closure_first_order}). However, this strategy crucially requires that we work in a setting where \emph{every} model satisfies the finite closure property, and not only the generic $M_\star$. Crucially, this is \emph{not} the case for the geometries from \cite{funk2}. In fact, the examples that we provide later to show that open Steiner systems, $n$-gons, $k$-nets, affine planes and Benz planes are not superstable also witness that there are models $M \models T^+$ and finite configurations $C\subseteq M$ such that for every finite $C \subseteq D \subseteq_\omega M$ we have $D\noleq M$. Still, under the presence of Condition \ref{context}(\ref{condition:H}) we have that every subset $A$ of a model $M\models T^+$ is contained in a set $D\oleq M$ of size $\leq |A|+\aleph_0$.  Our completeness proof will thus reproduce the back and forth argument that one can find in many Hrushovski constructions, but it crucially requires models of $T^+$ that are $(\mathcal{K}, \oleq,\aleph_1)$-saturated, and not simply $(\mathcal{K}, \oleq,\aleph_0)$-saturated. Interestingly, our case resembles the one from \cite{shelah}, \cite{laskowski} and \cite[Sec.~6]{baldwin_generic}, i.e., the non-superstable case of Hrushovski predimension constructions with ``irrational $\alpha$''.  In particular, although Laskowski's completeness proof \cite[Cor.~5.7]{laskowski} is syntactical, he also mentions that ``more model-theoretic proofs of these results are possible by passing to sufficiently saturated elementary extensions'' \cite[p.~168]{laskowski}. Essentially all that our proofs are doing is to follow this alternative model-theoretic route.
\end{remark}

Recall from \ref{notational_conventions} that we denote by $\mrm{acl}_B^{\mrm{qf}}(A)$ the algebraic closure of $A$ in some model $B$ obtained by considering only quantifier-free formulas. We prove the following lemma, which follows essentially from Condition \ref{context}(\ref{condition:H}). From this, together with the former  assumptions, we then prove the completeness of $T^+$ in \ref{prop_completeness}.

\begin{lemma}\label{acl_open_rk}
	Let $B\models T_\forall$ and  $A\subseteq B$ such that $A = \mrm{acl}^{\mrm{qf}}_B(A)$ then $A\oleq B$.
\end{lemma}
\begin{proof}
	Suppose there is a  finite  $B_0\subseteq B$ such that $A\cap B_0\not \oleq B_0$ but $B_0\setminus A\nsubseteq \mrm{acl}^{\mrm{qf}}_B(A)$. Without loss of generality we can also assume that $A\noleq B_0$ is a confined extension. Then, since $B\models T_\forall$, it follows by compactness that for every $n<\omega$ we can find some $C_n \in \mathcal{K}$ such that $C_n$ contains at least $n$ many disjoint copies of $B_0$ over $A\cap B_0$. This contradicts   \ref{context}(\ref{condition:H}).
\end{proof}

\begin{theorem}\label{prop_completeness} \label{el_sub_lemma}  Suppose  $T^+$ satisfies \ref{main_th}(\hyperref[completeness_axiom]{C1}). Then the theory $T^+$ is complete and decidable. Moreover, if $A, B \models T^+$ then the following are equivalent:
	\begin{enumerate} [(1)]
		\item $A = \mrm{acl}_B(A)$;
		\item $A \oleq B$;
		\item $A \preccurlyeq B$.
	\end{enumerate}	
\end{theorem}
\begin{proof}
	We first prove that $T^+$ is complete. Let $M, N \models T^+$, then by \ref{main_th}(\hyperref[completeness_axiom]{C1}) we can also assume without loss of generality that $M$ and $N$ are $(\mathcal{K}, \oleq,\aleph_1)$-saturated. We claim that the following set $\mathcal{I}$ is a back and forth system between $M$ and $N$:
	$$\mathcal{I} = \{f: A \cong B : A \oleq M, B \oleq N \text{ and } |A|, |B| \leq \aleph_0\}.$$
	Clearly this suffices, as then by \cite[Ex.~1.3.5]{tent_book} we have that $M$ and $N$ are elementarily equivalent. Let $f: A \cong B$, where $A\oleq M$, $B\oleq N$, $|A|=|B|\leq \aleph_0$. Let $a\in M\setminus A$ and notice that by Lemma \ref{acl_open_rk} we have that $ \mrm{acl}_M(Aa)\oleq M$. Moreover, we have that $A\oleq \mrm{acl}_M(Aa)$ and that $ \mrm{acl}_M(Aa)$ is countable. Since $N$ is $(\mathcal{K}, \oleq,\aleph_1)$-saturated it follows that $f:A\to B$ lifts to a strong embedding $\hat{f}: \mrm{acl}_M(Aa)\to N $, showing that $\hat{f}\in \mathcal{I}$. The decidability of $T^+$ then follows immediately from completeness and the fact that the axiomatisation from \ref{the_theory} is recursive.
	
	\medskip
	\noindent We next show that (1)-(3) are equivalent. This proof generalises \cite[Thm. 2.28]{tent} to our abstract setting. Direction  $ (1) \Rightarrow (2)$ follows from Lemma \ref{acl_open_rk}, and direction $(3)\Rightarrow (1)$ is obvious.  We consider the remaining direction  $(2)\Rightarrow (3)$. In addition to the assumption that $A \oleq B$, by \ref{main_th}(\hyperref[completeness_axiom]{C1}) (and Remark \ref{the_type_remark}) we can also assume without loss of generality that $A$ and $B$ are $(\mathcal{K}, \oleq,\aleph_1)$-saturated. Now, notice that if $M\oleq A$ then by transitivity $M\oleq B$ and, if $N\oleq B$ and $N\subseteq A$ then it follows from \ref{context}(\ref{condition:E})  that $N\oleq A$. Using this one shows that $A \preccurlyeq B$, as in fact for every finite set $D \subseteq A$ we can consider the following set of isomorphisms $\mathcal{I}$: 	
	\[\{f: M \cong N : D\subseteq M \oleq A, \; D\subseteq N \oleq B, \; |M|, |N| \leq \aleph_0, \; f\restriction D=\mathrm{id}_D\}.\]	
	In particular, by reasoning as in the previous completeness proof, one can show that $\mathcal{I}$ is a back and forth system, thus completing the proof.
\end{proof}

\subsection{$\mrm{HF}$-orderings}\label{subsec:2}
We introduce in this section the technique of $\mrm{HF}$-orderings and the related Assumption \ref{ass:hf_closure}(\hyperref[the_hf_axiom]{C3}), which expresses the existence of a $\mrm{HF}$-closure (cf.~Definition~\ref{def_closure_operator}). Moreover, we also introduce the technical assumptions \ref{technical_assumptions}(\hyperref[trivial_condition]{D1})-(\hyperref[extension]{D3}), which we shall use in concrete cases to verify that a theory satisfies \ref{main_th}(\hyperref[completeness_axiom]{C1})-(\hyperref[the_K_homogeneous_lemma]{C2}). 

Essentially, $\mrm{HF}$-orderings are an important tool from the literature on incidence geometries which allows one to study any extension $A\oleq B$ by decomposing it in smaller ``atomic'' extensions. More precisely, the notion of $\mrm{HF}$-order was introduced by Siebenmann in \cite{sieben} in the context of projective planes and later developed further in various references such as \cite{dem, ditor}. By way of example, one can associate a linear order to any finite open projective plane, by requiring that each element has at most two incidences with its set of predecessors. In this context, however, $\mrm{HF}$-orders were always assumed to be  well-founded, thus limiting their scope of application in the context of first-order logic. In fact, as observed in \cite[Rem.~3.7]{paolini&hyttinen}, if $A=(A,<)$ is an infinite structure with a $\mrm{HF}$-order (or, for that matter, with any linear order), then its ultrapower $A^{\mathfrak{U}}$ will have infinite descending chains. Therefore, in order to apply the technique of  $\mrm{HF}$-orders in the setting of model theory, one also has to make sense of non-well-founded ones. This approach was taken by Hyttinen and Paolini in \cite[Def.~3.4]{paolini&hyttinen}. Here we introduce a further generalisation of their definition, as we define $\mrm{HF}$-orderings abstractly for any class $\mathcal{K}$ satisfying the conditions from \ref{context}. Most importantly, we stress the crucial role of conditions \ref{context}(\ref{condition:I})--(\ref{condition:J}), which are essential to make the notion of $\mrm{HF}$-order first-order expressible (cf.~Remark \ref{hf_first_order} below). We start by fixing some notation and then proceed to define $\mrm{HF}$-orders.

\begin{definition}\label{linear_ordering_convention}
	Let $(X,<)$ be a poset and let $Y\subseteq X$. We write $Y^\downarrow$ for the downset generated by $Y$, i.e., $Y^\downarrow = \{b\in X: \exists a\in Y \text{ s.t. } b\leq a    \}$. We say that $Y\subseteq X$ is \emph{convex} if $a,c\in Y$ and $a<b<c$ entail $b\in Y$. Moreover if $(X,<)$ is a linear order and $Y,Z\subseteq X$, we write $Y<Z$ if  $y<z$ for every $y\in Y$ and $z\in Z$.
\end{definition}

\begin{definition}[$\mrm{HF}$-order]\label{def_HF_order} Let $A \subseteq B \models T_\forall$.  a $\mrm{HF}$-order of $B$ over $A$ is a pair $(<, P_<)$, where:
	\begin{enumerate}[(H1)]
		\item\label{HF1} $<$ is a linear order of $B\setminus A$, and $P_<$ is a partition of $B\setminus A$ into convex sets of size $\leq \mrm{max}(\mathbf{X}_\mathcal{K})$ (where $\mathbf{X}_\mathcal{K}$ is given by Condition \ref{context}(\ref{condition:I}));
		\item\label{HF2} for every $b \in B\setminus A$  and for every finite $C_0 \subseteq b^\downarrow\setminus P_<(b)$ we have that
		\[AC_0 \oleq AC_0P_<(b),\]
		where $P_<(b)$ is the unique set in $P_<$ such that $b\in P_<(b)$.
	\end{enumerate}
	For ease of notation, we generally simply say that $<$ is a $\mrm{HF}$-order of $B$ over $A$, and we write $P$ instead of $P_<$ when the ordering $<$ is clear from the context. If there is such an order we also write $A \hleq B$. When $A = \emptyset$ we say that $<$ is a $\mrm{HF}$-order of $B$ instead of saying that $<$ is a $\mrm{HF}$-order of $B$ over~$\emptyset$.  
\end{definition}

\begin{notation}\label{hat_notation}
	We notice that, if $<$ is a $\mrm{HF}$-order of $B$ over $A$, then every piece $X\in P$ is essentially a finite ordered subchain of $B\setminus A$. For this reason, we usually think of these pieces as finite tuples with the order induced by $<$. We write $\widehat{c}$ for the tuple $(c_1,\dots, c_n)$  such that $P(c)=\{  c_1,\dots, c_n\}$ and $c_1<\dots<c_n$. Differently, for a set $X\subseteq B\setminus A$ we simply let $\widehat{X}=\bigcup_{x\in X}P(x)$. Thus for an element $c\in B\setminus A$ we have that $\widehat{c}$ is a tuple, while for a set $C\subseteq  B\setminus A$ we have that $\widehat{C}$ is a set.
\end{notation}

\begin{example}
	We spell out what the previous definition means in the case of projective planes  (cf.~\ref{example:projective_planes}). Since in this case $\mathbf{X}_\mathcal{K}=\{1\}$, Condition \ref{def_HF_order}(\hyperref[HF1]{H1}) simply says that $<$ is a linear ordering over $B\setminus A$, while Condition \ref{def_HF_order}(\hyperref[HF2]{H2}) says that each element  $b\in B\setminus A$ is incident with at most two elements $c_0,c_1\in A\cup b^\downarrow$. This is exactly the notion of $\mrm{HF}$-order for projective planes defined in \cite[Def.~3.4]{paolini&hyttinen}.
\end{example}

\begin{remark}\label{hf_first_order}	
	We observe that  the property of  being a $\mrm{HF}$-order is first-order expressible, by a (possibly infinite) set of formulas, in the language $L\cup \{ <, P\}$  where $<$ is a binary relation and $P$ is a relation of arity $\mrm{max}(\mathbf{X}_\mathcal{K})$ (notice that the existence of $\mrm{max}(\mathbf{X}_\mathcal{K})$ is always guaranteed by Condition \ref{context}(\ref{condition:I})). The fact that $<$ is a linear order is clearly expressible by a sentence and, similarly, we can express that $P$ is a partition of $B\setminus A$ in convex sets of size $\leq \mrm{max}(\mathbf{X}_\mathcal{K})$ (notice that we also need to use some technical device such as repetition of the last occurring element to encode pieces of the partition which have length shorter than $\mrm{max}(\mathbf{X}_\mathcal{K})$). It follows that Condition \ref{def_HF_order}(\hyperref[HF1]{H1}) can be expressed by a first-order formula. Similarly, since by Remark \ref{the_type_remark} we can capture the strong submodel relation $\oleq$ in first-order logic, we can also express Condition \ref{def_HF_order}(\hyperref[HF2]{H2}) by a set of formulas in  $L\cup \{ <, P\}$.
\end{remark}

The following Lemma is essentially a generalisation of \cite[Lem.~1]{sieben} to the abstract setting of a class $(\mathcal{K},\oleq)$ satisfying the conditions from Context~\ref{context}.

\begin{lemma}\label{constructing_finite_pieces} Let $A \oleq C\in \mathcal{K}$, then there is a $\mrm{HF}$-order of $C$ over $A$.
\end{lemma}
\begin{proof}  Since $A\oleq C$, by Condition \ref{context}(\ref{condition:I}) there is some tuple $\bar{c}_{1}=(c^1_{1},\dots,c^{k_1}_{1})\in (C\setminus A)^{k_1}$ with $k_1\in \mathbf{X}_\mathcal{K}$ and such that $C\setminus \{c^1_{1},\dots,c^{k_1}_{1}\}\oleq C$. By Condition \ref{context}(\ref{condition:E}) it follows that $A\oleq C\setminus\{ c^1_{1},\dots,c^{k_1}_{1}  \}$, thus there is some $\bar{c}_{2}=(c^1_{2},\dots,c^{k_2}_{2})\in (C\setminus Ac^1_{1},\dots,c^{k_1}_{1} )^{k_2}$ with $k_2\in \mathbf{X}_\mathcal{K}$ and such that 
	\[  C\setminus \{c^1_{1},\dots,c^{k_1}_{1}c^1_{2},\dots,c^{k_2}_{2}\}  \oleq C\setminus \{c^1_{1},\dots,c^{k_1}_{1}\}.\]
	Since $C$ is finite, after finitely many steps, we find some $n<\omega$ and tuples $\bar{c}_{1},\dots,\bar{c}_{n}$ such that $C=A\cup \{c^i_j : 1\leq i\leq k_{j}, \;    1\leq j\leq n \}$. We then let $P(c^i_j)=\{c^i_j : 1\leq i\leq k_{j}\}$ and we define
	\begin{align*}
		c^i_j < c^\ell_m \; \Longleftrightarrow \; (m<j) \text{ or } (j=m \text{ and } i<\ell).
	\end{align*}
	Then $(<, P)$ is a $\mrm{HF}$-order of $C$ over $A$.
\end{proof}

If $<$ is a $\mrm{HF}$-order of a structure $B$ and $C\subseteq B$, it is often useful to restrict the $\mrm{HF}$-order  of $B$ to  the subset $C$ (we do this for example in the following proof of \ref{existence_HF_orders}). To this end, we introduce the following definition.

\begin{definition}\label{restriction:hf}
	Let  $A\subseteq B$ and suppose that $<$ is a $\mrm{HF}$-order of $B$ over $A$. Let $C\subseteq B$, then the \emph{restriction}  of $<$ to $C$  is the pair $(<',P_{<'})$ defined by letting $<'\coloneq <\restriction C$  and $P_{<'}(c)=P_{<}(c)\cap C$ for all $c\in C$. With a slight abuse of notation we often write  simply $< \restriction C$ for the restriction of the $\mrm{HF}$-order $(<,P_<)$ to  $C$.
\end{definition}

\noindent We notice that, if $<$ is a $\mrm{HF}$-order $<$ of $B$ over $A$ and $C\subseteq B$, then the restriction $<\restriction C$ is always a $\mrm{HF}$-order of $C$ over $A\cap C$. This follows immediately from Definition \ref{def_HF_order} and Condition \ref{context}(\ref{condition:E}). The next proposition is an abstract version of \cite[Prop.~3.14]{paolini&hyttinen}. The proof is essentially the same.

\begin{proposition}\label{existence_HF_orders} Let $A \oleq B \models T_\forall$, then there is a $\mrm{HF}$-order $<$ of $B$ over $A$.
\end{proposition}
\begin{proof} Let $X$ be the set of all finite substructures of $B$, by \ref{constructing_finite_pieces} every $C\in X$ has a $\mrm{HF}$-order $<_C$ over $\emptyset$. Let  $\mathfrak{U}$ be an ultrafilter on $X$ such that for all $C \in X$ we have $X_C = \{ D \in X : C \subseteq D \} \in \mathfrak{U}$. For every $C\in X$ we let $<^1_C, ..., <^{n(C)}_C$ be an injective enumeration of all the $\mrm{HF}$-orderings of $C$ over $\emptyset$. Then, since $\mathfrak{U}$ is an ultrafilter, it follows that for all $C\in X$ there is a unique $\mrm{HF}$-ordering $(<^*_C, P_C^*)$ such that:
	\[Y_C \coloneqq \{ D \in X : C \subseteq D, \; <_D \restriction C = <^*_C \text{ and } P_{<_D}\restriction C= P_C^* \} \in \mathfrak{U},\]
	where $P_{<_D}\restriction C= \{X\cap C : X\in P_{<_D} \}$. Then for all $c\in B$ there is a structure $C\in X$ such that $P_{<^*_C}(c)=P_{<_D}(c)$ for all $C\subseteq D\in X$. We let $<_*=\bigcup_{C\in X}<^*_C$ and we define a partition $P_*$ by letting, for every $c\in B$, $P_*(c)=  P_{<^*_C}(c)$,	where $C$ is chosen to satisfy $P_{<^*_C}(c)=P_{<_D}(c)$ for all $C\subseteq D\in X$. It is then easy to verify that $(<_*, P_*)$ is a $\mrm{HF}$-order of  $B$.
\end{proof}

\noindent  From Proposition~\ref{existence_HF_orders} we obtain the following corollary, which allows us to translate the language of open configuration into the language of $\mrm{HF}$-orders, and \emph{vice versa}.

\begin{corollary}\label{equivalence_HFo.ordering} Let $A \subseteq B \models T_\forall$. The following are equivalent:
	\begin{enumerate}[(1)]
		\item $A \oleq B$;
		\item $A \hleq B$.
	\end{enumerate}
\end{corollary}
\begin{proof}
	The direction from (1) to (2) is Proposition \ref{existence_HF_orders}. The direction from (2) to (1) follows from Definition \ref{infinite_strong extensions} and induction on the size of $C\subseteq_\omega B\setminus A$.
\end{proof}

If $<$ is a $\mrm{HF}$-order of $B$ and $C$ an arbitrary subset of $B$, it is \emph{not} necessarily the case that $<\restriction (B\setminus C)$ induces a $\mrm{HF}$-order of $B$ {\em over $C$}. In fact, as $C$ is an arbitrary subset of $B$, we could actually even have that $C\noleq B$. We generalise \cite[Def. 3.9]{paolini&hyttinen} and we define the notion of \emph{$\mrm{HF}$-closure of a set $C$ relative to a specific $\mrm{HF}$-order $<$}, which is an extension $\mrm{cl}_<(C)\supseteq C$ of size $\leq |C|+\aleph_0$ such that $\mrm{cl}_<(C)\oleq  B$  always holds.  Notice also that, in the following definition and afterwards in the paper, we abide to the model-theoretic convention of writing $AB$ instead of $A\cup B$, so in particular $A\mrm{cl}_<(C)$ refers to the union $A\cup\mrm{cl}_<(C)$ and should not be confused with the algebraic closure (whose notation we fixed in \ref{notational_conventions}).

\begin{definition}[$\mrm{HF}$-closure]\label{def_closure_operator} Let $A\subseteq B \models T_\forall$, suppose $<$ is a $\mrm{HF}$-order of $B$ over $A$, and recall the  notation from \ref{hat_notation}.  We say that an operator $\mrm{cl}_<\coloneqq \bigcup_{n<\omega }\mrm{cl}_<^{n}$ is a \emph{$\mrm{HF}$-closure operator with respect to $B$ and $<$} if, for all $C\subseteq B\setminus A$:
	\begin{enumerate}[(1)]
		\item $C\subseteq\mrm{cl}^n_<(C)\subseteq \mrm{cl}^{n+1}_<(C)$  for all $n<\omega$;
		\item 	$|\mrm{cl}^n_<(C)| < |C|+\aleph_0$ for all $n<\omega$;
	\end{enumerate}
	and, moreover, $\mrm{cl}_<$ satisfies the following key property for all $C,D\subseteq B\setminus A$:
	\begin{enumerate}[(3)]
		\item  $A\mrm{cl}_<(CD)=A\mrm{cl}_<(C)\otimes_{A\mrm{cl}_<(C)\cap A\mrm{cl}_<(D)} A\mrm{cl}_<(D)\oleq B$.
	\end{enumerate}
\end{definition}

In the definition above, the $\mrm{HF}$-closure operator $\mrm{cl}_<(C)$ is defined abstractly, although in essentially all the concrete examples from Section~\ref{sec:application} it will be defined by means of the Gaifman closure operator. The existence of a $\mrm{HF}$-closure in all models of $T_\forall$ is our next requirement in \ref{ass:hf_closure}(\hyperref[the_hf_axiom]{C3}). Later in Section~\ref{sec:application} we shall directly verify Assumption \ref{ass:hf_closure}(\hyperref[the_hf_axiom]{C3}) for all the examples of structures that we consider.

\begin{assumption}\label{ass:hf_closure}
	Let $(\mathcal{K}, \oleq)$ and $T^+$ be respectively as in \ref{context} and \ref{the_theory}. We assume that $T^+$ satisfies the following condition:
	\begin{enumerate}[(C3)]
	\item\label{the_hf_axiom} every model $A\subseteq B\models T_\forall$ with an associated $\mrm{HF}$-order $<$ of $B$ over $A$ has a $\mrm{HF}$-closure operator $\mrm{cl}_<\coloneqq\bigcup_{n<\omega}\mrm{cl}^n_<$ (cf.~\ref{def_closure_operator}).
\end{enumerate}
\end{assumption}

Additionally, \emph{and only in this section}, we assume the following technical assumptions  \ref{technical_assumptions}(\hyperref[trivial_condition]{D1})-(\hyperref[extension]{D3}). Essentially, we use these to prove that for all $\kappa\geq \aleph_1$ every $\kappa$-saturated models is both  $(\mathcal{K}, \oleq,\kappa)$-saturated and  $(\mathcal{K}, \oleq,\kappa)$-homogeneous, namely to verify \ref{main_th}(\hyperref[completeness_axiom]{C1})-(\hyperref[the_K_homogeneous_lemma]{C2}). This explains how we will proceed in Section \ref{sec:application} later, where we show that both \ref{ass:hf_closure}(\hyperref[the_hf_axiom]{C3}) and  \ref{technical_assumptions}(\hyperref[trivial_condition]{D1})-(\hyperref[extension]{D3}) are satisfied in all our intended applications. We also notice that the assumption that $M$ is $\aleph_1$-saturated in \ref{technical_assumptions}(\hyperref[minimality_condition]{D2}) is used only in the case of $n$-open graphs from Section \ref{sec:n-open-grpahs}, and it is not needed to verify \ref{technical_assumptions}(\hyperref[minimality_condition]{D2}) in the other concrete examples.

\begin{technical}\label{technical_assumptions} Suppose $T^+$ satisfies \ref{ass:hf_closure}(\hyperref[the_hf_axiom]{C3}) and so the $\mrm{HF}$-closure operator $\mrm{cl}_<\coloneqq\bigcup_{n<\omega}\mrm{cl}^n_<$ is well-defined.  In this section we assume that $T^+$ satisfies the following technical conditions:
	\begin{enumerate}[(D1)]
		\item\label{trivial_condition} if $M \models T^+$, $A \subseteq M$ is finite, $<$ is a $\mrm{HF}$-order of $M$ and $(A,Ab)$ is a trivial one-element extension, then for all $ n<\omega$ there is some $b'\in M$ such that:
		\begin{enumerate}
			\item $Ab \cong_A Ab'$;
			\item\label{trivial_condition_item_c} $Ab' \oleq A\mrm{cl}_<^n(b')$;
		\end{enumerate}
		\item\label{minimality_condition}\label{algebraic_condition} if $M \models T^+$ is $\aleph_1$-saturated, $A \oleq M$ is countable, and for every trivial one-element extension $A\oleq Ab$ there is $b' \in M$ such that $Ab \cong_A Ab'\oleq M$, then for every minimal extension $A\oleq A\bar{c}$ there is $\bar{d} \in M^{<\omega}$ such that $A\bar{c} \cong_A A\bar{d}\oleq M$;
		\item\label{extension} if $M\models T^+$ is $\aleph_1$-saturated, $A_0,B_0\oleq M$ and $f:A_0\cong B_0$ is an isomorphism, then there are $A_1,B_1\preccurlyeq M$ of size $\leq|A_0|+\aleph_0$ such that $A_0\oleq A_1$, $B_0\oleq B_1$ and there is an isomorphism $\hat{f}:A_1\cong B_1$ with $\hat{f}\restriction A_0=f$. 
	\end{enumerate}
\end{technical}

We thus conclude this section by showing that \emph{modulo} \ref{ass:hf_closure}(\hyperref[the_hf_axiom]{C3}), the technical assumptions \ref{technical_assumptions}(\hyperref[trivial_condition]{D1})-(\hyperref[extension]{D3}) entail both \ref{main_th}(\hyperref[completeness_axiom]{C1}) and \ref{main_th}(\hyperref[the_K_homogeneous_lemma]{C2}).

\begin{proposition}\label{the_K_saturated_lemma} 
	Suppose $T^+$ satisfies both \ref{ass:hf_closure}(\hyperref[the_hf_axiom]{C3}) and the technical assumptions \ref{technical_assumptions}(\hyperref[trivial_condition]{D1})-(\hyperref[minimality_condition]{D2}). Then  $T^+$ also satisfies \ref{main_th}(\hyperref[completeness_axiom]{C1}), i.e., for all $\kappa\geq \aleph_1$ every $\kappa$-saturated model $M\models T^+$ is $(\mathcal{K}, \oleq,\kappa)$-saturated.
\end{proposition}
\begin{proof}
	Suppose $M\models T^+$ is $\kappa$-saturated for some $\kappa\geq \aleph_1$, we claim that $M$ is also $(\mathcal{K}, \oleq,\kappa)$-saturated. We let $<$ be a $\mrm{HF}$-ordering of $M$ and $\mrm{cl}_<$ the associated $\mrm{HF}$-closure operator.
	
	\smallskip 
	\noindent Firstly, we prove the claim for extensions of the form $A\oleq B$ where $|A|<\kappa$ and $|B\setminus A|<\aleph_0$. By induction, it is clearly sufficient to consider only the minimal strong extensions $A\oleq A\bar{b}$. Moreover, by Condition (\hyperref[minimality_condition]{D2}), it suffices to consider only the trivial one-element extensions (cf.~\ref{def_extensions}(\ref{def_trivial})).
	
	\smallskip 
	\noindent  Let $A \oleq M$, $|A|<\kappa$ and suppose that $A\oleq Ab$ is a trivial extension. Let $A'\coloneqq \mrm{cl}_{<}(A)$, then by \ref{ass:hf_closure}(\hyperref[the_hf_axiom]{C3}) we have that $A'\oleq M$ and $|A'|<\kappa$. Consider the type $p(x/A')$ that says, for every $n < \omega$ and for every finite subset $D\subseteq A'$, that $Dx \oleq D\mrm{cl}^n_{<}(x)$ and that $Dx$ is a trivial extension of $D$. Notice that this is first-order expressible by reasoning as in Remark \ref{the_type_remark}. Then, by Condition (\hyperref[trivial_condition]{D1}), it follows that $p(x/A')$ is satisfiable. We let $b_* \in M$ be a realisation of this type. Clearly $Ab_*\cong_A Ab$. We show that also $Ab_* \oleq M$ holds.  First of all, notice that since $Ab_* \oleq M$, it follows by \ref{ass:hf_closure}(\hyperref[the_hf_axiom]{C3}) and \ref{def_closure_operator} that $\mrm{cl}_{<}(Ab_*) \oleq M$, and also $A'\mrm{cl}_{<}(b_*) \oleq M$.   Moreover, since $b_*$ satisfies $Db_* \oleq D\mrm{cl}^n_{<}(b_*)$ for every $D\subseteq_\omega A'$, it follows that $A'b_* \oleq A'\mrm{cl}_{<}(b_*)$.  By transitivity we conclude that  $A'b_* \oleq M$. Now, since $b_*$ is trivial over $A'$ it follows that $\mrm{gcl}_{A'b_*}(A'b_*\setminus Ab_*)=\mrm{gcl}_{A'}(A'\setminus A)$. Since we also have that $A\oleq A'$, it follows by Condition \ref{context}(\ref{condition:K}) that $Ab_*\oleq A'b_*$. By transitivity we conclude that  $Ab_*\oleq M$.
	
	\smallskip 
	\noindent We are now ready to prove the full statement of the Lemma. Consider an arbitrary extension $A\oleq B$ with $A\oleq M$ and both $A$ and $B\setminus A$ of size $<\kappa$. Fix an enumeration $(b_i)_{i<\kappa}$ of $B\setminus A$ and consider the infinitary type $p(\bar{x},A)$, where $\bar{x}=(x_i)_{i<\kappa}$, and such that:
	\begin{enumerate}[(i)]
		\item for all $\alpha<\omega$,  $p\restriction\{(x_i)_{i<\alpha} \}=\type^{\mrm{qf}}((b_i)_{i<\alpha}/A)$;
		\item for all finite $A_0\subseteq_\omega A$  and $I\subseteq_\omega \kappa$, $p(\bar{x},A)$ says that $A_0\cup \{x_i : i\in I\}\oleq M$.
	\end{enumerate}
	
	\noindent Now, condition (i) can be expressed simply by adding formulas to $p(\bar{x}/A)$ that say that the isomorphic type of $(x_i)_{i<\alpha}$ over $A$ is exactly  $\type^{\mrm{qf}}((b_i)_{i<\alpha}/A)$. Condition (ii) is more subtle, but one can see that it is indeed a first-order condition in the light of Remark \ref{the_type_remark}. In particular, Clause (ii) is equivalent to letting:
	\begin{align*}
		\forall \bar{y}\; ( \neg \delta_{(A_0b_{i_0}\dots b_{i_{n-1}}, C)}(\bar{a},x_{i_0}\dots x_{i_{n-1}},\bar{y} )     )\in p(\bar{x},A),
	\end{align*}	
	for all $i_0<\dots<i_{n-1}< \kappa$ , $A_0\subseteq_\omega A$, $A_0b_{i_0}\dots b_{i_{n-1}}\noleq C$ with $C\in \mathcal{K}$, and where $\bar{a}$ is an enumeration of $A_0$ and the  formula $\delta_{(A_0b_{i_0}\dots b_{i_{n-1}}, C)}(x_{i_0}\dots x_{i_{n-1}},\bar{y} )$ is defined as in Remark \ref{the_type_remark}.
	
	\smallskip
	\noindent Now, since $A\oleq B$, it follows in particular that, for all $I\subseteq_\omega$, $A\oleq A\cup \{b_i : i\in I\}$ (recall \ref{infinite_strong extensions}).  Since we have already proven the Lemma for the case of finite extensions it follows that there are $c_i\in M$ for all $i\in I$ such that $ \type^{\mrm{qf}}((b_i)_{i\in I}/A)=\type^{\mrm{qf}}((c_i)_{i\in I}/A) $ and $A\cup \{c_i : i\in I\}\oleq M$. Thus, the type $p(\bar{x},A)$ is consistent. Since $M$ is $\kappa$-saturated it then follows that $p(\bar{x},A)$ is realised in $M$. Therefore, we can find a sequence $(d_i)_{i<\kappa}$ in  $M$ such that $ \type^{\mrm{qf}}((d_i)_{i<\kappa}/A)=\type^{\mrm{qf}}((b_i)_{i<\kappa}/A) $ and, additionally, $A\cup \{d_i : i\in I\}\oleq M$ for all $I\subseteq_\omega \kappa$. By \ref{infinite_strong extensions} this entails that $A\cup \{d_i: i<\kappa \}\oleq M$, which completes our proof.
\end{proof}

\begin{proposition}\label{lemma:homogeneity}
	Suppose $T^+$ satisfies both \ref{ass:hf_closure}(\hyperref[the_hf_axiom]{C3}) and the technical assumption \ref{technical_assumptions}(\hyperref[extension]{D3}). Then  $T^+$ also satisfies \ref{main_th}(\hyperref[the_K_homogeneous_lemma]{C2}), i.e., for all $\kappa\geq \aleph_1$ every $\kappa$-saturated model is $(\mathcal{K}, \oleq,\kappa)$-homogeneous.
\end{proposition}
\begin{proof}
	Suppose $M$ is $\kappa$-saturated and let $A_0,B_0\oleq M$ be of size strictly less than $\kappa$ such that $f_0:A\cong B$ is an isomorphism. By Assumption \ref{technical_assumptions}(\hyperref[extension]{D3}) it follows that we can find two models $A_1\preccurlyeq M$ and $B_1\preccurlyeq M$ of size $<\kappa$ with $A_0\oleq A_1$, $B_0\oleq B_1$ and an isomorphism $f_1:A_1\to B_1$ satisfying $f_1\restriction A=f_0$. Since $M$ is $\kappa$-saturated, it follows that $f_1$ extends to an automorphism $f_2$ of $M$ such that $f_2\restriction A_1=f_1$, and thus in particular $f_2\restriction A_0=f_0$. This shows that $M$ is $(\mathcal{K}, \oleq,\kappa)$-homogeneous.
\end{proof}

\subsection{Stability and Forking independence}\label{subsec:5}

In this section we prove that the theory $T^+$ is always stable, we study  the relation of forking independence $\ind$ in models of $T^+$ and we show that $\ind$ is equivalent to a relation $\ind^{\otimes}$, defined in purely combinatorial terms. This also exhibits a connection between the theories studied in this paper and the framework of \emph{theories with free amalgamation} studied in \cite{Mutchnik,Conant1,Conant2}.  Moreover, we also relate $\ind$ and $\ind^{\otimes}$ to the \emph{free algebraic completions} in the theory $T$.  First, we provide a direct proof of the stability of $T^+$ from Assumptions \ref{main_th}(\hyperref[completeness_axiom]{C1})-(\hyperref[the_K_homogeneous_lemma]{C2}) and \ref{ass:hf_closure}(\hyperref[the_hf_axiom]{C3}). The following proof generalises the proof of stability for open projective planes from \cite{paolini&hyttinen} to the present abstract setting.

\begin{notation}\label{monster:notation}
	We denote by $\mathfrak{M}$ the monster model of the theory $T^+$ in the language expanded with a $\mrm{HF}$-order. Thus in particular $\mathfrak{M}$ satisfies the saturation property from \ref{main_th}(\hyperref[completeness_axiom]{C1})-(\hyperref[the_K_homogeneous_lemma]{C2}). When dealing with the monster model $\mathfrak{M}$ we usually omit the indices and write  $\mathrm{tp}(a/B)$ instead of $\mathrm{tp}_\mathfrak{M}(a/B)$ to refer to the type of an element, and we write  $ \mrm{acl}(B)$ instead of $ \mrm{acl}_\mathfrak{M}(B) $ to refer to the algebraic closure of a set. We do the same for quantifier-free types and quantifier-free algebraic closures, and similarly we write just $ \mrm{icl}_\mathfrak{M}(B) $ for the intrinsic closure of $B$ computed in $\mathfrak{M}$.
\end{notation}

\begin{theorem}\label{stability} 
	Suppose $T^+$ satisfies \ref{main_th}(\hyperref[completeness_axiom]{C1})-(\hyperref[the_K_homogeneous_lemma]{C2}) and \ref{ass:hf_closure}(\hyperref[the_hf_axiom]{C3}). Then $T^+$ is stable.
\end{theorem}
\begin{proof} 
	Let $\kappa$ be such that $\kappa^{\aleph_0} = \kappa$, $A \models T^+$ with $|A| = \kappa$ and $A \preccurlyeq \mathfrak{M}$. By Assumption \ref{main_th}(\hyperref[the_hf_axiom]{C2}) it follows that $A \oleq  \mathfrak{M}$ and thus by \ref{existence_HF_orders} we can find a $\mrm{HF}$-ordering $<'$ of $\mathfrak{M}$ over $A$. Additionally, it also follows from \ref{existence_HF_orders} that we can extend $<'$ to a full $\mrm{HF}$-order $<$ of $\mathfrak{M}$ with $A$ as an initial segment. By \ref{ass:hf_closure}(\hyperref[the_hf_axiom]{C3}) we also assume without loss of generality that $A=\mrm{cl}_<(A)$.
	
	\smallskip
	\noindent Let $a,b \in \mathfrak{M}\setminus A$, let $A_0 = \mrm{cl}_{<}(a)$ and $B_0 = \mrm{cl}_<(b)$. By Assumption \ref{ass:hf_closure}(\hyperref[the_hf_axiom]{C3}) we have that $|A_0|\leq \aleph_0$, $|B_0|\leq \aleph_0$, and also  that $AA_0\oleq \mathfrak{M}$, $AB_0\oleq \mathfrak{M}$.  Now, suppose w.l.o.g. that $A_0 \cap A = B_0 \cap A$ and that there is an isomorphism $f: A_0 \cong B_0$ such that $f \restriction A_0 \cap A = \mrm{id}_{A_0 \cap A}$ and $f(a) = b$. Let $h:AA_0\to AB_0$ be such that $f\subseteq h$ and $h\restriction A=\mrm{id}_A$, then $h$ is an isomorphism. By Proposition~\ref{lemma:homogeneity} we have that $\mathfrak{M}$ is $(\mathcal{K}, \oleq,\kappa)$-homogeneous, whence  $h$ extends to an $A$-automorphism of $\mathfrak{M}$. We conclude that, modulo $\kappa$ many elements, the type of an element $a\in \mathfrak{M}$ over a set $A$ of size $\kappa$ is determined by its type over $\mrm{cl}_{<}(a)$. Since by Definition~\ref{def_closure_operator} we have $|\mrm{cl}_{<}(a)|\leq \aleph_0$ and $\kappa^{\aleph_0} = \kappa$, this completes the proof of the former statement. 
\end{proof}

Up to this point, we have proved that the theory $T^+$ is both complete and stable, provided that the assumptions in \ref{main_th}(\hyperref[completeness_axiom]{C1})-(\hyperref[the_K_homogeneous_lemma]{C2}) and \ref{ass:hf_closure}(\hyperref[the_hf_axiom]{C3}) are both satisfied. Since $T^+$ is stable, it follows that the relation  $\ind$ of forking independence is well-behaved, in the sense that it satisfies the key properties from \cite[Thm.~8.55]{tent_book}. 

\begin{remark}\label{stationarity}
	In the following, we always use the symbol $\ind$  to denote the relation of \textit{forking independence}, and we refer the reader to \cite[\S7.1]{tent_book} for its definition.  Let $\bar{a}\in \mathfrak{M}^{<\omega}$ and $B,C\subseteq \mathfrak{M}$, we write $\bar{a} \ind_B C$ if the type $\type(\bar{a}/BC)$ \emph{does not fork} over $B$, and we write $\bar{a} \nind_B C$ if the type $\type(\bar{a}/BC)$ \emph{does fork} over $B$.
\end{remark}

Next, we prove the following important lemma, which establishes that if $A\oleq A\bar{b}\oleq M\models T^+$, then the type of $\bar{b}$ over $A$ is uniquely determined by its quantifier-free part. We notice that, in most of our applications, we have by Corollary~\ref{model_completeness_QE} that $T^+$ does not have a prime model, it is not model complete and it does not eliminate quantifiers. However, exactly as in \cite[Thm.~1.6]{paolini&hyttinen}, one can show that the following proposition entails that in $T^+$ every formula is equivalent to a Boolean combination of existential formulas.

\begin{proposition}\label{lemma:quantifier-free-type}
	Let $A\oleq M\models T^+$ and $\bar{b},\bar{c}\in M^{<\omega}$. If $\atype(\bar{b}/A)=\atype(\bar{c}/A)$ and $ A\bar{b}\oleq \mathfrak{M}$, $A\bar{c}\oleq \mathfrak{M}$, then $\type(\bar{b}/A)=\type(\bar{c}/A)$.
\end{proposition}
\begin{proof}
	Let $|M|=\kappa$ and assume without loss of generality that $M\preccurlyeq \mathfrak{M}$. Then by Theorem~\ref{prop_completeness} we have that $M\oleq \mathfrak{M}$, whence we obtain $A\oleq A\bar{b}\oleq \mathfrak{M}$ and $A\oleq A\bar{c}\oleq \mathfrak{M}$. Let $\bar{b}=(b_1,\dots,b_n)$ and $\bar{c}=(c_1,\dots,c_n)$, let  $f:A\cup \{b_1,\dots,b_n\}\to A\cup \{c_1,\dots,c_n\}$ be a partial isomorphism witnessing that $\atype(\bar{b}/A)=\atype(\bar{c}/A)$, in particular $f\restriction A=\mrm{id}_A$ and we let $f(b_i)=c_i$ for all $1\leq i\leq n$. Since $\mathfrak{M}$ is $(\mathcal{K}, \oleq,\kappa^+)$-homogeneous by  \ref{main_th}(\hyperref[the_K_homogeneous_lemma]{C2}), it follows that $f$ can be extended to an $A$-automorphism $\pi$ of $\mathfrak{M}$, witnessing that  $\type(\bar{b}/A)=\type(\bar{c}/A)$.
\end{proof}

We now introduce the following relation $\ind^{\otimes}$ of \emph{free independence}, which we define purely in terms of the free amalgam from Definition~\ref{general_free_amalgam}.

\begin{definition}\label{def:free_independence}
	Let $A, B, C\subseteq \mathfrak{M}$ (and recall from \ref{monster:notation} that $\mathfrak{M}$ is the monster model of $T^+$), we define the relation   $\ind^{\otimes}$ by letting
	\begin{align*}
		B\ind^{\otimes}_A C \; \Longleftrightarrow \; \mrm{icl}(ABC)= \mrm{icl}(AB)\otimes_{\mrm{icl}(A)}\mrm{icl}(AC).
	\end{align*}
\end{definition}

We recall that we always assume that the monster model $\mathfrak{M}$ is sufficiently saturated and homogeneous in the language $L$, whence by \ref{main_th}(\hyperref[completeness_axiom]{C1})-(\hyperref[the_K_homogeneous_lemma]{C2}) $\mathfrak{M}$ is also  $(\mathcal{K}, \oleq,\kappa)$-saturated and  $(\mathcal{K}, \oleq,\kappa)$-homogeneous for cardinals $\kappa$ which are large enough. The strategy that we employ in the following arguments is essentially the same as in \cite{tent}, although we defined the relation $\ind^{\otimes}$ differently. In particular, we use the fact that in stable theories forking independence is the only independence relation satisfying the key properties from Fact \ref{properties.of.forking} below (cf.~\cite[Thm. 8.5.10]{tent_book}).

\begin{fact}\label{properties.of.forking}
	Let $T$ be a stable theory, then the relation of forking independence $\ind$ is the only ternary relation on subsets of $\mathfrak{M}$ which satisfies the properties below:
	\begin{enumerate}[(a)]
		\item \emph{Invariance}:  let $A,B,B',C\subseteq \mathfrak{M}$, if $B\ind_A C$ and $\type(B/C)=\type(B'/C)$, then $B'\ind_A C$;
		\item \emph{Local Character}:  there is a cardinal $\kappa$ such that for all finite $B\subseteq \mathfrak{M}$ and $C\subseteq \mathfrak{M}$, there is a subset $A\subseteq \mathfrak{M}$ of size $\leq \kappa$ such that $B\ind_A {C}$;
		\item \emph{Weak Boundedness}: for all $A\subseteq \mathfrak{M}$ there is a cardinal $\mu$ such that, for all $B,B'\subseteq_\omega \mathfrak{M}$ and $A\subseteq C\subseteq \mathfrak{M}$, there are at most $\mu$ many types $\type(B'/C)$ such that $B'\ind_A C$ and $\type(B'/A)=\type(B/A)$;
		\item \emph{Existence}:  for all finite $B\subseteq \mathfrak{M}$ and $A\subseteq C\subseteq \mathfrak{M}$, there is some $B'\subseteq \mathfrak{M}$ such that $B'\ind_A C$ and $\type(B/A)=\type(B'/A)$;
		\item \emph{Transitivity}: let $B\subseteq_\omega \mathfrak{M}$ and $A\subseteq C\subseteq D\subseteq \mathfrak{M}$, if $B\ind_A C$ and $B\ind_C D$, then $B\ind_A D$;
		\item \emph{Weak monotonicity}: let $A,B,C,D\subseteq \mathfrak{M}$, if $B\ind_A C$ and $D\subseteq C$, then $B\ind_A D$.
	\end{enumerate}	
\end{fact}

\noindent By the fact above, to show that $\ind=\ind^\otimes$ it suffices to verify that the relation $\ind^\otimes$ satisfies all the properties of forking independence.

\begin{theorem}\label{characterisation_forking}
	Suppose $T^+$ satisfies \ref{main_th}(\hyperref[completeness_axiom]{C1})-(\hyperref[the_K_homogeneous_lemma]{C2}) and \ref{ass:hf_closure}(\hyperref[the_hf_axiom]{C3}). Then the relation $\ind^\otimes$ satisfies all the properties of forking independence from Fact \ref{properties.of.forking}. In particular, for any $A,B,C\subseteq \mathfrak{M}$, we have that $B \ind_A C$ is equivalent to $B \ind^{\otimes}_A C$.
\end{theorem}
\begin{proof}
	We verify that $\ind^\otimes$ satisfies all the properties from Fact \ref{properties.of.forking}.	
	\medskip 
	\noindent \underline{Invariance}: let $A,B,B',C\subseteq \mathfrak{M}$, if $B\ind^\otimes_A C$ and $\type(B/C)=\type(B'/C)$, then $B'\ind^\otimes_A C$.
	
	\smallskip
	\noindent  We assume without loss of generality that $A\subseteq B\cap C$ and $A=\mrm{icl}(A)$, $B=\mrm{icl}(B)$ and $C=\mrm{icl}(C)$. Let $\pi$ be a $C$-automorphism of $\mathfrak{M}$ witnessing that $\type(B/C)=\type(B'/C)$ then $B'C=\pi(BC)=\pi(B\otimes_A C)=B'\otimes_A C$. Moreover, since $\Pi$ is an automorphism of $\mathfrak{M}$ and $B\otimes_{A} C=\mrm{icl}(BC)\oleq \mathfrak{M}$, we also have that $B'\otimes_{A} C=\mrm{icl}(B'C)\oleq \mathfrak{M}$. This shows that $B\ind^\otimes_A C$ if and only if $B'\ind_A^\otimes C$.

	\medskip 
	\noindent \underline{Weak Monotonicity}: let $A,B,C,D\subseteq \mathfrak{M}$, if $B\ind^\otimes_A C$ and $D\subseteq C$, then $B\ind^\otimes_A D$.
	
	\smallskip
	\noindent Suppose without loss of generality that $A\subseteq B\cap C$ and $A=\mrm{icl}(A)$, $B=\mrm{icl}(B)$ and $C=\mrm{icl}(C)$. If $B\ind^\otimes_A C$ then we have that $BC=B\otimes_A C\oleq \mathfrak{M}$. Then clearly for $D\subseteq C$ we have that $\mrm{icl}(D)\subseteq \mrm{icl}(C)$ (by Definition~\ref{baldwin:intrinsic_closure}), and so $B\cup \mrm{icl}(D)=B\otimes_A \mrm{icl}(D)$. Moreover, since we have that $\mrm{icl}(D)\oleq C\oleq \mathfrak{M}$, it follows also from Conditions \ref{context}(\ref{condition:A})-(\ref{condition:K}), that $B\otimes_A \mrm{icl}(D)\oleq B\otimes_A C \oleq \mathfrak{M}$, proving our claim.

	\medskip 
	\noindent \underline{Weak Boundedness}:  for all $A\subseteq \mathfrak{M}$ there is a cardinal $\mu$ such that, for all $B,B'\subseteq_\omega \mathfrak{M}$ and $A\subseteq C\subseteq \mathfrak{M}$, there are at most $\mu$ many types $\type(B'/C)$ such that $B'\ind_A C$ and $\type(B'/A)=\type(B/A)$;
	
	\smallskip
	\noindent Suppose without loss of generality that $A\subseteq B_0$ $A\subseteq B_1$, $B_0\cap C=B_1\cap C=A$ and $A=\mrm{icl}(A)$, $B_0=\mrm{icl}(B_0)$, $B_1=\mrm{icl}(B_1)$,  $C=\mrm{icl}(C)$. We assume that $B_0\ind^\otimes_A C$, $B_1\ind^\otimes_A C$ and  $\type(B_0/A)=\type(B_1/A)$, so in particular there is an isomorphism $f:B_0\to B_1$ such that $f\restriction A=\mrm{id}_A$. Also, by the definition of $\ind^\otimes$ it follows that $\mrm{icl}(B_0C)=B_0\otimes_{A} C\oleq \mathfrak{M}$ and $\mrm{icl}(B_1C)=B_1\otimes_{A} C\oleq \mathfrak{M}$. Then $f$ together with the identity map $\mrm{id}_C$ determines an isomorphism $g:B_0\otimes_A C\to B_1\otimes_A C$ such that $g\restriction B_0=f$ and  $g\restriction C=\mrm{id}_C$. It follows that $\atype(B_0/C)=\atype(B_1/C)$, and so by Proposition \ref{lemma:quantifier-free-type} we have   $\type(B_0/C)=\type(B_1/C)$. Since for any $\type(D/A)$ there at most $2^{|A|+\aleph_0}$-completions of  $\type(D/A)$ to $\type(\mrm{icl}(D)/A)$, this proves our claim.

	\medskip 
	\noindent \underline{Existence}: for all finite $B\subseteq \mathfrak{M}$ and $A\subseteq C\subseteq \mathfrak{M}$, there is some $B'\subseteq \mathfrak{M}$ such that $B'\ind^\otimes_A C$ and $\type(B/A)=\type(B'/A)$.
	
	\smallskip
	\noindent Let $B\subseteq \mathfrak{M}$ and $A\subseteq C\subseteq \mathfrak{M}$, by the definition of $\ind^\otimes$ we can assume w.l.o.g. that $A\subseteq B$, $A=\mrm{icl}(A)$, $B=\mrm{icl}(B)$, $C=\mrm{icl}(C)$ and also that the extension $A\oleq B$ is minimal. If $A\oleq B$ is $\mathcal{K}$-algebraic, then by the algebraic amalgamation property we have that either there is a strong embedding $f:B\to C$ over $A$ or $B\otimes_{A} C\in \mathcal{K}$. In case there is a strong embedding $f:B\to C$, then we have that $\mrm{icl}(f(B)C)=\mrm{icl}(C)$ and thus obviously $f(B)\ind^\otimes_A C$. Moreover, since $A\oleq B\oleq \mathfrak{M}$, $A\oleq f(B)\oleq \mathfrak{M}$ and $B\cong_A f(B)$, it follows from Proposition~\ref{lemma:quantifier-free-type} that $\type(B/A)=\type(f(B)/A)$. Similarly, if $B\otimes_{A} C\in \mathcal{K}$ then since $C\oleq B\otimes_{A} C$ there is a strong embedding $f:B\otimes_{A} C\to \mathfrak{M}$ such that $f\restriction C=\mrm{id}_C$. Letting $B'=f(B)$ it follows that $B'\ind_A^\otimes C\oleq \mathfrak{M}$. As in the previous case, since $A\oleq B\oleq \mathfrak{M}$, $A\oleq B'\oleq \mathfrak{M}$ and $B\cong_A B'$, it follows from Proposition~\ref{lemma:quantifier-free-type} that $\type(B/A)=\type(B'/A)$. Finally, in the case when $A\oleq B$ is not $\mathcal{K}$-algebraic, we have by the algebraic amalgamation property that $B\otimes_{A} C\in \mathcal{K}$, and so we can reason exactly as in the previous case.
	
	\medskip 
	\noindent \underline{Transitivity}:  let $B\subseteq_\omega \mathfrak{M}$ and $A\subseteq C\subseteq D\subseteq \mathfrak{M}$, if $B\ind^\otimes_A C$ and $B\ind^\otimes_C D$, then $B\ind^\otimes_A D$.
	
	\smallskip
	\noindent Suppose $B\ind^\otimes_A C$, $B\ind^\otimes_C D$ and assume without loss of generality that $A=\mrm{icl}(A)$, $B=\mrm{icl}(B)$, $C=\mrm{icl}(C)$ and $D=\mrm{icl}(D)$. Since $A\subseteq C\subseteq D\subseteq \mathfrak{M}$, we have that $\mrm{icl}(ABD)= \mrm{icl}(BD) = \mrm{icl}(BCD)$, and therefore:
	\begin{align*}
		\mrm{icl}(ABD) &= \mrm{icl}(BC)\otimes_{C} \mrm{icl}(CD) \\
		&= (\mrm{icl}(AB)\otimes_{A}  \mrm{icl}(AC)) \otimes_{C} \mrm{icl}(CD) \\
		&= \mrm{icl}(AB)\otimes_{A} \mrm{icl}(AD).
	\end{align*}
	Additionally, since $\mrm{icl}(BC)\otimes_{C} \mrm{icl}(CD)\oleq \mathfrak{M}$, it follows from the previous equalities also that $\mrm{icl}(AB)\otimes_{A} \mrm{icl}(AD)\oleq \mathfrak{M}$, whence $B\ind^\otimes_A D$.	

	\medskip 
	\noindent \underline{Local Character}:   there is a cardinal $\kappa$ such that for all finite $B\subseteq \mathfrak{M}$ and $C\subseteq \mathfrak{M}$, there is a subset $A\subseteq \mathfrak{M}$ of size $\leq \kappa$ such that $B\ind^\otimes_A {C}$.
	
	\smallskip
	\noindent  We show that local character holds with respect to $\kappa=\aleph_0$. Let $B\subseteq\mathfrak{M}$ be finite and consider an arbitrary subset $C\subseteq \mathfrak{M}$. By the definition of $\ind^\otimes$ we assume without loss of generality that $C=\mrm{icl}(C)$. Then, since $C\oleq \mathfrak{M}$, it follows from Corollary \ref{equivalence_HFo.ordering}  that there is a $\mrm{HF}$-order $<'$ of $\mathfrak{M}$ over $C$. Again by \ref{equivalence_HFo.ordering}, we can extend $<'$ to a $\mrm{HF}$-order $<$ of $\mathfrak{M}$ with $C$ as its initial segment. Consider now the $\mrm{HF}$-closures $B'=\mrm{cl}_{<}(B)$, $C'=\mrm{cl}_{<}(C)$ and notice that by \ref{ass:hf_closure}(\hyperref[the_hf_axiom]{C3}) and Definition~\ref{def_closure_operator} we have that $B'\oleq \mathfrak{M}$ and $|B'|\leq \aleph_0$. Let $A=B'\cap C'$, then it follows directly from Definition~\ref{def_closure_operator} that $\mrm{cl}_{<}(BC)=B'\otimes_A C'\oleq \mathfrak{M}$, which means that $B'\ind^\otimes_A C'$. Then, since the relation $\ind^\otimes$ is clearly symmetric by definition, it follows from the weak monotonicity of $\ind^\otimes$ that $B\ind^\otimes_A C$, which completes our proof.
\end{proof}

In most of our concrete examples we shall deal with theories of incidence structures which have an associated notion of \emph{free object}, such as the free projective planes studied in \cite{paolini&hyttinen} or the free generalised $n$-gons studied in \cite{tent}. We unify these cases and provide an additional characterisation of the forking relation, which applies to these situations, and that parallels the description provided in \cite{paolini&hyttinen,tent}. In particular, using the notion of free amalgam from Definition~\ref{general_free_amalgam} and Definition~\ref{algebraic_number},  we can define the notion of \emph{free algebraic completion} of some $A\models T_\forall$. This was essentially introduced  by Funk and Strambach in \cite[\S 1.4]{funk2}, where they provide a uniform model-theoretic definition of the free completions from the incidence geometry literature. We notice that we take minimal algebraic strong extensions as the building blocks of free completions, while for Funk and Strambach this role is played by the so-called \emph{bricks} (cf.~\cite[Def. 3]{funk2}).

\begin{definition}\label{free_algebraic_completion}
	Let $A\models T_\forall$, we define its \emph{free algebraic completion $F(A)$} as follows. Let $A_0=A$ and suppose $A_n$ is given. We let $(A_n,A_nB_i)_{i<m}$ enumerate all minimal strong $\mathcal{K}$-algebraic extensions of $A_n$, up to isomorphism. For every $i<m$, we let $n_i$ be the degree of $(A_n,A_nB_i)$ (recall \ref{algebraic_number}), and we let $(A_nB^j_i)_{1\leq j\leq n_i}$ be isomorphic copies of $A_nB_i$ such that $A_nB^j_i\cap A_nB^k_i=A$ for $j\neq k$. Then we define, for all $i<m$:
	\[ A^i_{n}\coloneqq A_nB^1_i\otimes_{A_n}A_nB^2_i\otimes_{A_n}\dots \otimes_{A_n} A_nB^{n_i}_i \]
	and we let
	\[  A_{n+1}\coloneq  A^0_{n}\otimes_{A_n} A^1_{n} \otimes_{A_n} \cdots \otimes_{A_n}A^{m-1}_n.  \]
	Finally, we let $F(A)=\bigcup_{n<\omega}A_n$ be the free (algebraic) completion of $A$.
\end{definition}

The following Proposition \ref{key_properties:free_completion} is essentially a corollary of the previous assumptions. It captures the key properties of the free completion $F(A)$. We then employ \ref{key_properties:free_completion} in \ref{characterisation_acl} to provide a characterisation of the algebraic closure in models of $T^+$. 

\begin{proposition}\label{key_properties:free_completion}
	Let $A\models T_\forall$, then $A\oleq F(A)\models T_\forall$, $|F(A)|=|A|+\aleph_0$, and:
	\begin{enumerate}[(a)]
		\item every isomorphism $h:A\cong B$ extends to an isomorphism $\hat{h}:F(A)\cong F(B)$;
		\item for every $M\models T^+$ and $h:A\to M$ with $h(A)\oleq M$ there is an embedding $\hat{h}: F(A)\to M$ such that $\hat{h}\restriction A = h$ and $\hat{h}(F(A))\oleq M$.
	\end{enumerate}
\end{proposition}
\begin{proof}
	The fact that  $|F(A)|=|A|+\aleph_0$ follows immediately by construction. Similarly $F(A)\models T_\forall$ since for every $n<\omega$ the algebraic amalgamation property entails that $A_n\models T_\forall$. Also, for every $n<\omega$ in the construction of $F(A)$, it follows from \ref{remark_pushout} that $A_n\oleq A_{n+1}$, whence we obtain that $A\oleq F(A)$.
	
	\smallskip
	\noindent  Consider (a). By the construction of $F(A)$, it is immediate to extend an isomorphism $h:A\cong B$  to an isomorphism $\hat{h}:F(A)\cong F(B)$.
	
	\smallskip
	\noindent Consider (b). Let $h:A\to M\models T^+$ with $h(A)\oleq M$. Suppose inductively that we have found an embedding $h_n\supseteq h$ such that $h_n: A_n\to M$ and $h_n(A_n)\oleq M$. Now, since $h_n(A_n)\oleq M$ we clearly also have that $h_n(A_n)\leq_n M$ for all $n<\omega$. It follows from the definition of $T^+$ (cf.~\ref{the_theory}) that every minimal strong extension of $h_n(A_n)$ is realised in $M$. Thus we can extend  $h_n$ to an embedding $h_{n+1}:A_{n+1}\to M$. Moreover, since $A_{n+1}$ is obtained from $A_n$ by adjoining only $\mathcal{K}$-algebraic strong extensions, it follows from \ref{condition:amalgam} that $h_{n+1}(A_{n+1})\oleq M$. Finally, let $\hat{h}=\bigcup_{n<\omega} h_n$, then clearly $\hat{h}: F(A)\to M$ is an embedding and $\hat{h}(F(A))\oleq M$.
\end{proof}

\begin{remark}
	Suppose $A\oleq M\models T^+$, from now on, with some slight abuse of notation, we identify $F(A)$ with its isomorphic image in $M$ (which exists by Proposition \ref{key_properties:free_completion}) and we simply write $F(A)\oleq M$. We also notice that, if $A\oleq B\otimes_A C$, then by Condition \ref{context}(\ref{condition:K}) it follows that $A\oleq B$ is algebraic if and only if $C\oleq B\otimes_A C$ is algebraic. By Definition \ref{free_algebraic_completion}, this in particular means that, similarly to what happens in \ref{key_properties:free_completion}, $A\oleq C$ entails $F(A)\oleq F(C)$.
\end{remark}

\begin{proposition}\label{characterisation_acl}
	Let $A\subseteq B\models T^+$, then $\mrm{acl}_B(A)=F(\mrm{icl}_B(A))$.
\end{proposition}
\begin{proof} 
	By \ref{context}(\ref{condition:H}) we immediately have that $\mrm{icl}_B(A)\subseteq \mrm{acl}_B(A)$. Moreover, by Definition \ref{free_algebraic_completion} of the free completion $F(\mrm{icl}_B(A))$ and \ref{key_properties:free_completion} it follows that  $F(\mrm{icl}_B(A))\subseteq \mrm{acl}_B(A)$. Consider the converse direction. By Definition \ref{baldwin:intrinsic_closure} it follows that $\mrm{icl}_B(A)\oleq B$, so then by Proposition \ref{key_properties:free_completion} we obtain $F(\mrm{icl}_B(A))\oleq B$. Now, since we have that $\mrm{acl}^{\mrm{qf}}_B(F(\mrm{icl}_B(A)))\oleq \mathfrak{M}$, by Proposition \ref{lemma:quantifier-free-type}  it is sufficient to show that $\mrm{acl}^{\mrm{qf}}_B(F(\mrm{icl}_B(A)))= F(\mrm{icl}_B(A))$.  Consider some $c\in \mrm{acl}^{\mrm{qf}}_B(A)=\mrm{acl}^{\mrm{qf}}_B(F(\mrm{icl}_B(A)))$ and suppose towards contradiction that $c\notin F(\mrm{icl}_B(A))$. Then, there is a step $A_n$ from the construction of the free completion $F(\mrm{icl}_B(A))$ (cf. \ref{general_free_amalgam}) such that $c$, together possibly with some other elements, is algebraic over $A_n$. Let $C$ be a minimal extension of $A_n$ which contains $c$ and is algebraic over $A$. Since  $F(\mrm{icl}_B(A))\oleq B$ we also have that $A_n\oleq C$. It follows from Definition \ref{general_free_amalgam} that $C\subseteq A_{n+1}\subseteq F(\mrm{icl}_B(A))$, and so $c\in F(\mrm{icl}_B(A))$, contrary to our assumption.
\end{proof}

We next introduce the notion of canonical amalgam and we show, generalising the prior work from \cite{paolini&hyttinen} and \cite{tent}, that $\ind^\otimes$ can be rephrased in terms of canonical amalgamation.

\begin{definition}\label{def_canonical_amalgam} Let $A, B, C\models T_\forall$ be  such that $A\oleq B$, $A\oleq C$, $B \cap C = A$ and suppose that $A$ has no $\mathcal{K}$-algebraic strong extension in neither $B$ nor $C$. The {\em canonical amalgam} $B \oplus_A C$ of  $B$ and $C$ over $A$ is the structure
	\[ B \oplus_A C\coloneqq F(B \otimes_A C),   \]
	i.e., it is the free completion of the free amalgam $B \otimes_A C$.
\end{definition}

\noindent We conclude this section with the following additional characterisation of the forking relation $\ind$. This generalises the description of forking provided in \cite{paolini&hyttinen,tent}.

\begin{theorem}\label{characterisation_forking_3}
	Suppose $T^+$ satisfies \ref{main_th}(\hyperref[completeness_axiom]{C1})-(\hyperref[the_K_homogeneous_lemma]{C2}) and \ref{ass:hf_closure}(\hyperref[the_hf_axiom]{C3}). Then for all $A,B,C\subseteq \mathfrak{M}$ we have that $B \ind_A C$ is equivalent to the following condition:
\[\mrm{acl}(ABC)= \mrm{acl}(AB) \oplus_{\mrm{acl}(A)} \mrm{acl}(AC).\]
\end{theorem}
\begin{proof}
	Immediate from Theorem \ref{characterisation_forking}, Proposition~\ref{key_properties:free_completion} and Proposition \ref{characterisation_acl}.
\end{proof}

\subsection{Failure of superstability}\label{subsec:no_superstability}

It was shown in \cite{paolini&hyttinen} that the theory of infinite open projective planes is stable but not superstable. Later, this result was generalised in \cite{tent} to the case of all open generalised $n$-gons. We prove in this section that \emph{most, but not all} of the examples of theories $T^+$ that we study in this paper are not superstable. More precisely, we identify the combinatorial condition  \ref{assumption:no-superstability_2}(\hyperref[CP]{C4})  and we show that, when satisfied, it entails the non-superstability of $T^+$. Interestingly, this conditions behaves as an important dividing line among our class of intended applications. While all the geometries considered in Section \ref{sec:application} satisfy the assumptions from \ref{main_th}(\hyperref[completeness_axiom]{C1})-(\hyperref[the_K_homogeneous_lemma]{C2}) and \ref{ass:hf_closure}(\hyperref[the_hf_axiom]{C3}), not all of them satisfy Condition \ref{assumption:no-superstability_2}(\hyperref[CP]{C4}) as well. The crucial counterexamples are the theories $T^+_0$ and $T^+_1$ from Section \ref{sec:n-open-grpahs}. While $T^+_0$ is simply the theory of an infinite set, $T^+_1$ is the theory of the free pseudoplane, which is well-known to be $\omega$-stable (cf.~\cite{tent2}). This shows that our framework applies to both superstable and non-superstable classes of models.

\begin{assumption}\label{assumption:no-superstability_2}
	Let $(\mathcal{K}, \oleq)$ and $T^+$ be respectively as in \ref{context} and \ref{the_theory}. In this section we assume that $T^+$ satisfies the following condition:
	\begin{enumerate}[(C4)]
		\item\label{CP}  there is a finite configuration $B\subseteq_\omega \mathfrak{M}$ and $A\subseteq \mrm{icl}(B)$ with $|\mrm{icl}(B)|=\aleph_0$, $A=\{a_i : i<\omega \}$ and $A_n = \{a_i : i\leq n\}\oleq \mrm{icl}(B)$ for all $n<\omega$.
	\end{enumerate}
\end{assumption}

\begin{remark}\label{remark:construction_principle}
	Under Assumptions \ref{main_th}(\hyperref[completeness_axiom]{C1})-(\hyperref[the_K_homogeneous_lemma]{C2}) and \ref{ass:hf_closure}(\hyperref[the_hf_axiom]{C3}) what Assumption \ref{assumption:no-superstability_2}(\hyperref[CP]{C4}) essentially says is that there is some $B\oleq \mathfrak{M}$ whose associated $\mrm{HF}$-order is \emph{not well-founded}. This observation provides an important bridge between  Assumption \ref{assumption:no-superstability_2}(\hyperref[CP]{C4}) and the so-called \emph{Construction Principle (CP)}, originally introduced by Eklof and Mekler in their work on $\mathfrak{L}_{\infty,\kappa}$-algebras (cf.~\cite{EM}). Generalising their approach, in \cite{HPQ} we introduced an abstract version of the Construction Principle $\mrm{CP}(\mathbf{K},\ast)$ for classes with free amalgamation. If one defines the relation $A\leq_{\ast} B$ to hold whenever there is a well-founded $\mrm{HF}$-order of $B$ over $A$ (which is essentially what we did in \cite{HPQ} for Steiner systems and generalised $n$-gons), then the Construction Principle for $T^+$ says exactly that there are a chain of models $(A_i)_{i<\omega}$   of $T^+$ and $B\models T^+$ such that $A_i \leq_{\ast} B $ for all $i<\omega$ but $\bigcup_{i<\omega}A_i\not \leq_{\ast} B$. In \cite{HPQ} we provided concrete constructions and showed that $\mrm{CP}(\mathbf{K},\ast)$ holds both in the context of Steiner systems and generalised $n$-gons. Together with the results from \cite{HPQ2} it is then easy to show that the theories of open $(k,n)$-Steiner systems and open generalised $n$-gons are not superstable. Here we proceed differently and we show that, if $T^+$ satisfies \ref{assumption:no-superstability_2}(\hyperref[CP]{C4}), then it is not superstable. We point out that the requirements in \ref{assumption:no-superstability_2}(\hyperref[CP]{C4}) are weaker than those in $\mrm{CP}(\mathbf{K},\ast)$, since in \ref{assumption:no-superstability_2}(\hyperref[CP]{C4}) we do not require any of the configurations $A_i$ or $B$ to be a model of $T^+$, nor we ask that $B$ admits a wellfounded $\mrm{HF}$-order over the configurations $A_i$.
\end{remark}

We recall that a stable theory $T$ is superstable if, for all finite $B\subseteq \mathfrak{M}$ and $C\subseteq \mathfrak{M}$, there is a finite subset $A\subseteq_\omega C$ such that $B\ind_{A}C$ (cf.~\cite[Def.~8.6.3]{tent_book}). Thus, superstable theories are exactly those stable theories where the property of local character from Fact \ref{properties.of.forking} holds already with respect to finite subsets. From this fact and the previous characterisation of forking independence from Theorem \ref{characterisation_forking} it is possible to show that \ref{assumption:no-superstability_2}(\hyperref[CP]{C4}) entails the non-superstability of $T^+$.

\begin{theorem}\label{failure:superstability}
	Suppose $T^+$ satisfies \ref{main_th}(\hyperref[completeness_axiom]{C1})-(\hyperref[the_K_homogeneous_lemma]{C2}), \ref{ass:hf_closure}(\hyperref[the_hf_axiom]{C3}) and \ref{assumption:no-superstability_2}(\hyperref[CP]{C4}). Then $T^+$ is not superstable.
\end{theorem}
\begin{proof}
	Let $B\subseteq_\omega \mathfrak{M}$, $A=\{a_i:i<\omega\}\subseteq \mrm{icl}(B)$ and $A_n=\{a_i:i<\leq n\}$ be as in Assumption \ref{assumption:no-superstability_2}(\hyperref[CP]{C4}). For all ${n<\omega}$ we have by Theorem \ref{characterisation_forking} that $B\ind_{A_n} A$ if and only if $\mrm{icl}(B)=\mrm{icl}(B)\otimes_{\mrm{icl}(A_n)}\mrm{icl}(A)$. Since $A_n\oleq B\oleq \mathfrak{M}$ it follows that $\mrm{icl}(A_n)=A_n$ and so $\mrm{icl}(B)\cap \mrm{icl}(A)=\mrm{icl}(A)\neq \mrm{icl}(A_n)$. This shows that, for all $n<\omega$, we have $\mrm{icl}(B)\neq \mrm{icl}(B)\otimes_{\mrm{icl}(A_n)}\mrm{icl}(A)$ and thus $B\nind_{A_n} A$. It follows that $T^+$ is not superstable.
\end{proof}

In order to simplify our treatment of non-superstability in the concrete setting of the theories from Section~\ref{sec:application}, we introduce a set of technical assumptions in \ref{assumption:no-superstability}(\hyperref[technical:D4]{D4}) and we show it entails the abstract Assumption \ref{assumption:no-superstability_2}(\hyperref[CP]{C4}) above. We shall then prove in all cases from Section~\ref{sec:application} that they verify \ref{assumption:no-superstability}(\hyperref[technical:D4]{D4}). To this end, we first introduce the following technical notion of \emph{$k$-iterate} of a finite configuration.

\begin{definition}\label{construction:no_superstability}
	Let $C=\{ c_0,\dots, c_n\}\in \mathcal{K}$ be a configuration where $c_0,c_n$ have the same sort, and let $1\leq k < \omega$. We define its \emph{$k$-iterate} $I_k(C)$ with respect to the enumeration $C=\{ c_0,\dots, c_n\}$ as the structure $D$ defined as follows:
	\begin{enumerate}[(1)]
		\item the domain of $D$ is $\{d^i_j : i< k, \; j\leq n \}$, where $d^{i}_j$ and $d^\ell_j$ have the same sort for all $i<\ell< k$ and $j\leq n$, and with the constraint that $d^i_n=d^{i+1}_0$ for all $i<k$;
		\item for every incidence symbol $R\in L_{\mrm{i}}$ we let $D\models R(d^{i}_{j_0}, d^{i}_{j_1},\dots, d^{i}_{j_n})$ if and only if $C\models R(c_{j_0}, c_{j_1}, \dots, c^{i}_{j_1})$, for all $j_0,j_1,\dots, j_n\leq n$ and $i<k$;
		\item for every (local) equivalence relation $P\in L_{\mrm{p}}$ we let $D\models P(d^{i}_{j_0}, d^{i}_{j_1}, d^{i}_{j_2},\dots, d^{i}_{j_\ell})$ if $C\models P(c_{j_0}, c_{j_1},c_{j_2},\dots, c_{j_\ell})$, for all $j_0,j_1,\dots, j_\ell \leq n$ and $i<k$, and we then transitively close the relation $P(x,y,\bar{d})$ for all $\bar{d}\in D^{<\omega}$;
		\item no other relation holds in $D$ than those above.
	\end{enumerate}
\end{definition}

\begin{technical}\label{assumption:no-superstability}
	Suppose $T^+$ satisfies \ref{main_th}(\hyperref[completeness_axiom]{C1})-(\hyperref[the_K_homogeneous_lemma]{C2}), \ref{ass:hf_closure}(\hyperref[the_hf_axiom]{C3}) and \ref{assumption:no-superstability_2}(\hyperref[CP]{C4}), in this section we also assume that $T^+$ satisfies the following condition:
	\begin{enumerate}[(D4)]\label{technical:D4}
		\item there is a configuration $C=\{c_0,c_1,\dots,c_n \}\in \mathcal{K}$ such that:
		\begin{enumerate}[(a)]
			\item $c_0c_1\oleq C$ and there is no relation between $c_0$ and $c_1$;
			\item  $c_n\noleq C$ is confined (cf.~\ref{def_extensions}) and $c_n$ has the same sort of $c_0$;
			\item for every $k<\omega$, the $k$-iterate of $C$ from Definition \ref{construction:no_superstability} is in $\mathcal{K}$.
		\end{enumerate}
	\end{enumerate}
\end{technical}

We next prove that, when $T^+$ satisfies the former technical condition, then it also satisfies \ref{assumption:no-superstability_2}(\hyperref[CP]{C4}). The key point of the following proof is the construction of an element $d^0_n$ whose intrinsic closure is infinite. We crucially use the $k$-iterate of the configuration $C$ from  \ref{assumption:no-superstability}(\hyperref[technical:D4]{D4}).

\begin{proposition}\label{prop:technical_superstability}
	Suppose $T^+$ satisfies \ref{main_th}(\hyperref[completeness_axiom]{C1})-(\hyperref[the_K_homogeneous_lemma]{C2}), \ref{ass:hf_closure}(\hyperref[the_hf_axiom]{C3}) and also the technical assumption  \ref{assumption:no-superstability}(\hyperref[technical:D4]{D4}). Then  $T^+$ also satisfies \ref{assumption:no-superstability_2}(\hyperref[CP]{C4}).
\end{proposition}
\begin{proof}
	Let $C=\{c_0,c_1,\dots,c_n \}\in \mathcal{K}$ be as in \ref{assumption:no-superstability}. Let $A=\{a_i :i<\omega\}\oleq \mathfrak{M}$ be a countable set containing infinitely many elements of the same sort as $c_1$, with no relation between them. Also, we assume without loss of generality that $A_n=\{a_i :i\leq n\}\oleq A$ for all $n<\omega$.  By the saturation of $\mathfrak{M}$ we can find a copy $C_0$ of $C$ such that $A\oleq AC_0\oleq \mathfrak{M}$ and $a_0$ plays the role of $c_0$ in $C$. Thus letting $C_0= \{c^0_0,c^0_1,\dots,c^0_n \}$ we have $a_0=c_0^0$. Now, suppose inductively that we have found configurations $C_0,\dots, C_k$ such that they are all isomorphic to the configuration $C$ from \ref{assumption:no-superstability} and
	\[ A\oleq AC_0\oleq \dots \oleq AC_0\dots C_{k-1} \oleq AC_0\dots C_{k-1}C_k\oleq \mathfrak{M}. \]
	Then for every $\ell\leq k$ let $C_\ell=\{c^\ell_0,c^\ell_1,\dots,c^\ell_n \}$ be the enumeration inherited from the isomorphism with $C$. Additionally, we assume inductively that $c^{\ell+1}_0=c^{\ell}_n$ for all $\ell< k$ and $c^\ell_1=a_\ell$ for all $\ell\leq k$. Then, by \ref{assumption:no-superstability}(c), we can iterate this procedure and obtain (by \ref{assumption:no-superstability}(a) and \ref{context}(\ref{condition:K})) that $AC_0\dots C_{k-1}C_k\oleq AC_0\dots C_kC_{k+1}$, where $C_{k+1}$ is a configuration isomorphic to $C$ from \ref{assumption:no-superstability}, but where $c^{k+1}_0=c^k_n$ and $c^{k+1}_1=a_{k+1}$. By the saturation of $\mathfrak{M}$ we then have (modulo isomorphism) that $AC_0\dots C_kC_{k+1}\oleq \mathfrak{M}$. Thus we can continue this construction for all $k<\omega$.
	
	\smallskip
	\noindent Now, using the saturation of $\mathfrak{M}$, we can find a sequence of configurations $(D_i)_{i<\omega}$ in $\mathfrak{M}$ such that, for every $i<\omega$, $D_i=\{d^i_0,d^i_1,\dots,d^i_n \}$ is isomorphic to $C$, $d^i_n=d^{i-1}_0$ and $d^i_1=a_i\in A$. Moreover, for all $i<j<\omega$ we have:
	\[A\oleq A\cup \bigcup_{i< j}D_j \oleq A\cup \bigcup_{i\leq j}D_j\oleq   \mathfrak{M}. \] 
	Intuitively, while the chain of configurations $(C_i)_{i<\omega}$ is \emph{ascending}, the chain of configurations $(D_i)_{i<\omega}$ is \emph{descending}. We claim that $A$ together with $\mrm{icl}(d^0_n)$ satisfy the conditions from \ref{assumption:no-superstability_2}(\hyperref[CP]{C4}).
	
	\smallskip
	\noindent First, we claim that every configuration $D_i$ is contained in $\mrm{icl}(d^0_n)$ and so that $A\subseteq \mrm{icl}(d^0_n)$. By construction $d^0_n\noleq D_0$ is confined and thus it follows from Definition~\ref{baldwin:intrinsic_closure} that $D_0\subseteq \mrm{icl}(d^0_n)$. Suppose inductively that we have showed $D_i\subseteq \mrm{icl}(d^0_n)$ for all $i\leq k$. Then since $d^{k+1}_{n}=d^k_0$ and $d_n^{k+1}\noleq D_{k+1}$ is confined, it follows that $D_{k+1}\subseteq \mrm{icl}(d^k_0)\subseteq \mrm{icl}(d^0_n)$. Since we assumed that $A=\{d^i_1 :i<\omega \}$, it follows that $A\subseteq \mrm{icl}(d^0_n)$ and in particular $|\mrm{icl}(d^0_n)|=\aleph_0$. Therefore, since for all $n<\omega$ we have that $A_n\oleq A\oleq \mathfrak{M}$ and $A\subseteq\mrm{icl}(d^0_n) $, it follows that $A_n\oleq \mrm{icl}(d^0_n)$. Since the former conclusions hold for all $n<\omega$, this verifies that Condition \ref{assumption:no-superstability_2}(\hyperref[CP]{C4}) holds, and thus completes our proof.
\end{proof} 

\subsection{Non-existence of prime models}\label{subsec:no_prime}

We conclude the general part of the article with some negative results about $T^+$. We show that, in the specific case where the algebraic free completion $F(A)$ is already a model of $T^+$, then the theory $T^+$ does not have a prime model, it is not model complete, and thus does not admit quantifier elimination. This is the case in all the incidence geometries considered in \cite{funk2}, but crucially it is false in the case of the free pseudoplane, namely the theory $T^+_1$ from Section \ref{sec:n-open-grpahs}. We identify in this section the additional set of Assumptions \ref{assumptions:no_prime}(\hyperref[F=C]{C5})-(\hyperref[delta-rank]{C7}), from which we establish that $T^+$ does not have a prime model and it is not model complete. We first define an abstract notion of predimension for the classes $(\mathcal{K},\oleq)$.

\begin{definition}\label{delta_strong} \label{open.vs.delta}
	Let $\delta:\mathcal{K}\to \mathbb{N}$ and let, for  $A\subseteq B\models T_\forall $, $A \leq_\delta B$ if and only if $\delta(A\cap B_0)\leq \delta(B_0)$  for every finite $ B_0 \subseteq B$. We say that $\delta$ is a \emph{predimension function for $(\mathcal{K},\oleq)$} if it satisfies the following properties:	
	\begin{enumerate}[(a)]
		\item $A\cong B$ implies $\delta(A)=\delta(B)$;
		\item $A\oleq B$ implies $A \leq_\delta B$;
		\item if $A\oleq B$ is minimal, then $\delta(A)= \delta(B)$ if and only if $B$ is algebraic over $A$;
		\item if $A \leq_\delta B $ and $C \subseteq B$, then $A\cap C \leq_\delta C$. 
	\end{enumerate}
\end{definition}

In \cite[Def.~11]{funk2} Funk and Strambach provide an explicit definition of a predimension function for all the main examples of incidence geometries they consider, with the exception of generalised $n$-gons, whose predimension is defined in \cite{tent}, and $(k,n)$-Steiner systems, that we address in Section \ref{sec_steiner}. We next spell out the additional Assumptions \ref{assumptions:no_prime}(\hyperref[F=C]{C5})-(\hyperref[delta-rank]{C7}). 

\begin{assumption}\label{assumptions:no_prime}\label{delta_lemma_1}
	Let $(\mathcal{K}, \oleq)$ and $T^+$ be respectively as in \ref{context} and \ref{the_theory}. In this section we assume that $T^+$ satisfies the following conditions:
	\begin{enumerate}[(C5)]
		\item[(C5)]\label{F=C} there is $\mathbf{m}_\mathcal{K}<\omega$ such that, if $A$ contains at least $\mathbf{m}_\mathcal{K}$ many elements of every sort then $F(A)\models T^+$;		
		\item[(C6)]\label{hopf} if $F(A)\models T^+$ and $A,B\in \mathcal{K}$ then there is a non-surjective embedding of $F(B)$ into $F(A)$;
		\item[(C7)] \label{delta-rank} there exists a predimension function $\delta:\mathcal{K}\to \mathbb{N}$ which satisfies the properties from Definition \ref{delta_strong}.
	\end{enumerate}		
\end{assumption}

\begin{example}
	In the case of projective planes, the predimension function  $\delta(A)$ is defined as follows. Let $I_A$ be the number of incidences in $A$, i.e., the number of different unordered pairs $\{ a,b\} \subseteq A$ such that $a\;\vert \; b$. Letting $\delta(A)=|A|-2\cdot I_A$ then the conditions from \ref{open.vs.delta} can be verified by a routine argument. One can see that this is exactly the notion of predimension from \cite{funk2} and also \cite[\S 10.4]{tent_book}. We stress that, for projective planes, the notion of predimension and its key properties were already studied by Hall in his seminal paper, see for instance \cite[Theorem 4.10]{hall_proj}
\end{example}

Following the terminology of our concrete examples, we shall here call \emph{non-degenerate} those configurations $A\in \mathcal{K}$ for which $F(A)$ is a model of $T^+$. In all our concrete cases, this will be equivalent to say that $F(A)$ contains infinitely many elements of every sort. From Assumptions \ref{assumptions:no_prime}(\hyperref[F=C]{C5})-(\hyperref[delta-rank]{C7}) the proof that $T^+$ does not have a prime model and it is not model complete follows as in \cite{paolini&hyttinen,tent}, but in our abstract axiomatic setting.

\begin{lemma}\label{the_surjective_lemma} Let $A = F(A_0)$, for $A_0 \in \mathcal{K}$, then every elementary embedding of $A$ into $A$ is surjective.
\end{lemma}
\begin{proof} Let $\alpha: A \rightarrow A$ be an elementary embedding and let $A = \bigcup_{i < \omega} A_i$ be a witness that $A = F(A_0)$. Obviously we have that $A_0 \oleq A$. Let $A'_0 = \alpha(A_0)$. Now, as $\alpha$ is elementary and $A_0 \oleq A$, we have that $A'_0 \oleq A$ (by \ref{the_type_remark}). Furthermore, there is $j < \omega$ such that for every $i\geq j$ we have that $A'_0 \subseteq A_i$. Now, fix one such $j < \omega$ and let $i \geq j$, then, as $A'_0 \oleq A$, we have that $A'_0 \oleq A_i$. In particular, we have that there is a $\mrm{HF}$-order $<$ of $A_i$ over $A'_0$. Let $a_0 < \cdots < a_{m-1}$ be the elements in $A_i \setminus A'_0$ with respect to the order $<$. Suppose that one of the minimal extensions $A'_0a_0, ..., a_{\ell -t} \oleq A'_0a_1, ..., a_{\ell}$ is not algebraic, and let $\ell \in [0, m-1]$ be minimal with respect to this property, then (by \ref{assumptions:no_prime}(\hyperref[delta-rank]{C7})) $\delta(A'_0) = \delta(A'_0a_1, ..., a_{\ell-t})$ but $\delta(A'_0a_1, ..., a_{\ell}) > \delta(A'_0)$, on the other hand $\delta(A_i) = \delta(A_0) = \delta(A'_0)$, but this contradicts the fact that $<$ is a $\mrm{HF}$-order $<$ of $A_i$ over $A'_0$. It follows that for every $i \geq j$ we have that $A_i \subseteq  \mrm{acl}_A(A'_0) \subseteq \alpha(F(A))$.
\end{proof}

\begin{lemma}\label{el_and_delta} Assume that $A_0 \in \mathcal{K}$ and $B \models T^+$. If $B$ embeds elementarily in  $F(A_0)$, then there is a finite $B_0 \subseteq B$ such that $B \cong F(B_0)$ and $\delta(B_0) \leq \delta(A_0)$. 
\end{lemma}
\begin{proof} Let $<$ be the natural $\mrm{HF}$-order of $F(A_0)$ witnessing that $A := F(A_0)= \bigcup_{i < \omega} A_i$ and such that $\delta(A_i) = \delta(A_0)$ for all $i < \omega$. In particular $<$ is a well-order. Suppose that $B$ embeds elementarily in  $A$ and assume w.l.o.g. that this elementary embedding is simply the inclusion $B \subseteq A$. For each $i<\omega$ we let $B_i=A_i\cap B$. Since $B\preccurlyeq  A$,  it follows from \ref{el_sub_lemma} that $B\oleq A$. From this we deduce that $B_i\oleq A_i$ for all $i<\omega$ and so by \ref{assumptions:no_prime}(\hyperref[delta-rank]{C7}) we have $B_i\leq_\delta A_i$ for all $i<\omega$.  Since $\delta(A_0)$ is finite and $\delta(A_0)=\delta(A_i)$ for all $i<\omega$, there is an index $k<\omega$ such that $\delta(B_k)\geq\delta(B_j)$ for all $j> k$. Moreover, since we have that $A_k\oleq A_j$ (for $k< j$), we obtain also $B_k\oleq B_j$ and thus we conclude by \ref{open.vs.delta}(d) that $\delta(B_k)=\delta(B_j)$ for all $j> k$. Again by \ref{open.vs.delta}(c), it follows that  for all $j>k$ the elements in $B_{j}$ are algebraic over $B_{j-1}$. It follows that $B=\mrm{acl}_B(B_k)$ and so, since $B_k\oleq B$, we conclude by \ref{characterisation_acl} that $F(B_k)\cong \mrm{acl}_B(B_k)$. Since $\delta(B_k) \leq \delta(A_k)=\delta(A_0)$, this completes our proof.	
\end{proof}

\begin{theorem} \label{model_completeness_QE}
	Suppose $T^+$ satisfies \ref{main_th}(\hyperref[completeness_axiom]{C1})-(\hyperref[the_K_homogeneous_lemma]{C2}), \ref{ass:hf_closure}(\hyperref[the_hf_axiom]{C3}) and \ref{assumptions:no_prime}(\hyperref[F=C]{C5})-(\hyperref[delta-rank]{C7}). Then the following hold:
	\begin{enumerate}[(1)]
		\item $T^+$ does not have a prime model;
		\item $T^+$ is not model complete;
		\item $T^+$ does not eliminate quantifiers.
	\end{enumerate}
\end{theorem}
\begin{proof} 
	We first prove that $T^+$ does not have a prime model. Suppose towards contradiction that $B$ is a prime model of $T^+$. Then, by  Lemma \ref{el_and_delta}, it follows that $B = F(B_0)$ with $B_0\in \mathcal{K}$ and such that $\delta(B_0)$ is minimal among the non-degenerate configurations in $\mathcal{K}$. It then follows from  \ref{assumptions:no_prime}(\hyperref[hopf]{C6}) that we can choose copies $(A_i : i < \omega)$ of $B$ such that $A_i \subsetneq A_{i+1}$ and let $A = \bigcup_{i < \omega} A_i$.  Since the theory $T^+$ is a $\forall\exists$-theory (cf.~\ref{the_theory}), it follows that $A\models T^+$. Now, since by assumption $B$ is a prime model of $T^+$, there is an elementary embedding $f:B\to A$. Let  $C := f(B)$ and $C_0=f(B_0)$, then since $B_0$ is finite there is an index $i<\omega$ such that $B_0\subseteq A_i$. It then follows, since $A_i\models T^+$ and $C_0$ generates $C$, that $C\subseteq A_i$. Then, by \ref{el_sub_lemma} and the fact that $A_{i+1}\models T^+$, we have:
	$$\begin{array}{rcl}
		C \preccurlyeq A & \Longrightarrow & C \oleq A \\
		& \Longrightarrow & C \oleq A_{i + 1} \\
		& \Longrightarrow & C \preccurlyeq A_{i+1}.
	\end{array}$$
	But then $f: B \rightarrow A_{i+1}$ is non-surjective elementary embedding of $B$ into (an isomorphic copy of) itself, which contradicts Lemma \ref{the_surjective_lemma}.
	
	\smallskip
	\noindent We next prove that $T^+$ is not model complete and it does not eliminate quantifiers. Let $A\oleq B\in \mathcal{K}$ be such that $A$ and $B$ are non-degenerate (i.e., $F(A)\models T^+$ and $F(B)\models T^+$), and suppose $\delta(A)<\delta(B)$ (these exist because trivial extensions are always strong and increase the value of $\delta$). By \ref{assumptions:no_prime}(\hyperref[hopf]{C6}) there is an embedding $h:F(B)\to F(A) $. Notice that, if $F(B)\cong F(D)$, then by \ref{open.vs.delta}(c) we have that $\delta(B)=\delta(D)$. Therefore, by Lemma \ref{el_and_delta}, it follows that $h$ is not elementary, for otherwise we would have $\delta(B)\leq \delta(A)$. This shows that $T^+$ is not model complete and, as a consequence, it does not eliminate quantifies.
\end{proof}

\section{Applications}\label{sec:application}

In Section \ref{sec_general} we have introduced the class $(\mathcal{K},\oleq)$ and studied the properties of the theory $T^+$ under the assumptions from \ref{main_th}(\hyperref[completeness_axiom]{C1})-(\hyperref[the_K_homogeneous_lemma]{C2}), \ref{ass:hf_closure}(\hyperref[the_hf_axiom]{C3}), \ref{assumption:no-superstability_2}(\hyperref[CP]{C4}) and \ref{assumptions:no_prime}(\hyperref[F=C]{C5})-(\hyperref[delta-rank]{C7}). We now turn to concrete examples of geometries, and we apply our abstract machinery to several cases of open incidence structures. We point out that such examples of incidence geometries can be divided into two families.

On the one hand, we study in Section \ref{sec:n-open-grpahs} the $n$-open graphs introduced in \cite{baldwin_preprint}, of which the free pseudoplane is a major example. We verify in these cases Assumptions \ref{main_th}(\hyperref[completeness_axiom]{C1})-(\hyperref[the_K_homogeneous_lemma]{C2}) and \ref{ass:hf_closure}(\hyperref[the_hf_axiom]{C3}), and we derive the completeness and stability of the associated theory $T^+$ in this context. We stress that, crucially, in this case the algebraic completions $F(A)$ trivialise, as they are not in general models of the theory $T^+$.  Interestingly, for $n\geq 2$ the model-theoretic and the combinatorial notions of complexity diverge in this class of examples. While the fact that $F(A)$ is not necessarily a model of $T^+$ (even when $A$ has infinitely many elements of every sort) seems to indicate that the class in question is rather tame, it is still possible to exhibit configurations satisfying Assumption \ref{assumption:no-superstability_2}(\hyperref[CP]{C4}) and thus verify that $T^+_n$ is not superstable, where $T_n^+$ is the theory determined as in \ref{the_theory} by the class of $n$-open graphs. We stress that all these results concerning $n$-open graphs were already shown in \cite{baldwin_preprint}, but with different techniques.

On the other hand, in Sections \ref{sec_steiner}--\ref{last_section} we study all the examples of open incidence geometries from \cite{funk2}. From the model-theoretic point of view these cases behave quite similarly: they all verify the assumptions from \ref{main_th}(\hyperref[completeness_axiom]{C1})-(\hyperref[the_K_homogeneous_lemma]{C2}), \ref{ass:hf_closure}(\hyperref[the_hf_axiom]{C3}), \ref{assumption:no-superstability_2}(\hyperref[CP]{C4}) and \ref{assumptions:no_prime}(\hyperref[F=C]{C5})-(\hyperref[delta-rank]{C7}). In particular, in these cases we have that the free algebraic completion $F(A)$ is a model of $T^+$, provided that $A$ has enough elements of every sort. It thus follows from our results that the free completions of non-degenerate partial models in $\mathcal{K}$ are all elementarily equivalent. On the other hand, this family of examples also showcases several differences from the combinatorial point of view. Two issues are worth noting. Firstly, generalised $n$-gons make for the only case where minimal strong extensions are not necessarily one-element extensions -- we shall expand later on this issue in Section \ref{section:ngons}. Secondly, affine planes and projective M\"obius planes exhibit cases of incidence geometries with an additional relation of parallelism (or tangency) between elements of the same sort. As far as we know, the present work is the first model-theoretic study of incidence geometries which also involves such additional notions. However, here as in many other places, our treatment of the subject is largely indebted to the work of Funk and Strambach in \cite{funk2}, as they were the first to identify the technical notion of Gaifman closure (cf.~\ref{def:interior_closure}) that we use to deal with these difficulties. Finally, we point out again that the case of generalised $n$-gons was already considered in the literature (by \cite{paolini&hyttinen} for the case $n=3$ and by \cite{tent} for the case $n\geq 3$). The completeness result for $T^+$ in the case of $(2, 3)$-Steiner systems was recently showed by Alaimo in his thesis \cite{alaimo} and, independently, by Barbina and Casanovas in  \cite{barbina:2}, where they prove also the stability of the theory of open non-degenerate $(2,3)$-Steiner systems. All the other results from Sections \ref{sec_steiner}--\ref{last_section} are novel.

\subsection{$n$-open graphs}\label{sec:n-open-grpahs}

We consider the case of $n$-open graphs introduced by Baldwin, Freitag and Mutchnik in \cite{baldwin_preprint}.  In this section we let $L=\{ \vert \}$ be the language of graphs. For every $n<\omega$ we introduce a class of graphs $(\mathcal{K}_n,\specialonleq)$ and we show that it satisfies the conditions from \ref{context}, \ref{main_th}(\hyperref[completeness_axiom]{C1})-(\hyperref[the_K_homogeneous_lemma]{C2}) and \ref{ass:hf_closure}(\hyperref[the_hf_axiom]{C3}). As a consequence, it follows from our results in Section \ref{sec_general} that the theory $T^+_n$ defined from $(\mathcal{K}_n,\specialonleq)$ as in \ref{the_theory} is complete, stable and with $\ind=\ind^{\otimes}$. Moreover, for $n\geq 2$, we also prove that $T^+_n$ satisfies Assumption \ref{assumption:no-superstability_2}(\hyperref[CP]{C4}), and thus that $T^+_n$ is not superstable. We sum up these results in the following Theorem \ref{corollary:ngraphs}. We stress that all the results from this section are already known from the literature, but we believe they provide an interesting application of our general setting. In particular, our proof of the non-superstability of $T^+_n$ for $n\geq 2$ is significantly different from the one from \cite{baldwin_preprint}, and it highlights the combinatorial nature of the theories $T^+_n$.

\begin{theorem}\label{corollary:ngraphs}
	For every $n<\omega$, the theory $T^+_n$ is complete, stable, and with $\ind=\ind^{\otimes}$. Moreover, if $n\geq 2$, then $T^+_n$ is not superstable.
\end{theorem}

We start by recalling the definition of $n$-open graphs from  \cite{baldwin_preprint}. If $A$ is a graph then the we denote by $\vartheta_A(a)$ the  valency (cf.~\ref{def:graphs}) of an element $a\in A$. We notice that each class $(\mathcal{K}_n,\specialonleq)$ defined in \ref{pseudoplane:def.hyperfree}  determines a theory $T^+_n$ as in \ref{the_theory}. It is straightforward to verify that our theories $T^+_n$ coincide with the theories $T_n$ defined  in \cite{baldwin_preprint}. In particular, we remark that $T^+_0$ is the theory of an infinite set, while $T^+_1$ is the theory of the free pseudoplane.

\begin{definition}\label{pseudoplane:def.hyperfree}
	Let $A\subseteq B$ be graphs, then for every $ n<\omega$ we let $A\specialonleq B$ if  every non-empty $C \subseteq B \setminus A$ contains an element $c\in C$ with $\vartheta_{AC}(c)\leq n$. We then adopt the notions of \emph{open}, \emph{closed}, \emph{confined}, etc.~specialising Definition~\ref{def_extensions}. For every $ n<\omega$ we  let $\mathcal{K}_n$ be the class of all finite $n$-open graphs, i.e., those graphs $A$ such that $\emptyset\specialonleq A$, and we let $T_\forall^n= \mrm{Th}_\forall(\mathcal{K}_n)$.  We say \emph{$n$-open} instead of \emph{open} when we want to stress the parameter $n<\omega$ used to define the class $(\mathcal{K}_n,\specialonleq)$. When the parameter $n<\omega$ is clear from the context, we write simply $\onleq$ instead of $\specialonleq$.
\end{definition}

From now on, we fix some parameter $n<\omega$ and consider the associated class $(\mathcal{K}_n,\onleq)$ of $n$-open graphs. First, we start by verifying all the finitary conditions from Context \ref{context}. Conditions \ref{context}(\ref{condition:A})--(\ref{condition:F}) follow immediately from the definition of the relation $\onleq$. Condition \ref{context}(\ref{condition:I})  follows by letting $\mathbf{X}_{\mathcal{K}_n}=\{1\}$ and Condition \ref{context}(\ref{condition:J}) by letting $\mathbf{n}_{\mathcal{K}_n}=n$. Conditions \ref{context}(\ref{condition:K}) and \ref{context}(\ref{condition:L}) also follow directly from Definition \ref{pseudoplane:def.hyperfree}. It remains to verify \ref{context}(\ref{condition:G}) and \ref{context}(\ref{condition:H}).

\begin{proposition}[\ref{context}(\ref{condition:G})]\label{pseudoplane:conditionG}
	The class $(\mathcal{K}_n,\onleq)$ has the algebraic amalgamation property.
\end{proposition}
\begin{proof}
	By Definition~\ref{pseudoplane:def.hyperfree} the class $(\mathcal{K}_n,\onleq)$ has no algebraic extensions, whence it suffices to show that $\mathcal{K}_n$ is closed under free amalgamation in the following sense. We show that if $A,B,C\in \mathcal{K}_n$ with $A\onleq B$, $A\leq C$ and $B\cap C=A$ then $B\otimes_A C\in \mathcal{K}_n$. By the definition of the free amalgam $D\coloneqq B\otimes_A C$ (recall \ref{general_free_amalgam}), we have that $\vartheta_D(b)=\vartheta_B(b)$ for every $b\in B$ and $\vartheta_D(c)=\vartheta_C(c)$ for every $c\in C$. Therefore, since $A,B,C\in \mathcal{K}_n$, it follows from Definition~\ref{pseudoplane:def.hyperfree} that every subset of $B\otimes_A C$ contains an element with valency $\leq n$, and so $B\otimes_A C\in \mathcal{K}_n$. This shows that $(\mathcal{K}_n,\onleq)$ satisfies the conditions from Definition~\ref{def:sharp_amalgamation}.
\end{proof}

We next verify Condition \ref{context}(\ref{condition:H}). Recall that $\mrm{gcl}_M(A)$ is the Gaifman closure from \ref{def:interior_closure}, and thus in the setting of graphs we have that $\mrm{gcl}_M(A)=\mrm{gcl}^{\mrm{i}}_M(A)$, as there is no local equivalence relation in $L$.

\begin{proposition}[\ref{context}(\ref{condition:H})]\label{pseudoplane:conditionH}
	If $A\subseteq B \in \mathcal{K}_n$ and $A \nonleq B$ is confined, then  every $C \in \mathcal{K}_n$ contains at most $n$ many disjoint copies of $B$ over $A$.
\end{proposition}
\begin{proof}
	Suppose $A \nonleq B$ is confined, then every element in $B\setminus A$ has valency at least $n+1$ in $B$. Suppose $C\in \mathcal{K}_n$ contains at least $n+1$ disjoint copies $(B_i)_{0\leq i\leq n}$ of $B$ over $A$, then we consider the set $D\coloneqq \bigcup_{0\leq i\leq n}\mrm{gcl}_C(B_i\setminus A)$. Now, if $b\in B_i\setminus A$ then by assumption it is incident to at least $n+1$ many elements from $\mrm{gcl}_C(B_i\setminus A)$. Moreover, by construction and the fact that $\mrm{gcl}_C(B_i\setminus A)=\mrm{gcl}^{\mrm{i}}_C(B_i\setminus A)$, we also have that every element $a\in\mrm{gcl}_C(B_i\setminus A)\cap A$ is incident to at least $n+1$ many elements from $D$, as it is incident to at least one element for each $B_i$. It follows that  every $d\in D$ has valency $\vartheta_D(d)\geq n+1$, contradicting  $C\in \mathcal{K}_n$.
\end{proof}	

It thus follows from Propositions \ref{pseudoplane:conditionG}-\ref{pseudoplane:conditionH} and the previous observations, that $(\mathcal{K}_n, \onleq)$ satisfies the conditions from \ref{context}. We next prove that it also verifies the assumptions from \ref{main_th}(\hyperref[completeness_axiom]{C1})-(\hyperref[the_K_homogeneous_lemma]{C2}) and \ref{ass:hf_closure}(\hyperref[the_hf_axiom]{C3}). We first establish \ref{ass:hf_closure}(\hyperref[the_hf_axiom]{C3}) and then verify \ref{technical_assumptions}(\hyperref[trivial_condition]{D1})-(\hyperref[minimality_condition]{D2}), from which \ref{main_th}(\hyperref[completeness_axiom]{C1}) follows by Proposition \ref{the_K_saturated_lemma}. We notice that, since minimal extensions in $n$-open graphs are always one-element extensions, and their language contains no symbol for local equivalence relations, the closure operator $\mrm{cl}_<$ can be defined purely in terms of the incidence Gaifman closure. We verify in Proposition \ref{psedoplane:HF_orders} that this is in fact a $\mrm{HF}$-closure operator.

\begin{definition}\label{pseudoplane:closure_operator} Let $A\subseteq B \models T_\forall^n$ and suppose $<$ is a $\mrm{HF}$-order of $B$ over $A$.  For all $C\subseteq B\setminus A$ we define $\mrm{cl}_<(C)\coloneqq \bigcup_{n<\omega }\mrm{cl}_<^{n}(C)$ inductively by letting $\mrm{cl}_<^0(C) \coloneqq C$ and $\mrm{cl}_<^{n+1}(C) \coloneqq \mrm{gcl}^{\mrm{i}}_{A\mrm{cl}_<^{n}(C)^\downarrow}(\mrm{cl}_<^{n}(C))$ for all $n<\omega$.
\end{definition}

\begin{proposition}[{\ref{ass:hf_closure}(\hyperref[the_hf_axiom]{C3})}]\label{psedoplane:HF_orders}
	Every model $A\subseteq B\models T^n_\forall$ with an associated $\mrm{HF}$-order $<$ of $B$ over $A$ has a $\mrm{HF}$-closure operator $\mrm{cl}_<\coloneqq\bigcup_{n<\omega}\mrm{cl}^n_<$.
\end{proposition}
\begin{proof}
	Let $<$ be a $\mrm{HF}$-order of $B$ over $A$ and let $\mrm{cl}_<(C)\coloneqq \bigcup_{n<\omega }\mrm{cl}_<^{n}(C)$ be the operator defined in \ref{pseudoplane:closure_operator}. Clearly  $\mrm{cl}_<(C)\coloneqq \bigcup_{n<\omega }\mrm{cl}_<^{n}(C)$ satisfies conditions (1)-(2) from Definition~\ref{def_closure_operator}, since every element occurring in a $\mrm{HF}$-order is incident to at most $n$ many elements before it in the ordering. It thus suffices to verify that 
	\[A\mrm{cl}_<(CD)=A\mrm{cl}_<(C)\otimes_{A\mrm{cl}_<(C)\cap A\mrm{cl}_<(D)} A\mrm{cl}_<(D)\oleq B\]
	holds for all $C\subseteq B\setminus A$.

	\medskip
	\noindent First, notice that if there is an edge between some $c\in \mrm{cl}_<(C)$ and $d\in \mrm{cl}_<(D)$, then since $<$ is a linear order it follows from Definition~\ref{pseudoplane:closure_operator} that either $c\in \mrm{cl}_<(D)$ or $d\in \mrm{cl}_<(C)$, which entails that $A\mrm{cl}_<(CD)=A\mrm{cl}_<(C)\otimes_{A\mrm{cl}_<(C)\cap A\mrm{cl}_<(D)} A\mrm{cl}_<(D)$. 
	
	\medskip
	\noindent It remains to show that $A\mrm{cl}_<(CD)\oleq B$. By Corollary \ref{equivalence_HFo.ordering} it suffices to prove that $<\restriction B\setminus(A\mrm{cl}_<(CD)) $ is a $\mrm{HF}$-order of $B$ over $A\mrm{cl}_<(CD)$. Condition \ref{def_HF_order}(\hyperref[HF1]{H1}) is immediate to verify. We show that  Condition \ref{def_HF_order}(\hyperref[HF2]{H2}) holds as well. 
	
	\smallskip
	\noindent Let $b\in B\setminus A\mrm{cl}_<(CD)$ and suppose towards contradiction that $A\mrm{cl}_<(CD)E\noleq A\mrm{cl}_<(CD)Eb$ for some finite $E\subseteq b^\downarrow \setminus \{b\}$. By Definition \ref{infinite_strong extensions}, there is a finite subset $C_0\subseteq \mrm{cl}_<(CD)$ such that $AC_0E\noleq AC_0Eb$. Additionally, we can also assume that $C_0$ is minimal with this property, i.e., for all $C_1\subsetneq C_0$ it holds that $AC_1E\oleq AC_1Eb$. Recall the notion of distance from \ref{gaifman_graph}, we then distinguish the following cases.
	
	\smallskip 
	\noindent \underline{Case 1}. $c<b$ for all $c\in C_0$.
	\newline Then $AC_0E\noleq AC_0Eb$ contradicts the fact that $<$ is an $\mrm{HF}$-order of $B$ over $A$.
	
	\smallskip 
	\noindent \underline{Case 2}. $b<c$ for some $c\in C_0$ and  $d(c,b/A\mrm{cl}_<(CD)b^\downarrow)>1$.
	\newline Let $C_1=C_0\setminus \{c \}$. Since there is no relation between $c $ and $b $, it follows that $ b$ has the same Gaifman closure in $AC_0E$ and $AC_1E$, thus by  Condition \ref{context}(\ref{condition:K}) it follows that $AC_1E\noleq AC_1Eb$, contradicting the minimality of $C_0$.
	
	\smallskip 
	\noindent \underline{Case 3}. $b<c$ for some $c\in C_0$ and  $d(c,b/A\mrm{cl}_<(CD)b^\downarrow)= 1$.
	\newline Then $b\in \mrm{gcl}_{A\mrm{cl}_<(CD)}(c)$ and so, by the definition of $\mrm{HF}$-closure, it follows that $ b\in \mrm{cl}_<(CD)$, contradicting our choice of $b$.
	
	\smallskip 
	\noindent This shows that $<\restriction B\setminus(A\mrm{cl}_<(CD)) $ satisfies also Condition \ref{def_HF_order}(\hyperref[HF2]{H2}) and thus it is a $\mrm{HF}$-order of $B$ over $A\mrm{cl}_<(CD)$, completing our proof.
\end{proof}

The previous proposition makes us sure that there is always a well-defined $\mrm{HF}$-closure operator $\mrm{cl}_<\coloneqq\bigcup_{n<\omega}\mrm{cl}^n_<$ in models of $T^n_\forall$. We employ this to verify Assumptions \ref{technical_assumptions}(\hyperref[trivial_condition]{D1})-(\hyperref[minimality_condition]{D2}), from which \ref{main_th}(\hyperref[completeness_axiom]{C1}) also follows.

\begin{proposition}[{\ref{technical_assumptions}(\hyperref[trivial_condition]{D1})}]\label{pseudoplane:trivial_condition1}
	If $M \models T^+_n$, $A \subseteq M$ is finite, $<$ is a $\mrm{HF}$-order of $M$ and $A \onleq Ac \in \mathcal{K}_n$ is a trivial extension, then for all $\ell < \omega$ there exists $ c'_\ell \in M$ s.t.:
	\begin{enumerate}[(a)]
		\item $Ac \cong_A Ac'$;
		\item $Ac' \onleq A\mrm{cl}_<^\ell(c')$.
	\end{enumerate}
\end{proposition}
\begin{proof} By induction on $\ell< \omega$ we prove that for every $A \onleq M$ and $A \onleq Ac \in \mathcal{K}_n$ we can find $c'_\ell$ as in (a)-(b). The case for $n=0$ is easy: it is enough to pick an element $c'_0>A$ which has no incidence with elements from $A$. Clearly, we can find such an element in any model $M\models T^+_n$ with a $\mrm{HF}$-order $<$.
	\newline Thus suppose the inductive hypothesis holds with respect to $\ell$,  we define $c'_{\ell+1}$ so that (a)-(b) are satisfied for it. By assumption, and proceeding recursively,  we can find $n$ many elements $d_1,\dots,d_n$ such that, for each $0\leq i<n$, we have that the element $d_{i+1}$ has no incidence with elements from $A\mrm{cl}_<^\ell(\{d_1,\dots,d_i  \})$ and 
	\[ A\mrm{cl}_<^\ell(\{d_1,\dots,d_i  \})d_{i+1}\onleq A\mrm{cl}_<^\ell(\{d_1,\dots,d_i,d_{i+1}  \}).  \]
	Now, notice that the extension $d_1\dots d_n\onleq d_1\dots d_n b$ with $b$ incident to $d_1,\dots,d_n$ is not algebraic. Since $M\models T^+_n$ there are infinitely many elements in $M$ isomorphic to $b$ over $d_1,\dots,d_n$. In particular, since $<$ is a $\mrm{HF}$-order, there is some $c'_{\ell+1}\in M $ such that $d_i<c'_{\ell+1}$ for all $1\leq i\leq n $. It follows that $c'_{\ell+1}$ is incident only to $d_1,\dots,d_n$ from $(c'_{\ell+1})^\downarrow$. Thus clearly we have that  $Ac \cong_A Ac'_{\ell+1}$. 
	\newline We next claim that $Ac'_{\ell+1} \onleq A\mrm{cl}_<^{\ell+1}(c'_{\ell+1})$. To see this, notice that  we can define a $\mrm{HF}$-order $<'$ by first letting $Ac'_{\ell+1}<'d_{1}<'\dots<' d_n$. By construction, $c'_{\ell+1}$ is incident only to $d_1,\dots,d_n$ in $\mrm{cl}_<^{\ell+1}(c'_{\ell+1})$, and, for $1\leq i<j\leq n$, no element in $\mrm{cl}_<^{\ell}(d_{i})$ is incident to elements from $\mrm{cl}_<^{\ell}(d_{j})$. It follows that we can extend the order $<'$ and obtain $Ac'_{\ell+1}\hleq Ac'_{\ell+1} \mrm{cl}_<^\ell(\{d_1,\dots,d_i,d_{i+1}  \})$. Since by construction we have that $ \mrm{cl}_<^{\ell+1}(c'_{\ell+1})=\mrm{cl}_<^\ell(\{d_1,\dots,d_n  \})\cup \{c'_{\ell+1} \}$ we obtain that $Ac'_{\ell+1} \onleq A\mrm{cl}_<^{\ell+1}(c'_{\ell+1})$. This concludes our proof.
\end{proof}

\begin{lemma}\label{pseudoplane:algebraic_lemma}
	Let $A\onleq C\models T^n_\forall$ and $c\in C$ be such that  $\vartheta_{Ac}(c)=n$, then $Ac\onleq C$. 
\end{lemma}
\begin{proof}
	Consider any finite subset $D\subseteq C$ such that $Ac\subsetneq D$, we need to show that $D\setminus Ac$ contains some element with valency $\leq n$.  Now, since $A\onleq C$ and $D\subseteq C$, there is an element $d\in D\setminus A$ with $\vartheta_{AD}(d)\leq n$. If $c\neq d$ then we have $d\in D\setminus Ac$ and $\vartheta_{AD}(d)\leq n$, and we are done. Thus suppose that $c$ is the unique element in $D\setminus A$  with valency $\leq n$. Since $\vartheta_{Ac}(c)=n$, it follows that $\vartheta_{AD}(c)=n$, and  in particular it follows that $c$ is not incident to any element from $D\setminus A$. Then let $D'=D\setminus\{c\}$ and consider an element $d'\in D'\setminus A$ with $\vartheta_{AD'}(d')\leq n$ -- this exists since $A\onleq C$ and $Ac\subsetneq D$. Then $c$ and $d'$ are not incident, whence $\vartheta_{AD}(d')\leq n$. This shows that every finite subset $D\subseteq C$ such that $Ac\subsetneq D$ contains an element with valency $\leq n$, and thus proves that $Ac\onleq C$. 
\end{proof}

\begin{proposition}[{\ref{technical_assumptions}(\hyperref[minimality_condition]{D2})}]\label{pseudoplane:minimality_condition}
	Suppose $M \models T_n^+$ is $\aleph_1$-saturated and $A\onleq M$ is countable. If for every trivial one-element extension $A\onleq Ab$ there is $b' \in M$ such that $Ab \cong_A Ab'\onleq M$, then for every minimal extension $A\onleq Ac\models T_\forall$ there is an element $c' \in M$ such that $Ac \cong_A Ac'\onleq M$.
\end{proposition}
\begin{proof}
	Let $M\models T_n^+$ be $\aleph_1$-saturated  and $A\onleq M$ be countable. Suppose $A\onleq Ac\in \mathcal{K}_n$ is a minimal strong extension with $c$  incident to elements $a_1,\dots,a_\ell$ from $A$, for $1\leq \ell\leq n$. By assumption, we can find elements $b_{\ell+1},\dots, b_{n}$ in $M\setminus A$ such that each $b_{\ell+k+1}$ is trivial over $Ab_{\ell+1}\dots b_{\ell+k}$ and $Ab_{\ell+1}\dots b_{n}\onleq M$. Then, since $M\models T^+_n$, it follows that there are infinitely many elements which are incident to $a_1,\dots,a_\ell,b_{\ell+1},\dots,b_{n}$. In particular, it follows that the type $p(x,Ab_{\ell+1}\dots b_{n})$ which says that $x$ is incident to $a_1,\dots,a_\ell,b_{\ell+1},\dots,b_{n}$ and $x\neq a$ for all $a\in A$ is consistent. It follows from the fact that $M$ is $\aleph_1$-saturated that there is some $c'\in M\setminus A$ incident to $a_1,\dots,a_\ell,b_{\ell+1},\dots,b_{n}$. Then, since $Ab_{\ell+1}\dots b_{n}\onleq M$, we obtain that $c'$ is incident only to these elements from $Ab_{\ell+1}\dots b_{n}$. It follows in particular that $Ac\cong_A Ac'$. Moreover, since $\vartheta_{Ab_{\ell+1},\dots,b_{n}}(c')=n$, we obtain from Lemma \ref{pseudoplane:algebraic_lemma} that $Ab_{\ell+1}\dots b_{n} c'\onleq M$. Since $b_{\ell+1},\dots,b_{n}$ are trivial extensions of $A$, we conclude from \ref{context}(\ref{condition:K})  that $Ac' \onleq  Ab_{\ell+1}\dots b_{n} c'\onleq M$.
\end{proof}

It follows from our previous results that $T^+_n$ satisfies Assumptions \ref{main_th}(\hyperref[completeness_axiom]{C1}) and \ref{ass:hf_closure}(\hyperref[the_hf_axiom]{C3}). To establish also \ref{main_th}(\hyperref[the_K_homogeneous_lemma]{C2}) we use Proposition~\ref{lemma:homogeneity} and verify the technical condition \ref{technical_assumptions}(\hyperref[extension]{D3}). To this end, we introduce a \emph{free completion process} for $n$-open graphs, which was already considered in \cite{baldwin_preprint}. As noticed by the authors, this free construction can be considered as an adaptation of Hall's construction of a free projective plane in the context of $n$-open graphs. However, we stress that, while in Hall's case the free completion is algebraic over the starting configuration, this is not the case in the setting of $n$-open graphs. We reflect this in our notation and we use $C(A)$ for the abstract notion  of completion used in this context, while we reserve $F(A)$ for the notion of algebraic completion defined in \ref{free_algebraic_completion}. The fact that $C(A)$ behaves like a free algebraic completion essentially follows from Lemma~\ref{pseudoplane:algebraic_lemma}, which in particular shows that one could define $C(A)$ similarly as we did in \ref{free_algebraic_completion} but considering at each step of the construction minimal extensions $A\oleq Ab$ where $b$ has valency exactly $n$. Since our interest in this paper is mostly in geometries with a free \emph{algebraic} completion we do not pursue this idea here, and we simply verify Condition \ref{technical_assumptions}(\hyperref[extension]{D3})

\begin{definition}	\label{free_pseudoplane_def}
	Let $A\models T^n_\forall $. We let $A_0\coloneqq A$ and, for every $i<\omega$, we let $A_{i+1}$ be obtained from $A_i$ by adding, for every distinct $c_1,\dots,c_n\in  A_i$,  infinitely many elements which are incident to $c_1,\dots,c_n$ and to no other element from $A_i$. We then define $C(A)\coloneqq \bigcup_{i<\omega}A_i$ and we say that $C(A)$ is the \emph{free completion} of $A$.
\end{definition}

\begin{proposition}[{\ref{technical_assumptions}(\hyperref[extension]{D3})}]\label{free_theory_pseudoplane}
	Let $A\models T^n_\forall$ and suppose that $A$ contains at least $n$ many elements. Then $C(A)\models T_n^+$, $A\onleq C(A)$, $|C(A)|=|A|+\aleph_0$ and:
	\begin{enumerate}[(a)]
			\item $C(A)$ is unique up to $A$-isomorphisms, i.e., every isomorphism $h:A\cong B$ extends to an isomorphism $\hat{h}:F(A)\cong F(B)$;
			\item $C(A)$ is $\oleq$-prime over $A$, i.e., for every $M\models T^+$ and $h:A\to M$ with $h(A)\oleq M$ there is $\hat{h}: C(A)\to M$ s.t. $\hat{h}\restriction A = h$ and $\hat{h}(C(A))\oleq M$.
		\end{enumerate}
	In particular, it follows that if $A\cong B\onleq M\models T_n^+$  and $M$ is $\aleph_1$-saturated,  then $C(A)\cong C(B)\preccurlyeq M$.
\end{proposition}
\begin{proof}
	Since it is immediate to verify that the construction of $C(A)$ induces a $\mrm{HF}$-order of $C(A)$ over $A$, it follows that $A\onleq C(A)$. Similarly, it follows from Definition \ref{free_pseudoplane_def} that  $|C(A)|=|A|+\aleph_0$ and also that any isomorphism $h:A\cong B$ extends to an isomorphism $\hat{h}:C(A)\cong C(B)$.
		
	\smallskip
	\noindent We show that, if $A\models T^n_\forall$ contains at least $n$ many elements, then $C(A)\models T_n^+$. Notice that, if $|A|\geq n$ then $C(A)$ is infinite. We verify that $C(A)$ satisfies the axioms from \ref{the_theory}. Since in $\mathcal{K}_n$ there are no algebraic extensions, it suffices to verify \ref{the_theory}(1) and \ref{the_theory}(3). By $A\onleq C(A)$ it is immediate to verify that $C(A)\models T^n_\forall$. Consider then some $D\subseteq_\omega C(A)$ and a minimal extension $D\onleq Dc\in \mathcal{K}_n$. Then $c$ is incident exactly to some elements $d_1,\dots, d_\ell\in D$ for some $\ell\leq n$. Pick some new elements $d_{\ell+1},\dots,d_{n}\in C(A)\setminus D$ and let $A_i$ be the stage  with the smallest index in the construction of $C(A)$ satisfying $\{d_1,\dots,d_\ell,d_{\ell+1},\dots,d_{n} \}\subseteq  A_i$. By construction $A_{i+1}$ contains infinitely many elements $(c_i)_{i<\omega}$ incident exactly to $d_1,\dots,d_\ell,d_{\ell+1},\dots,d_{n}$ in $D\cup \{ d_{\ell+1},\dots,d_{n}  \}$. For every $i<\omega$ we have that $Dc\cong_D Dc_i\subseteq C(A)$, verifying \ref{the_theory}(3) and thus showing that $C(A)\models T^+_n$.
	
	\smallskip
	\noindent Now, suppose that $M\models T^+$ and that $h:A\to M$ is an embedding with $h(A)\onleq M$. We suppose inductively that there is $h_i\supseteq h$ with $h_i:A_i\to M$ and  $h_i(A_i)\onleq M$. Let $c\in A_{i+1}\setminus A_i$, then $c$ is incident to exactly $n$ many elements $d_1,\dots, d_n \in A_i$. Since $M\models T^+_n$ there is an element $c'\in M\setminus h_i(A_i)$ which is incident to $h_i(d_1),\dots, h_i(d_n)$. Since $h_i(A_i)\onleq M$ it follows that $c'$ is not incident to any other element from $h_i(A_i)$. Let $h_{i+1}\coloneqq h_i\cup \{(c,c')  : c\in A_{i+1}\setminus A_i  \}$ where the element $c'$ is chosen as above. It follows immediately from the choice of such elements that $h_{i+1}:A_{i+1}\to M$ is an embedding. Moreover, since every element in $h_{i+1}(A_{i+1})\setminus h_{i+1}(A_{i})$  has valency $n$ in $h_{i+1}(A_{i+1})$, it follows from Lemma \ref{pseudoplane:algebraic_lemma} that $h(A_{i+1})\onleq M$. Finally, letting $\hat{h}=\bigcup_{i<\omega}h_i$, we obtain an embedding $\hat{h}:C(A)\to M$ satisfying $\hat{h}\restriction A=h$ and $\hat{h}(C(A))\onleq M$.
	
	\smallskip
	\noindent Finally, the fact that $A\cong B\onleq M\models T_n^+$ with $M$ $\aleph_1$-saturated  entails $C(A)\cong  C(B)\onleq M$ follows immediately from items (a) and (b). Since we also showed that $C(A),C(B)\models T_n^+$, it then follows from Theorem~\ref{prop_completeness}  that  $C(A)\cong  C(B)\preccurlyeq M$, verifying \ref{technical_assumptions}(\hyperref[extension]{D3}) before  (notice in particular that \ref{prop_completeness} holds as we already established \ref{main_th}(\hyperref[completeness_axiom]{C1}).
\end{proof}

With the previous result, we verified all assumptions from \ref{main_th}(\hyperref[completeness_axiom]{C1})-(\hyperref[the_K_homogeneous_lemma]{C2}) and \ref{ass:hf_closure}(\hyperref[the_hf_axiom]{C3}), thus establishing that $T^+_n$ is complete, stable, and with $\ind=\ind^{\otimes}$. Since every strong extension $A\onleq B$ is not algebraic it follows that the conditions from \ref{assumptions:no_prime} do not directly apply to this class of examples. It thus remains to consider \ref{assumption:no-superstability_2}(\hyperref[CP]{C4}), and to determine the stability spectrum of the theory $T^+_n$. Since $T^+_0$ and $T^+_1$ are respectively the theory of the infinite set and of the pseudoplane, they are both $\aleph_0$-stable and thus do not satisfy \ref{assumption:no-superstability_2}(\hyperref[CP]{C4}). We show however that for all $n\geq 2$ the theory $T^+_n$ satisfies \ref{assumption:no-superstability}(\hyperref[technical:D4]{D4}), and is thus not superstable by Proposition \ref{prop:technical_superstability} and Theorem~\ref{failure:superstability}. This completes the proof of Theorem~\ref{corollary:ngraphs} above.

\begin{proposition}[{\ref{assumption:no-superstability}(\hyperref[technical:D4]{D4})}]
	Let $n\geq 2$, then there is a configuration $C\in \mathcal{K}_n$ satisfying the conditions from \ref{assumption:no-superstability}.
\end{proposition}
\begin{proof}
	Fix some $n\geq 2$, then we construct a configuration  $C\in \mathcal{K}_n$ which satisfies the conditions from \ref{assumption:no-superstability}. Consider the configuration $C=\{c_0,c_1\}\cup\{ c_i : 2\leq i\leq n+2\} $ defined as follows: (1) for all $2\leq i\leq n$, $c_i$ is incident to $c_j$ for all $j<i$; (2) $c_{n+1}$ is incident to $c_i$ for $i<n$; (3) $c_{n+2}$ is incident to $c_0$, $c_1$, $c_n$ and $c_{n+1}$.
	
	\smallskip
	\noindent Clearly $C$ is a graph. Moreover, it is straightforward to verify that the order obtained by letting $c_i<'c_j$ if $i<j$ is a $\mrm{HF}$-order, thus we have that $C\in \mathcal{K}_n$. We verify conditions (a)-(c) from \ref{assumption:no-superstability}.
	
	\smallskip
	\noindent Firstly, we clearly have from the fact that $<'$ is a $\mrm{HF}$-order that $c_0,c_1\onleq C$. By definition we have that there is no edge between $c_0$ and $c_1$. Thus \ref{assumption:no-superstability}(a) holds.
	
	\smallskip
	\noindent Now, notice that the elements $c_0,c_1$ are incident to $c_{2},c_{3},\dots, c_{n+1},c_{n+2}$, and for every $2\leq i\leq n+1$, the element $c_{i}$ is incident to all elements $c_j$ with $0\leq j\leq n+1$ and $j\neq i$.  Hence every element in $C\setminus \{ c_{n+2} \}$ has valency $n+1$ in $C\setminus \{ c_{n+2} \}$, and so  $c_{n+2}\nonleq C$ is a confined extension. Also $L$ has only one sort, thus all elements in $C$ have the same sort.  This shows that \ref{assumption:no-superstability}(b) also holds.
	
	\smallskip
	\noindent Finally, one can continue this construction by taking $c_{n+2+1},\dots, c_{2n+4}$ such that the configuration $C'=\{c_{n+2},c_{n+2+1},\dots, c_{2n+4}  \}$ is isomorphic to $C$. Then the resulting configuration $C\cup C'$ is still a model in $\mathcal{K}_n$. Since this construction can be continued to obtain a sequence $(c_i)_{i<\omega}$, it follows that also \ref{assumption:no-superstability}(c) is verified.
\end{proof}

\subsection{$(k, n)$-Steiner systems}\label{sec_steiner}

In this section we start to study the incidence geometries considered in \cite{funk2}. We first consider $(k, n)$-Steiner systems, which are an example of incidence geometry consisting of points and blocks. Given the recent interest in the model theory of Steiner systems (as witnessed by the references \cite{baldwin_steiner:1,baldwin_steiner:2,baldwin_steiner:3,barbina:1,barbina:2, horsley}), we provide details for most proofs, while in subsequent sections we will omit the details of those proofs that follow as in the case of Steiner systems. We thus proceed as follows. We introduce in Section \ref{Steiner:definitions} the class $\mathcal{K}$ of finite, open, $(k, n)$-Steiner systems and we define a relation $A\oleq B$ in this context. Then, in Section \ref{steiner:finitary_conditions} we verify the finitary conditions \ref{context}(\ref{condition:A})--\ref{context}(\ref{condition:L}). In Section \ref{steiner:completeness_section}, we verify the assumptions from \ref{main_th}(\hyperref[completeness_axiom]{C1})-(\hyperref[the_K_homogeneous_lemma]{C2}) and \ref{ass:hf_closure}(\hyperref[the_hf_axiom]{C3}), and additionally we show that if $A$ is a non-degenerate structure in $\mathcal{K}$ its free completion $F(A)$ is a model of $T^+$. Finally, in Section \ref{steiner:superstability_prime} we show that \ref{assumption:no-superstability_2}(\hyperref[CP]{C4}) and \ref{assumptions:no_prime}(\hyperref[F=C]{C5})-(\hyperref[delta-rank]{C7}), also hold. The following theorem is a consequence of the results obtained in this section together with those from Section~\ref{sec_general}.

\begin{theorem}\label{corollary:steiner}
	Let $(\mathcal{K},\oleq)$ be the class of finite, open, partial $(k, n)$-Steiner systems and let $T^+$ be defined as in \ref{the_theory}. Then $T^+$ is complete, stable, and with $\ind=\ind^{\otimes}$. Moreover, $T^+$ is not superstable, it does not have a prime model, it is not model complete and it does not eliminate quantifiers.
\end{theorem}

\subsubsection{Definition of $(\mathcal{K},\oleq)$}\label{Steiner:definitions}

We recall in this section the definition of $(k, n)$-Steiner systems from \cite[p.~748]{funk2} and we introduce the class $\mathcal{K} $ of partial, finite, open, $(k, n)$-Steiner systems.

\begin{definition}
	Let $L$ be the two-sorted language where the symbols $p_0,p_1,\dots$ denote \emph{points}, the symbols $b_0,b_1,\dots$ denote \emph{blocks} and the symbol $p\, \vert \, b$ means that the point $p$ is incident with the block $b$. For any $2\leq k<n$ the theory of \emph{$(k, n)$-Steiner systems} consists of the following axioms:
	\begin{enumerate}[(1)]
		\item given $k$ many distinct points $p_0,\dots,p_{k-1}$ there is a unique block $b$ such that $p_i\, \vert \, b$ for all $i<k$;
		\item for every block $b$ there are exactly $n$ many points $p_0,\dots,p_{n-1}$ such that $p_i\, \vert \, b$ for all $i<n$. 
	\end{enumerate}
	The universal fragment of the theory of $(k, n)$-Steiner systems is axiomatised exactly by the formulas saying that each block is incident to at most $n$ many points and that, whenever both $\bigwedge_{i<k} p_i\, \vert\,b$ and $\bigwedge_{i<k} p_i\, \vert\,b'$ hold, then $b=b'$. We refer to models of this universal theory as \emph{partial  $(k, n)$-Steiner systems}. 
\end{definition}

\begin{definition}\label{Steiner:def_hyperfree}
	For finite partial $(k, n)$-Steiner systems $A\subseteq B$, we say that an element $a\in A$ is \emph{hyperfree} in $B$ if it satisfies one of the following conditions:
	\begin{enumerate}[(i)]
		\item $a$ is a point incident with at most one block from $B$;
		\item $a$ is a block incident with at most $k$ many points from $B$.
	\end{enumerate}
	We then write $A \oleq B$ if every non-empty $C \subseteq B \setminus A$ contains an element which is hyperfree in $AC$. We then adopt the notions of \emph{open}, \emph{closed}, \emph{confined}, etc. specialising Definition~\ref{def_extensions}. We let $\mathcal{K}$ be the class of finite, open, partial $(k, n)$-Steiner systems and we let $T_\forall$ be the universal theory of $\mathcal{K}$. The previous  definitions immediately extend to infinite models of $T_\forall$ as in \ref{infinite_strong extensions}. The notion of hyperfree element in a partial $(k, n)$-Steiner system is exactly as in \cite[p.~753]{funk2}.
\end{definition} 

We remark what  the algebraic strong extensions in the context of $(k, n)$-Steiner systems exactly are. We then prove an important lemma, which essentially establishes the part \ref{condition:amalgam}(2) of the algebraic amalgamation property for $(\mathcal{K},\oleq)$. 

\begin{remark}\label{Steiner:algebraic_extension}
	In the case of $(k, n)$-Steiner systems, a minimal strong extensions $A\oleq Ac$ is algebraic if $c$ is a point incident with some block from $A$, or if $c$ is a block incident with exactly $k$ many points from $A$.
\end{remark}

\begin{lemma}\label{Steiner:algebraic_lemma}
	Let $A\oleq C\models T_\forall$ and let $c\in C$ be such that $A\oleq Ac$ is an algebraic strong extension. Then $Ac\oleq C$. 
\end{lemma}
\begin{proof}
	Consider any finite subset $D\subseteq C$ such that $Ac\subsetneq D$, we need to show that $D\setminus Ac$ contains some element hyperfree in $D$. Now, since $A\oleq C$ and $D\subseteq C$, there is an element $d\in D\setminus A$ hyperfree in $D$. If $c\neq d$ then $d$ also witnesses that $D\setminus Ac$ contains an element hyperfree in $D$, and we are done. Thus suppose that $c$ is the unique element in $D\setminus A$  hyperfree in $D$. 
	\newline  Notice that, since the strong extension $A\oleq Ac$ is algebraic, $c$ is either a point incident to one block in $A$ or a block incident to $k$ many points in $A$. In both cases, $c$ is not incident to any element from $D\setminus A$. Then let $D'=D\setminus\{c\}$ and consider an element $d'\in D'\setminus A$ hyperfree in $D'$ -- this exists since $A\oleq C$ and $Ac\subsetneq D$. Since $c$ is incident only to elements from $A$, it follows that $c$ and $d'$ are not incident, and therefore  $d'$ is also hyperfree  in $D$. This shows that every finite subset $D\subseteq C$ with $Ac\subsetneq D$ is such that $D\setminus Ac$ contains an element hyperfree in $D$.
\end{proof}

\subsubsection{Conditions \ref{context}}\label{steiner:finitary_conditions}

We now start considering the conditions from \ref{context}. Conditions (\ref{condition:A}), (\ref{condition:B}), (\ref{condition:D}), (\ref{condition:E}),  (\ref{condition:F}), (\ref{condition:K}) and (\ref{condition:L}) follow immediately from the definition of the relation $\oleq$ for finite, open, partial $(k, n)$-Steiner systems. By Definition \ref{Steiner:def_hyperfree} it is clear that Condition  \ref{context}(\ref{condition:I}) holds with $\mathbf{X}_\mathcal{K}=\{1\}$, i.e., the minimal extensions are exactly the one-element extensions. It remains to verify conditions  (\ref{condition:C}), (\ref{condition:G}), (\ref{condition:H}) and (\ref{condition:J}), which we consider separately in the following propositions. 

\begin{proposition}[\ref{context}(\ref{condition:C})] If $A, B, C \in \mathcal{K}$ and $A \oleq B \oleq C$, then $A \oleq C$.
\end{proposition}
\begin{proof} We have to show that for every non-empty $D \subseteq C \setminus A$ there is $d \in D$ which is hyperfree in $AD$.
	\newline \underline{Case 1}. $D \cap (C \setminus B) \neq \emptyset$.
	\newline Let $D' := D \cap (C \setminus B)$. As $\emptyset \neq D' \subseteq C \setminus B$ and $B \oleq C$ there is $d \in D'$ which is hyperfree in $BD'$, but then $d \in D \setminus A$ and $d$ is hyperfree in $AD'$ (as it is hyperfree in $BD'$ and $AD' \subseteq BD'$).
	\newline \underline{Case 2}. $D \subseteq B \setminus A$.
	\newline This case is clear as $A \oleq B$.
\end{proof}

\noindent The following proof adapts \cite[Prop.~2.21]{tent} and \cite[Lem. 8.12]{paolini&hyttinen} to the context of partial $(k, n)$-Steiner systems. Recall that by Remark~\ref{remark_pushout}, and the fact that we later establish Condition \ref{context}(\ref{condition:K}), the following proposition also entails that  $(\mathcal{K},\oleq)$ has the amalgamation property.

\begin{proposition}[{\ref{context}}(\ref{condition:G})]\label{Steiner:conditionG}
	The class $(\mathcal{K},\oleq)$ has the algebraic amalgamation property.
\end{proposition}
\begin{proof}
	Consider $A,B,C\in \mathcal{K}$ with $A\oleq B$ minimal, $A\leq_{|B\setminus A|} C$ and $B\cap C=A$. We distinguish the following cases.
	
	\smallskip
	\noindent \underline{Case 1}. $B=A\cup \{ b \}$ where $b$ is a block.
	\newline Since $A\oleq Ab$ then $b$ is incident with at most  $k$ many points from $A$. 
	\newline\underline{Case 1.1}. $b$ is incident with exactly $k$ many points $p_0,\dots, p_{k-1}$ from $A$ and $C$ contains a block $b'$ incident with $p_0,\dots, p_{k-1}$.
	\newline Since $A\leq_1 C$ there is no other point in $A$ incident to $b'$ and so the map $f:B\to C$ such that $f\restriction A= \mrm{id}_A$ and $f(b)=b'$ is an embedding. Moreover, if $A\oleq C$, then since $f(B)=Ab'$ and $b'$ is algebraic over $A$, it follows by \ref{Steiner:algebraic_lemma} that $f(B)\oleq C$.
	\newline\underline{Case 1.2}. $b$ is incident with exactly $k$ points $p_0,\dots, p_{k-1}$ and $C$ does not contain a block $b'$ incident with $p_0,\dots, p_{k-1}$.
	\newline\noindent In this case the free amalgam $B\otimes_A C$ is a partial $(k, n)$-Steiner system and so $B\otimes_A C\in \mathcal{K}$.	
	\newline\underline{Case 1.3}. $b$ is incident with strictly less than $k$ points from $A$.
	\newline As in the previous subcase, it follows immediately that $B\otimes_A C\in \mathcal{K}$.
	
	\smallskip
	\noindent\underline{Case 2}. $B=A\cup \{ p \}$ where $p$ is a point. 
	\newline  Since $A\oleq Ap$ then $p$ is incident with at most a single block in $A$.
	\newline\underline{Case 2.1}. $p$ is not incident with any block.
	\newline It follows immediately that $B\otimes_A C\in \mathcal{K}$.
	\newline\underline{Case 2.2}. $p$ is incident with a block $b\in A$ and $C\setminus A$ contains less than $n$ points incident to $b$.
	\newline It follows immediately that $B\otimes_A C\in \mathcal{K}$.
	\newline\underline{Case 2.3}. $p$ is incident to $b$ and $C\setminus A$ contains $n$ points $p'_0,\dots,p'_{n-1}$ incident to $b$.
	\newline Since $A\oleq A\cup \{p'_0,\dots,p'_{n-1} \}$, then each point $p'_{i}$ is incident only with $b$. Then, for any $i<n$, the map $f:B\to C$ defined by letting $f\restriction A=\mrm{id}_A$ and $f(p)=p'_i$ is an embedding. If $A\oleq C$, then since $A\oleq Ap_i'$ is algebraic, it follows from \ref{Steiner:algebraic_lemma} that $f(B)\oleq C$. 
	
	\smallskip
	\noindent Finally, by Remark~\ref{Steiner:algebraic_extension} the previous argument establishes the algebraic amalgamation property.
\end{proof}

The next proposition shows that confined extensions of finite, open partial partial $(k, n)$-Steiner systems are algebraic, i.e., we verify \ref{context}(\ref{condition:H}).

\begin{proposition}[{\ref{context}(\ref{condition:H})}]\label{Steiner:conditionH}
	If $A\subseteq B \in \mathcal{K}$ and $A \noleq B$ is confined, then there is $m < \omega$ such that every $C \in \mathcal{K}$ contains at most $m$ disjoint copies of $B$ over $A$.
\end{proposition}
\begin{proof}
	Suppose $A \not\oleq B$ is confined, then  every element in $B$ is not hyperfree in $AB$, meaning that  every point $p\in B$ is incident with at least two blocks from $AB$ and every block $b\in B$ is incident with at least $k+1$ points from $AB'$. We distinguish the following cases.
	
	\medskip
	\noindent \underline{Case 1}. $B\setminus A$ contains a point  $p_0$ which is incident with some block $b\in A$. 
	\newline  Then, if $C\in \mathcal{K}$ contains at least $n+1$ copies of $B$ over $A$ (where $n$ is as in the definition of $(k, n)$-Steiner systems), in particular it contains $n+1$ points $p_0,\dots,p_{n}$ which are all incident with $b\in A$. But this is impossible since $b$ can be incident with at most $n$ points (as $C$ is a partial $(k, n)$-Steiner system).
	
	\medskip
	\noindent \underline{Case 2}. $B\setminus A$ contains no point incident to blocks in $A$.	
	\newline Then, if also no block in $B\setminus A$ is incident to points in $A$, it follows that every point in $B\setminus A$ is incident with at least two blocks from $B\setminus A$ and every block in $B\setminus A$ is incident with at least $k+1$ points from $B\setminus A$. But then $B\setminus A$ contains no element hyperfree in $B\setminus A$, contradicting $B\in \mathcal{K}$. Thus  $B\setminus A$  contains a block  $b$ incident to some point $p\in A$.	
	\newline Then, let $p_0,\dots, p_{m-1}$ be  \emph{all} the points in $A$ which are incident to \emph{some} block in $B\setminus A$.  We let $B_0\coloneqq B$ and let $B_1$ be a disjoint copy of $B_0$ over $A$. We claim that $D=(B_0\setminus A)\cup (B_1\setminus A)\cup\{ p_0,\dots, p_{m-1}  \}$ does not contain any element hyperfree in $D$. By Definition \ref{Steiner:def_hyperfree} this shows that any $C$ containing at least two copies of $B$ over $A$ is not open, proving our claim.  We need to consider the following elements.
	
	\smallskip
	\noindent\underline{Case (a)}. Consider the points $p\in B_i\setminus A$, for $i\in \{0,1\}$.
	\newline Since by assumption we have that no element in $B_i\setminus A$ is hyperfree in $B_i$, we have that $p$ is incident to at least two blocks from $B_i$. By assumption we have that $B_i$ does not contain points incident to blocks in $A$, thus $p$ is incident to at least two blocks from $(B_i\setminus A)\subseteq D$.
	
	\smallskip
	\noindent \underline{Case (b)}. Consider the points $p=p_i$, for $i<m$.
	\newline By choice of the elements $p_0,\dots, p_{m-1}$ we have that $p_i\, \vert\,b_0^i$ for some $b_0^i\in B_0\setminus A$. Since $B_1$ is isomorphic to $B_0$ over $A$, we can also find some block $b_1^i\in B_1\setminus A$ such that $p_i\, \vert\,b_1^i$. So $p_i$ is incident to at least two blocks from $D$.
	
	\smallskip
	\noindent \underline{Case (c)}. Consider the blocks $b\in B_i$,  for $i\in \{0,1\}$.
	\newline Since no element in $B_i\setminus A$ is hyperfree in $B_i$, we have that every block $b\in B_i$ is incident with at least $k+1$ points from $B_i$. By the choice of the points $p_0,\dots, p_{m-1}$, it follows it is incident to at least $k+1$ points from $(B_i\setminus A)\cup \{ p_0,\dots, p_{m-1} \}\subseteq D $. This concludes the proof.
\end{proof}

Finally, the following condition is straightforward to verify, but we provide an argument to show the role of the Gaifman closure, as defined in \ref{def:interior_closure}.

\begin{proposition}[{\ref{context}(\ref{condition:J})}]\label{Steiner:conditionJ}
	There is $\mathbf{n}_\mathcal{K} < \omega$ such that, if  $A \oleq A c\in \mathcal{K}$ is a one-element extension, then $|\mrm{gcl}_{Ac}(c)\setminus\{c\}|\leq \mathbf{n}_\mathcal{K}$.
\end{proposition}
\begin{proof}
Let $A \oleq A c$, then by  Definition \ref{Steiner:def_hyperfree} we have that $c$ is either a point incident to one block from $A$, or it is a block incident to $k$ many points from $A$. In other words, we have from the definition of Gaifman closure from \ref{def:interior_closure} that $\mrm{gcl}_{Ac}(c)\setminus\{c\}=\{b\}$, for $b\in A$ a block, or $\mrm{gcl}_{Ac}(c)\setminus\{c\}=\{p_1,\dots, p_k\}$, for $p_1,\dots, p_k$ points. Letting $\mathbf{n}_\mathcal{K}=k$, the claim is verified immediately.
\end{proof}

\subsubsection{Assumptions \ref{main_th}}\label{steiner:completeness_section}

We showed in the previous section that $(\mathcal{K},\oleq)$ fits the specific subclass of Hrushovski constructions that we introduced in \ref{context}. It follows that $(\mathcal{K},\oleq)$ determines a theory $T^+$, as defined in \ref{the_theory}. We verify in this section that  $T^+$ satisfies also \ref{main_th}(\hyperref[completeness_axiom]{C1})-(\hyperref[the_K_homogeneous_lemma]{C2}) and \ref{ass:hf_closure}(\hyperref[the_hf_axiom]{C3}). By Theorems \ref{prop_completeness}, \ref{stability} and \ref{characterisation_forking} this shows that the theory $T^+$ of open $(k, n)$-Steiner systems is complete, stable and with $\ind=\ind^\otimes$. We first show that Assumption \ref{ass:hf_closure}(\hyperref[the_hf_axiom]{C3}) holds. Since the language of Steiner systems does not include any local equivalence relation, the verification of \ref{ass:hf_closure}(\hyperref[the_hf_axiom]{C3}) proceeds exactly as in the previous case of $n$-open graphs. In fact, also in this case we define the operator $\mrm{cl}_<$ purely in term of the incidence Gaifman closure.

\begin{definition}\label{steiner:closure_operator} Let $A\subseteq B \models T_\forall^n$ and suppose $<$ is a $\mrm{HF}$-order of $B$ over $A$.  For all $C\subseteq B\setminus A$ we define $\mrm{cl}_<(C)\coloneqq \bigcup_{n<\omega }\mrm{cl}_<^{n}(C)$ inductively by letting $\mrm{cl}_<^0(C) \coloneqq C$ and $\mrm{cl}_<^{n+1}(C) \coloneqq \mrm{gcl}^{\mrm{i}}_{A\mrm{cl}_<^{n}(C)^\downarrow}(\mrm{cl}_<^{n}(C))$ for all $n<\omega$.
\end{definition}

\begin{proposition}[{\ref{ass:hf_closure}(\hyperref[the_hf_axiom]{C3})}]\label{steiner:HF_orders}
	Every model $A\subseteq B\models T^n_\forall$ with an associated $\mrm{HF}$-order $<$ of $B$ over $A$ has a $\mrm{HF}$-closure operator $\mrm{cl}_<\coloneqq\bigcup_{n<\omega}\mrm{cl}^n_<$.
\end{proposition}
\begin{proof}
	This follows exactly as in Proposition~\ref{psedoplane:HF_orders}.
\end{proof}

We next consider the technical conditions \ref{technical_assumptions}(\hyperref[trivial_condition]{D1})-(\hyperref[minimality_condition]{D2}). Then, by Proposition \ref{the_K_saturated_lemma}, it will follow that $T^+$ also satisfies Assumption \ref{main_th}(\hyperref[completeness_axiom]{C1}). Notice that we state the condition slightly differently than in \ref{technical_assumptions}(\hyperref[trivial_condition]{D1}), as we add a third condition to simplify our inductive argument. It is immediate to see that the following proposition entails \ref{technical_assumptions}(\hyperref[trivial_condition]{D1}). Also, recall that by \ref{context}(\ref{condition:L}) in the context of $(k, n)$-Steiner systems trivial extensions are always strong.

\begin{proposition}[{\ref{technical_assumptions}(\hyperref[trivial_condition]{D1})}]\label{Steiner:trivial_condition}
	If $M \models T^+$, $A \subseteq M$ is finite, $<$ is a $\mrm{HF}$-order of $M$ and $A\oleq Ac$ is a trivial extension, then for all $\ell < \omega$ there is $ c' = c'_\ell \in M$ s.t.:
	\begin{enumerate}[(a)]
		\item $Ac \cong_A Ac'$;
		\item $Ac' \oleq A\mrm{cl}_<^\ell(c')$;
		\item $A < \mrm{cl}_<^{\ell+1}(c')$ (recall \ref{linear_ordering_convention}).		
	\end{enumerate}
\end{proposition}
\begin{proof} By induction on $\ell < \omega$ we prove that for every finite $A \subseteq M$ and $A \oleq Ac \in \mathcal{K}$ we can find $c'_\ell$ as in (a)-(c). If $\ell = 0$, then it is easy to see that we can find find $c'_0$ as wanted, as it is enough to pick an element of the same sort of $c$ that occurs after $A$ in the order $<$ and that has no relation to elements in $A$. So suppose the inductive hypothesis and let $c$ be such that $A \oleq Ac \in \mathcal{K}$ is a trivial extension, we want to define $c'_{\ell+1}$ so that (a)-(c) are satisfied for it. 
	
	\smallskip	
	\noindent \underline{Case 1}. $c$ is a block.
	\newline By inductive hypothesis we can find a point $d^0_\ell$ such that (a)-(c) is true for $A$. Additionally, by applying the induction hypothesis and since $<$ is a $\mrm{HF}$-order,  we can find inductively, for every $0 < i <k$, a block $b^i_\ell$ and a point $d^i_\ell$ such that $b^i_\ell<d^i_\ell$, $b^i_\ell\, \vert \,d^i_\ell$, $b^i_\ell$ is a trivial extension of $A\cup \bigcup \{\mrm{cl}_<^\ell(d^j_\ell) : 0 \leq j < i\}$ and (a)-(c) is true with respect to $A\cup \bigcup \{\mrm{cl}_<^\ell(d^j_\ell) : 0 \leq j < i\}$. Let $c'_{\ell+1}$ be the (unique) block incident to $(d^i_\ell)_{i<k}$; then for every $0 < i < k$, we have that $c'_{\ell+1} \neq b^i_\ell$ since $b^i_\ell$ is not incident to $d^0_\ell$ but $c'_{\ell+1}$ is. Since $<$ is a $\mrm{HF}$-order, it follows that for every $0 < i<k$ we have $c'_{\ell+1} > d^i_\ell > b^i_\ell$. Further, notice that by construction we clearly have that:
	\begin{align*}
	\mrm{cl}^{\ell+1}_<(c'_{\ell+1})&=\{ c'_{\ell+1} \}\cup \bigcup \{\mrm{cl}^\ell_<(d^i_{\ell}) : i<k \}; \\
	\mrm{cl}^{\ell+2}_<(c'_{\ell+1})&=\{ c'_{\ell+1} \}\cup \bigcup \{\mrm{cl}^{\ell+1}_<(d^i_{\ell}) : i<k \}.
	\end{align*}
	\noindent Thus we obtain by construction and the induction hypothesis that $Ac \cong_A Ac'_{\ell+1}$ and $A < \mrm{cl}_<^{\ell+2}(c'_{\ell+1})$ (whence also $A\oleq \mrm{cl}_<^{\ell+2}(c'_{\ell+1})$). Finally, we claim that $Ac'_{\ell+1} \oleq A\mrm{cl}^{\ell+1}_<(c'_{\ell+1}) $. We show how to $\mrm{HF}$-construct $A\mrm{cl}^{\ell+1}_<(c'_{\ell+1})$ from $Ac'_{\ell+1}$. We start with $Ac'_{\ell+1}$ and define a $\mrm{HF}$-order $<'$ by adding $d^{k-1}_\ell <' d^{k-2}_\ell <' \cdots <'d^0_\ell$ (by the choice of our elements this is a $\mrm{HF}$-order).  Notice now that $c'_{\ell+1}$ is only incident to $(d^i_\ell)_{i<k}$ in $ \mrm{cl}^{\ell+1}_<(c'_{\ell+1}) $. Also, for every $0<i<k$ we have by condition (c) that $A\cup \bigcup \{\mrm{cl}_<^\ell(d^j_\ell) : 0 \leq j < i\}<\mrm{cl}^{\ell+1}_<(b^{i}_{\ell})$. Therefore, since $b^{i}_{\ell}<d^{i}_{\ell}$, it follows that no element in $\mrm{cl}^\ell_<(d^i_{\ell})$ is incident to any element in $\mrm{cl}^\ell_<(d^j_{\ell})$ for $i \neq j < k$. So using the condition (b) from the inductive hypothesis it is easy to continue to construct the $\mrm{HF}$-order $ <' $.
	
	\smallskip 
	\noindent \underline{Case 2}. $c$ is a point.
	\newline By inductive hypothesis we can find a block $b'_\ell$ such that (a)-(c) is true for $A$. In particular $Ab'_\ell \oleq A\mrm{cl}_<^\ell(b'_\ell)$. Then there is a point $c'_{\ell+1}>b'_\ell$ such that $b'_\ell$ is incident to $c'_{\ell+1}$. Clearly we have that $\mrm{cl}_<^{\ell+1}(c'_{\ell+1})=\mrm{cl}_<^{\ell}(b'_{\ell})\cup\{c'_{\ell+1}\}$. From this fact and $Ab'_\ell \oleq A\mrm{cl}_<^\ell(b'_\ell)$, it is easy to reason as in Case 1 and construct a $\mrm{HF}$-order witnessing that  $Ac'_{\ell+1}\oleq A\mrm{cl}_<^{\ell+1}(c'_{\ell+1})$.	By construction  we immediately obtain that $Ac \cong_A Ac'_{\ell+1}$ and, since $\mrm{cl}_<^{\ell+2}(c'_{\ell+1})=\mrm{cl}_<^{\ell+1}(b'_{\ell})\cup\{c'_{\ell+1}\}$, we also have that $A<\mrm{cl}_<^{\ell+2}(c'_{\ell+1})$.
\end{proof}	

We next consider the other assumption needed to prove that $T^+$ is complete, which is essentially a corollary of Lemma \ref{Steiner:algebraic_lemma}.

\begin{proposition}[{\ref{technical_assumptions}(\hyperref[minimality_condition]{D2})}]\label{Steiner:minimality_condition}
	If $M \models T^+$ is $\aleph_1$-saturated, $A \oleq M$ is countable, and for every trivial extension $A\oleq Ab$ there is $b' \in M$ such that $Ab \cong_A Ab'\oleq M$, then for every minimal extension $A\oleq Ac$ there is $c' \in M$ such that ${Ac \cong_A Ac'\oleq M}$.
\end{proposition}
\begin{proof}
	Let $M \models T^+$ and $A \oleq M$. Suppose $A\oleq Ac$ is a minimal strong extension. We distinguish the following cases.
	
	\noindent \underline{Case 1}. $c$ is a point.
	\newline\noindent If  $A\oleq Ac$ is a trivial extension then we are immediately  done by the  assumption. Otherwise, $A\oleq Ac$ is algebraic and $c$ is incident to exactly one block $b\in A$. Let $ p_0,\dots,p_\ell $  be the points incident to $b$ in $A$ and consider the set $B=\{p_0,\dots,p_\ell, b\}$.  Since $M\models T^+$, and since clearly $B\leq_{1}M$, there is some $c'\in M\setminus B$ incident to $b$. By the choice of the set $B$, it also follows that $c'\notin A$ and, since $A\oleq M$ it follows that $c'$ is incident to no other element in $A$, yielding $Ac\cong_A Ac'$. By Lemma \ref{Steiner:algebraic_lemma}, we obtain that $Ac'\oleq M$.
	
	\noindent \underline{Case 2}. $c$ is a block.
	\newline\noindent If $c$ is a block and $A\oleq Ac$ is not trivial, then $c$ is incident to $0<\ell\leq k$ points $p_0,\dots,p_{\ell-1}$ from $A$. By assumption, we can find $k-\ell$ new points $p_\ell,\dots,p_{k-1}$ in $M\setminus A$ which have no incidence with $A$ and such that $A\cup \{p_\ell,\dots,p_{k-1}\}\oleq M$. Let $c'\in M\setminus A$ be the unique block which is incident to $p_0,\dots,p_{k-1}$ (this exists because $M\models T^+$). Clearly we have that $Ac \cong_A Ac'$ and, by Lemma \ref{Steiner:algebraic_lemma}, we obtain $A\cup \{p_\ell,\dots,p_{k-1}\}\cup \{c'\}\oleq M$. Since $p_\ell,\dots,p_{k-1}$ are incident with no element in $A$, we also obtain that $Ac'\oleq Ac'\cup \{p_\ell,\dots,p_{k-1}\}$. By transitivity we conclude that $Ac'\oleq M$.	
\end{proof}

It already follows from the previous propositions and Theorem \ref{prop_completeness} that the theory $T^+$ of open $(k, n)$-Steiner systems is complete.  We next consider Assumption \ref{main_th}(\hyperref[the_K_homogeneous_lemma]{C2}), which is needed to show that $T^+$ is stable. As in the case of $n$-open graphs, we establish this condition by verifying \ref{technical_assumptions}(\hyperref[extension]{D3}) and applying Proposition~\ref{lemma:homogeneity}. We start by recalling the free extension process for partial $(k, n)$-Steiner systems. This makes explicit how the abstract construction of a free completion from \ref{free_algebraic_completion} behaves in this specific setting.

\begin{definition}\label{Steiner:free_completion}
	Let $A$ be a partial $(k, n)$-Steiner system. We let $A_0=A$ and for every $2i<\omega$ we define $A_{2i+1}$ and $A_{2i+2}$ as follows: 	
	\begin{enumerate}[(a)]
		\item for every sequence of $k$ points $p_1,\dots,p_{k}$ from $A_{2i}$ which do not have a common incident block in $A_{2i}$, we add a new block $b\in A_{2i+1}$ which is incident only with $p_1,\dots,p_{k}$;
		\item for every block $b\in A_{2i+1}$ which is incident only to $\ell<n$ points from $A_{2i+1}$, we add $n-\ell$ points $p_\ell,\dots,p_{n-1}\in A_{2i+2}$ incident only with $b$.
	\end{enumerate}	
	Then the structure $F(A)\coloneq\bigcup_{i < \omega}A_i$ is called the \emph{free completion} of $A$ and we say that $F(A)$ is freely generated over $A$. We say that $A$ is \emph{non-degenerate} if $F(A)$ is infinite, and \emph{degenerate} otherwise. 
\end{definition}

\begin{remark}\label{steiner:non-degenerate}
Notice that a partial $(k, n)$-Steiner system $A$ is non-degenerate if and only if it contains at least two blocks or at least $k+1$ many points not incident to the same block (cf.~\cite[p.~756]{funk2}). By letting $m_{\mathcal{K}}=k+1$ we thus have that any $A\in \mathcal{K}$ containing at least $\mathbf{m}_\mathcal{K}$ many elements of each sort is non-degenerate.
\end{remark}

The next proposition provides the key properties of the free completion process for Steiner systems, and it simultaneously verifies Condition \ref{technical_assumptions}(\hyperref[extension]{D3}) and also Assumption \ref{assumptions:no_prime}(\hyperref[F=C]{C5}). It follows that the theory $T^+$ of open $(k, n)$-Steiner systems is stable and its forking independence relation coincides with the relation $\ind^\otimes$ defined in \ref{def:free_independence}.

\begin{proposition}[{\ref{technical_assumptions}(\hyperref[extension]{D3})}, {\ref{assumptions:no_prime}(\hyperref[F=C]{C5})}]\label{free_theory_steiner}
	Suppose that $A\models T_\forall$ and that $A$ is non-degenerate. Then $F(A)\models T^+$, $A\oleq F(A)$, $|F(A)|=|A|+\aleph_0$ and:
	\begin{enumerate}[(a)]
		\item $F(A)$ is unique up to $A$-isomorphisms, i.e., every isomorphism $h:A\cong B$ extends to an isomorphism $\hat{h}:F(A)\cong F(B)$;
		\item $F(A)$ is $\oleq$-prime over $A$, i.e.,  for every $M\models T^+$ and every embedding $h:A\to M$ with $h(A)\oleq M$ there is an embedding $\hat{h}: F(A)\to M$ such that $\hat{h}\restriction A = h$ and $\hat{h}(F(A))\oleq M$.
	\end{enumerate}
	In particular, it follows that if $A\cong B\onleq M\models T_n^+$ and $M$ is $\aleph_1$-saturated, then $F(A)\cong F(B)\preccurlyeq M$.
\end{proposition}
\begin{proof} 
	By \ref{key_properties:free_completion} and Theorem~\ref{prop_completeness}  it suffices to show the following: if $A$ contains at least $\mathbf{m}_\mathcal{K}=k+1$ elements of every sort, then $F(A)\models T^+$. In fact, if this is the case, then $A\cong B\onleq M\models T_n^+$ for $M$ $\aleph_1$-saturated entails immediately that  $F(A)\cong  F(B)\onleq M$ and, moreover, we can assume by \ref{main_th}(\hyperref[completeness_axiom]{C1}) that $A,B$ are non-degenerate and so $F(A), F(B)\models T^+$. Then, we also get by \ref{prop_completeness} that $F(A)\cong  F(B)\preccurlyeq M$, establishing our condition.
	
	\medskip 
	\noindent We thus show that, if $F(A)$ is infinite, then $F(A)\models T^+$. The fact that $F(A)\models T_\forall$ follows already from \ref{key_properties:free_completion}. The algebraic family of axioms from \ref{the_theory}(2) is verified immediately by the construction of $F(A)$. In fact, by Remark \ref{Steiner:algebraic_extension} the $\mathcal{K}$-algebraic extensions of partial $(k, n)$-Steiner system are exactly those considered in the inductive construction of $F(A)$. In particular, for any algebraic extension $X\oleq Xc$ with $X\subseteq_\omega F(A)$ and $X\leq_1 F(A)$ we have that $X\subseteq A_i$ for some $i<\omega$. Then, either $X\oleq Xc$ is already realised in $A_i$, or $A_i\oleq A_ic$ is also a $\mathcal{K}$-algebraic  extension and is thus realised in $A_{i+1}\subseteq F(A)$.
	
	\medskip 
	\noindent   Finally, we next consider the non-algebraic family of axioms from \ref{the_theory}(3). Consider a minimal non-algebraic extension $B\oleq Bc$ with $B\subseteq A_i$, where $A_i$ is the $i$-th step of the construction from \ref{Steiner:free_completion}. We distinguish the following cases.
	
	\smallskip 
	\noindent \underline{Case 1}. $c$ is a point.
	\newline Then $B\oleq Bc$ is a trivial extension (cf.~\ref{def_extensions}). Since $A$ is non-degenerate, for every even $j>i$ there are a block $b_j\in A_{j+1}\setminus A_j$ and a point $c_j\in A_{j+2}\setminus A_{j+1}$ incident to $b_j$. Then $c_j$ is incident \emph{only} to $b_j$ and so $Bc\cong_B Bc_j\subseteq F(A)$ for all even $j<\omega$. 
	
	\smallskip 
	\noindent  \underline{Case 2}. $c$ is a block incident to $\ell<k$ points $p_0,\dots,p_{\ell-1}$ in $B$.
	\newline Since $A$ is non-degenerate, we can find an infinite sequence  of indices  $i=j_0<j_1<j_2<\cdots$ such that for every $0<m<\omega$ we can find $(k-\ell)$ many points $p^{j_m}_\ell,\dots,p^{j_m}_{k-1}\in A_{j_m}\setminus A_{j_{m-1}}$ which are not incident to any common block in $A_{j_m}$. Consider then, for each $0<m<\omega$, the block $c_{j_m}$ added in stage $j_m+1$  of the completion process (or $j_m+2$, depending on the parity of $j_m$) which is incident exactly to $p_0,\dots,p_{\ell-1}$ and $p^{j_m}_\ell,\dots,p^{j_m}_{k-1}$. Then $Bc\cong_B Bc_{j_m}\subseteq F(A)$ for all $0<m<\omega$, proving our claim and showing that $F(A)\models T^+$.
\end{proof}

\subsubsection{Assumptions \ref{assumption:no-superstability} and \ref{assumptions:no_prime}}\label{steiner:superstability_prime}

We consider in this section Assumption \ref{assumption:no-superstability_2}(\hyperref[CP]{C4}) and Assumptions \ref{assumptions:no_prime}(\hyperref[F=C]{C5})-(\hyperref[delta-rank]{C7}). First, to show that $T^+$ is not superstable we exhibit a configuration in $\mathcal{K}$ witnessing that the technical assumption from \ref{assumption:no-superstability}(\hyperref[technical:D4]{D4}) is satisfied. The results from this section thus completes the proof of Theorem \ref{corollary:steiner}.

\begin{proposition}[{\ref{assumption:no-superstability}(\hyperref[technical:D4]{D4})}]\label{superstability_steiner}
	There is  $C\in \mathcal{K}$ satisfying the conditions from \ref{assumption:no-superstability}.
\end{proposition}
\begin{proof}
	Recall that $k$ and $n$ are fixed. We define a configuration $C$ as in the following table. For simplicity, we let $k=2$ and $n=3$, but it is then straightforward to adapt our example to the arbitrary case of $2\leq k<n$.
	\begin{table}[H]
	\begin{tabular}{|c||c|c|c|c|c|} \hline
		& $c_2$ & $c_4$ & $c_7$ & $c_8$ & $c_{11}$ \\ \hline \hline
		$c_0$  & $\times$ &   & $\times$ &   &    \\ \hline
		$c_1$  & $\times$ &   &   & $\times$ &    \\ \hline
		$c_3$  & $\times$ & $\times$ &   &   &    \\ \hline
		$c_5$  &   & $\times$ & $\times$ &   &    \\ \hline
		$c_6$  &   & $\times$ &   & $\times$ &    \\ \hline
		$c_9$  &   &   & $\times$ &   & $\times$  \\ \hline
		$c_{10}$ &   &   &   & $\times$ & $\times$  \\ \hline
		$c_{12}$ &   &   &   &   & $\times$ \\ \hline
	\end{tabular}
\end{table}
\noindent  Let $C=\{c_i :  i \leq 12\}$, where $c_0,c_1,c_3,c_5,c_6,c_9,c_{10},c_{12}$ are points, $c_2,c_4,c_7,c_8,c_{11}$ are blocks, and the incidences between them are specified by the table above. Then $C$ is a partial $(2,3)$-Steiner system and the order $c_0<c_1<\dots<c_{12}$ is a $\mrm{HF}$-order, which means that $C\in \mathcal{K}$. 	As the other requirements from  \ref{assumption:no-superstability} are easily verified, it follows that $C$ is as desired.
\end{proof}

Finally, we consider the conditions from  \ref{assumptions:no_prime}(\hyperref[F=C]{C5})-(\hyperref[delta-rank]{C7}). Assumption \ref{assumptions:no_prime}(\hyperref[F=C]{C5}) follows immediately from \ref{free_theory_steiner}, whence it remains to consider \ref{assumptions:no_prime}(\hyperref[hopf]{C6}) and \ref{assumptions:no_prime}(\hyperref[delta-rank]{C7}).

\begin{proposition}[{\ref{assumptions:no_prime}(\hyperref[hopf]{C6})}]
	If $F(A)\models T^+$ and $A,B\in \mathcal{K}$ then there is a non-surjective embedding of $F(B)$ into $F(A)$.
\end{proposition}
\begin{proof}
Let $A=\bigcup_{i<\omega}A_i$ as in \ref{Steiner:free_completion}. Suppose for simplicity that $B$ contains exclusively points. Then, for $\ell<\omega$ large enough, we can find a stage $A_\ell$ in the free completion $F(A)$ of $A$, such that $A_\ell$ contains an isomorphic copy $C$ of $B$ such that $A_\ell\setminus C$ does not contain any block incident to at least $k$ many points in $C$. Then, it is clear from Definition \ref{Steiner:free_completion} that the free completion $F(A)$ will contain an isomorphic copy $F(C)$ of $F(B)$. Finally, we notice that if $B$ also contains blocks, then this argument can be adapted by making sure that any two different blocks $b,b'\in B$ are mapped to blocks $c,c'\in C\subseteq A_\ell$ in such a way that all the points incident to $c$ and $c'$ in $A_\ell\setminus C$ have no common incidence in $A_\ell$.
\end{proof}

Finally, we introduce a suitable predimension function for the setting of open $(k, n)$-Steiner systems. To our knowledge, this had not been considered previously in the literature. 

\begin{definition}\label{definition:delta} Let $A\in \mathcal{K}$ be a finite, open, partial $(k, n)$-Steiner system, we define a \emph{predimension function} $\delta:\mathcal{K}\to \omega$ as follows. Let $P_A$, $L_A$ and $I_A$ denote points, blocks and incidences of $A$, respectively. Then we define 
		\[ \delta(A) = |P_A| + k|L_A| - |I_A|.\] 
\end{definition}

\begin{proposition}[{\ref{assumptions:no_prime}(\hyperref[delta-rank]{C7})}]
	The predimension function $\delta:\mathcal{K}\to \omega$ satisfies the properties from Definition \ref{delta_strong}, i.e., if $A,B\in\mathcal{K}$, then:
	\begin{enumerate}[(a)]
		\item $A\cong B$ implies $\delta(A)=\delta(B)$;
		\item $A\oleq B$ implies $A \leq_\delta B$;
		\item if $A\oleq B$ is minimal, then $\delta(A)= \delta(B)$ if and only if $B$ is algebraic over $A$;
		\item if $A \leq_\delta B $ and $C \subseteq B$, then $A\cap C \leq_\delta C$. 
	\end{enumerate}
\end{proposition}
\begin{proof} 
	Clause (a) follows immediately by the definition of the predimension $\delta$. Then, by definition, it is clear that if $A\oleq Ab$ is an algebraic strong extension, then $\delta(A)=\delta(Ab)$. Differently, if $A\oleq Ab$ is a non-algebraic strong extension we have that $\delta(A)<\delta(Ab)$. Proceeding inductively, it follows that whenever $A\oleq C\in \mathcal{K}$ then $\delta(A)\leq \delta(C)$. 
	\newline It thus remain to verify condition (d). Now, if $X,Y\subseteq A$ are disjoint, we denote by $I(X,Y)$ the set of incidences between $X$ and $Y$ in $A$. Suppose $A \leq_\delta B $ and let $C \subseteq B$. By \ref{definition:delta} we have that:
	\begin{align*}
	\delta(A\cap C)	&= |P_{A\cap C}| + k|L_{A\cap C}| - |I_{A\cap C}| \\
	&= |P_{C}| - |P_{C\setminus A}| + k |L_{C}|-k |L_{C\setminus A}| - |I_{C}| +|I_{C\setminus A}| + I(C\cap A,C\setminus A)    \\
	&= \delta(C) - \delta(C\setminus A) + I(C\cap A,C\setminus A).
	\end{align*}
	Since $A \leq_\delta B$ we have $\delta(A)\leq \delta(AC)$, thus $\delta(A)\leq \delta(A)+\delta(C\setminus A)- I(C\cap A,C\setminus A)$ and in particular $\delta(C\setminus A)- I(C\cap A,C\setminus A)\geq 0$. It follows that 
	\begin{align*}
	\delta(A\cap C)=\delta(C) - \delta(C\setminus A) + I(C\cap A,C\setminus A) \leq \delta(C)
	\end{align*}
	which proves $A \cap C \leq_\delta C$.
\end{proof}

\subsection{Generalised $n$-gons} \label{section:ngons}

In this section we apply our framework to the case of generalised $n$-gons. We stress again that this has already been considered in the literature: Hyttinen and Paolini considered in \cite{paolini&hyttinen} the case of projective planes, i.e., 3-gons, while Ammer and Tent considered in \cite{tent} arbitrary generalised $n$-gons. Possibly, this is most interesting application of our framework, since generalised $n$-gons are the rank $2$ case of spherical buildings and have a natural connection to higher-rank geometries, which, as mentioned in the introduction, we intend to explore in a work in preparation. Additionally, generalised $n$-gons have a vast number of connections to the general mathematical literature (for which we refer the interested reader to \cite{maldeghem}). Since our techniques differ from those employed in \cite{tent} (although they are very much inspired by them), we provide details of the proofs, except for those that proceed as in the case of Steiner systems, or that are exactly as in \cite{tent} or \cite{paolini&hyttinen}. We stress that generalised $n$-gons are the key example where a minimal extension $A\oleq B$ is not necessarily a one-element extension, and they are the main reason why we introduced an arbitrary finite set $\mathbf{X}_\mathcal{K}$ in Condition \ref{context}(\ref{condition:I}) to encode the possible length of minimal strong extensions. The following theorem (which was proved already in \cite{tent}) holds in the setting of generalised $n$-gons. 

\begin{theorem}\label{corollary:ngons}
	Let $(\mathcal{K},\oleq)$ be the class of finite, open, partial generalised $n$-gons and let $T^+$ be defined as in \ref{the_theory}. Then $T^+$ is complete, stable, and with $\ind=\ind^{\otimes}$. Moreover, $T^+$ is not superstable, it does not have a prime model, it is not model complete and it does not eliminate quantifiers.
\end{theorem}

We start by recalling some graph-theoretic notions, and the graph-theoretic definition of the generalised $n$-gons  (cf.~\cite[Def. 2.1]{tent}). A different and more geometrical definition of $n$-gons can be found, e.g., in \cite[p.~745]{funk2}.

\begin{definition}
	Let $L$ be the two-sorted language where the symbols $p_0,p_1,\dots$ denote \emph{points} and the symbols $\ell_0,\ell_1,\dots$ denote \emph{lines}. The symbol $p\, \vert\,\ell$ (or $\ell\, \vert\,p$) means that the point $p$ is \emph{incident} with the line $\ell$.  For every $n\geq 3$, we define a \emph{generalised $n$-gon} as a bipartite graph (with points and lines) with girth $2n$ and diameter $n$. A \emph{partial $n$-gon} is a model of the universal fragment of the theory of generalised $n$-gons, i.e., it is a bipartite graph with girth at least $2n$.
\end{definition}

We notice that, by the fact that partial $n$-gons are bipartite, they do not have any odd-length cycle. Moreover, differently from \cite{tent}, we shall not assume that partial $n$-gons are always connected. We make explicit out choice of the class $\mathcal{K}$ of partial open $n$-gons in the following definition

\begin{definition}\label{ngons:def_hyperfree}
	For finite partial $n$-gons $A\subseteq B$, we say that a tuple $\bar{a}\in A^{<\omega}$ is \emph{hyperfree in $B$} if it satisfies one of the following conditions:
		\begin{enumerate}[(i)]
			\item $\bar{a}=a$ is a \emph{loose end}, i.e., it is a single element incident with at most one element in $B$;
			\item $\bar{a}$ is a \emph{clean arc}, i.e., $\bar{a}$ is a chain $a_1\, \vert \, a_2\, \vert\ \dots\, \vert\ a_{n-3}\, \vert\,a_{n-2}$ where each $a_i$ for $1\leq i\leq n-2$ is incident to exactly two elements in $B$.
		\end{enumerate}
		We then write $A \oleq B$ if every non-empty $C \subseteq B \setminus A$ contains a tuple which is hyperfree in $AC$. We then adopt the notions of \emph{open}, \emph{closed}, \emph{confined}, etc. specialising Definition~\ref{def_extensions}. We let $\mathcal{K}$ be the set of finite, open, partial $n$-gons and we let $T_\forall$ be the universal theory of $\mathcal{K}$.
\end{definition}

\begin{remark}\label{ngons:algebraic_extension} 
	We point out that our notion of hyperfree tuples follows \cite[Def. 2.10]{tent} and it differs from both \cite[p.~351]{funk1} and \cite[p.~753]{funk2}. In fact, the issue with these latter definitions is that they are introduced at the level of single elements, which creates problems with the amalgamation property. We thus stress that  generalised $n$-gons are the only concrete examples where we use the full generality of our setting, i.e., where minimal hyperfree extensions are not necessarily one-element extensions but can also consist of longer tuples.  In particular, an extension $A\oleq B\in \mathcal{K}$ is minimal if $B\setminus A=\{c\}$ is a single element incident to at most one element from $A$, or if $B\setminus A$ is a clean arc.
	
	\smallskip
	\noindent To see that clean arcs form minimal extensions, notice that if $A\oleq B$ and $B\setminus A$ is a clean arc $\bar{c}$, then given $A\subsetneq C\subseteq B$ we have that every element in $B\setminus C$ is incident to two elements from $B$, but the clean arc $\bar{c}$ is not fully included in $B\setminus C$. This shows that $C\noleq B$ and proves that clean arcs form minimal extensions.
	
	\smallskip
	\noindent Secondly, we claim that extensions $A\oleq A\bar{c}$ where $\bar{c}$ is a clean arc are exactly the only minimal algebraic extensions.  In fact, if $a$ and $b$ are the endpoints of the clean arc $\bar{c}$ then $d_{A\bar{c}}(a,b)=n-1$. Hence, if $C\in \mathcal{K}$ contains an isomorphic copy of $A\bar{c}$ over $A$, it follows that $a,b$ belong to two paths of length $n-1$, which together form a $(n-1)$-gon (i.e., a cycle of length $2n-2$). This contradicts $C\in \mathcal{K}$ and shows that $A\oleq A\bar{c}$ is algebraic. On the other hand, it is immediate to see that extensions by loose ends are not algebraic.	
\end{remark}

The next lemma is essentially Lemma \ref{Steiner:algebraic_lemma} but for partial $n$-gons. We omit the proof as it is fundamentally the same. Notice also that this establishes part \ref{condition:amalgam}(2) of the algebraic amalgamation property for $(\mathcal{K},\oleq)$.

\begin{lemma}\label{ngons:algebraic_lemma}
	Let $A\oleq C\models T_\forall$ and let $\bar{c}\in C^{<\omega}$ be such that $A\oleq A\bar{c}$ is a minimal algebraic extension. Then $A\bar{c}\oleq C$. 
\end{lemma}
\begin{proof}
	The proof follows analogously to the proof of Lemma \ref{Steiner:algebraic_lemma}.
\end{proof}

We start by verifying that the class $(\mathcal{K},\oleq)$ satisfies the conditions from \ref{context}.  As in the case of Steiner systems, conditions (\ref{condition:A}), (\ref{condition:B}), (\ref{condition:D}), (\ref{condition:E}), (\ref{condition:F}), (\ref{condition:K}) and (\ref{condition:L}) from \ref{context} follow immediately from Definition \ref{ngons:def_hyperfree}.  Then, Condition (\ref{condition:C}) follows as in the case of Steiner systems and Condition (\ref{condition:H}) follows exactly as in \cite[Lemma 2.24]{tent}. By Remark \ref{ngons:algebraic_extension} it is easy to see that Condition (\ref{condition:I}) holds by letting $\mathbf{X}_\mathcal{K}=\{1,n-2\}$ -- namely, minimal strong extensions are obtained  by adjoining either a loose end or a clean arc of length $n-2$. Condition (\ref{condition:J}) holds immediately with $\mathbf{n}_\mathcal{K}=1$. Finally, although the proof of Condition (\ref{condition:G}) is similar to \cite[Prop.~2.21]{tent}, we provide details of it to clarify that it holds also when considering non-connected partial $n$-gons, and that it actually verifies the algebraic amalgamation property for $(\mathcal{K},\oleq)$, not simply the amalgamation property.  By Remark~\ref{remark_pushout} and Condition (\ref{condition:K}), the following proposition also entails that  $(\mathcal{K},\oleq)$ has the amalgamation property.

\begin{proposition}[{\ref{context}}(\ref{condition:G})]\label{ngons:conditionG}
	The class $(\mathcal{K},\oleq)$ has the algebraic amalgamation property.
\end{proposition}
\begin{proof}
	Let $A,B,C\in \mathcal{K}$ with $A\oleq B$ a minimal strong extension, $A\leq_{|B\setminus A|} C$ and $B\cap C=A$. We distinguish the following cases.
	
	\smallskip
	\noindent \underline{Case 1}. $B\setminus A=\{c\}$ and $c$ is incident to at most one element in $A$.
	\newline\noindent Then the free amalgam $B\otimes_A C$ is an open partial $n$-gon, whence $B\otimes_A C\in \mathcal{K}$.	
	
	\smallskip
	\noindent\underline{Case 2}. $B\setminus A=\{c_1,\dots,c_{n-2}\}$ and $\bar{c}=(c_1,\dots,c_{n-2})$ is a clean arc with endpoints $a,b\in A$.
	\smallskip
	\newline\underline{Case 2.1}. There is a clean arc $\bar{d}=(d_1,\dots,d_{n-2})\in C^{n}$ with endpoints $a,b\in A$.
	\newline Then define $f:B\to C$ such that $f\restriction A= \mrm{id}_A$ and $f(c_i)=d_i$ for all $1\leq i \leq n-2$. This is obviously an embedding. If moreover we have $A\oleq C$, then since $f(B)=A\bar{d}$ and $\bar{d}$ is algebraic over $A$, it follows by \ref{Steiner:algebraic_lemma} that $f(B)\oleq C$.
	\smallskip
	\newline \underline{Case 2.2}. There is no clean arc $\bar{d}=(d_1,\dots,d_{n-2})\in C^{n-2}$ with endpoints $a,b\in A$.
	\newline Now, since $A\bar{c}$ has girth $\geq 2n$, it follows in particular that $d(a,b/A)\geq n+1$, for otherwise we would have a short cycle in $A\bar{c}$. Moreover, since $A\leq_{n-2}C$, it follows that $C$ does not contain any path between $a$ and $b$ of length $\leq n$.	We thus obtain that the shortest cycle in $B\otimes_A C$ to which $a$ and $b$ belong contains at least $2n$ many elements. Thus, the free amalgam $B\otimes_A C$ is an open partial $n$-gon in $\mathcal{K}$. 
	
	\smallskip
	\noindent Finally, it follows from Remark~\ref{Steiner:algebraic_extension} that the previous argument establishes the algebraic amalgamation property, which completes our proof.
\end{proof}

We next consider the assumptions from \ref{main_th}(\hyperref[completeness_axiom]{C1})-(\hyperref[the_K_homogeneous_lemma]{C2}) and \ref{ass:hf_closure}(\hyperref[the_hf_axiom]{C3}). We proceed as in the previous section and first verify \ref{ass:hf_closure}(\hyperref[the_hf_axiom]{C3}), i.e., we exhibit a $\mrm{HF}$-closure operator for partial open generalised $n$-gons with a $\mrm{HF}$-order. Since the language of generalised $n$-gons has no local equivalence relation, this proceed exactly as in the previous cases of $n$-open graphs and $(k,n)$-Steiner systems. The only difference is that, since $\mrm{HF}$-orders of partial $n$-gons contain both loose ends and clean arcs, we need at each step to close the incidence Gaifman closure under the  operator $\widehat{X}$ from \ref{hat_notation}. In particular, recall that if $(<,P) $ is a $\mrm{HF}$-order of $M$ and $c\in M$, then we write $\widehat{c}$ for the tuple $(c_0,\dots,c_k)$ such that $\{ c_0,\dots,c_k \}\in P$ and $c_0<\dots<c_k$.

\begin{definition}\label{ngons:closure_operator} Let $A\subseteq B \models T_\forall^n$ and suppose $<$ is a $\mrm{HF}$-order of $B$ over $A$.  For all $C\subseteq B\setminus A$ we define $\mrm{cl}_<(C)\coloneqq \bigcup_{n<\omega }\mrm{cl}_<^{n}(C)$ inductively by letting $\mrm{cl}_<^0(C) \coloneqq \widehat{C}$ and $\mrm{cl}_<^{n+1}(C) \coloneqq \widehat{\mrm{gcl}^{\mrm{i}}_{A\mrm{cl}_<^{n}(C)^\downarrow}(\mrm{cl}_<^{n}(C))}$ for all $n<\omega$.
\end{definition}

\begin{proposition}[{\ref{ass:hf_closure}(\hyperref[the_hf_axiom]{C3})}]\label{ngons:HF_orders}
	Every model $A\subseteq B\models T^n_\forall$ with an associated $\mrm{HF}$-order $<$ of $B$ over $A$ has a $\mrm{HF}$-closure operator $\mrm{cl}_<\coloneqq\bigcup_{n<\omega}\mrm{cl}^n_<$.
\end{proposition}
\begin{proof}
	This follows exactly as in Proposition~\ref{psedoplane:HF_orders}.
\end{proof}

Then, given \ref{ass:hf_closure}(\hyperref[the_hf_axiom]{C3}), recall that to establish \ref{main_th}(\hyperref[completeness_axiom]{C1}) it suffices to verify   \ref{technical_assumptions}(\hyperref[trivial_condition]{D1}) and \ref{technical_assumptions}(\hyperref[minimality_condition]{D2}), and to establish \ref{main_th}(\hyperref[the_K_homogeneous_lemma]{C2}) it suffices to verify \ref{technical_assumptions}(\hyperref[extension]{D3}). We do this by first showing condition \ref{technical_assumptions}(\hyperref[trivial_condition]{D1}). As in the setting of Steiner systems, we add a further condition (c) which is useful in the induction. 

\begin{proposition}[{\ref{technical_assumptions}(\hyperref[trivial_condition]{D1})}]\label{ngons:trivial_condition}
	If $M \models T^+$, $A \subseteq M$ is finite, $<$ is a $\mrm{HF}$-order of $M$ and $A \oleq Ac$ is a trivial extension, then for all $\ell < \omega$ there is $ c' = c'_\ell \in M$ s.t.:
	\begin{enumerate}[(a)]
		\item $Ac \cong_A Ac'$;
		\item $Ac' \oleq A\mrm{cl}_<^\ell(c')$;
		\item $d(A,c'/A\mrm{cl}_<^{\ell+1}(c'))>\ell+n$.
	\end{enumerate}
\end{proposition}
\begin{proof}	
	We assume w.l.o.g. that $n$ is odd, for the case with $n$ even is analogous. In particular this means that clean arcs have as endpoints elements of the same sort. We proceed similarly to the proof of \ref{Steiner:trivial_condition}. By induction on $\ell < \omega$ we prove that for every $A \subseteq M$ and $A \oleq Ac \in \mathcal{K}$ we can find $c'_\ell$ satisfying conditions (a)-(c). 
	\newline \underline{Base case $\ell = 0$}. 
	\newline 
	\noindent We find an element $c'_{\ell}$ as wanted. Since $M\models T^+$ and $<$ is a $\mrm{HF}$-order, there are cofinitely many clean arcs in the ordering $<$  (for otherwise we would have points at distance $>n$).  It follows in particular that we can find a clean arc $\bar{x}=(x_1,\dots,x_{n-2})$ with endpoints $b,c\in M$ such that $\widehat{b}$ and $\widehat{c}$ do not contain any element from $A$. Similarly, let $\bar{y}=(y_1,\dots,y_{n-2})>\bar{x}$ be another such clean arc with endpoints $d,e\in M $ such that $\widehat{d}$ and $\widehat{e}$ do not contain any element from $A$. Additionally, we assume w.l.o.g. that $b,c,d,e$ all have the same sort and that $x_1<\dots<x_{n-2}<y_1<\dots<y_{n-2}$. Let $k=\frac{n-1}{2}$. Then in $y_{n-2}^\downarrow$ the elements $x_k$ and $y_k$ must have distance $\geq n+1$, since the elements $b,c,d,e$ are pairwise different and have the same sort. Consider then the clean arc $\bar{z}=(z_1,\dots, z_{n-2})$ between $x_{k}$ and $y_k$ and notice that we must have $z_i>y_{n-2}$ for all $1\leq i \leq n-2$. Then we let  $c'_0=z_{k}$ and from \ref{ngons:closure_operator} it follows that we have:
	\begin{align*}
		\mrm{cl}_<^{1}(c'_0)&=\{x_1\dots, x_{n-2}, y_1\dots, y_{n-2},z_1\dots, z_{n-2} \}\\
		\mrm{cl}_<^{0}(c'_0)& =\{z_1\dots, z_{n-2}\}.
	\end{align*}
	Then, we obtain that $d(A,c'_0/A\mrm{cl}_<^{1}(c'_0))>n$, as no element from the tuples $\bar{x}, \bar{y}, \bar{z}$  is  incident to elements from $A$. Similarly, we also have that $Ac'_0 \oleq A\mrm{cl}_<^{0}(c'_0)$, as we can define a $\mrm{HF}$-order $<'$ by letting $Ac'_0<'z_{k-1}<'\dots<z_1<'z_{k+1}<'\dots <'z_{n-2}$. Finally, we clearly have that   $Ac \cong_A Ac'_0$.
	\smallskip
	\noindent \newline \underline{Inductive step}. 
	\newline \noindent Now, suppose the inductive hypothesis holds for $\ell\geq n$ and let $c$ be such that $A \oleq Ac$ is a trivial extension, we want to define $c'_{\ell+1}$ so that (a)-(c) are satisfied for it. By inductive hypothesis we can find a point $d^0_\ell$ satisfying (a)-(c) with respect to $A$. Additionally we can find a second point $d^1_\ell$ such that $d^1_\ell$ is a trivial extension of $A\mrm{cl}_<^{\ell}(d^0_\ell)$ and (a)-(c) are true with respect to it. Now, by the choice of $d^1_\ell$ and condition (c) it follows that there is no path of length $\leq n$ between $d^0_\ell$ and $d^1_\ell$ in $\mrm{cl}_<^{\ell}(d^1_\ell)$. But then since $d^0_\ell<d^1_\ell$ and $<$ is a $\mrm{HF}$-order, it follows that the shortest path between $d^0_\ell$ and $d^1_\ell$ is a clean arc $\bar{z}=(z_1,\dots,z_{n-2})$ s.t. $d^1_\ell<z_1<\dots<z_{n-2}$.  Let $c_{\ell+1}'=z_1$, then we have:
	\begin{align*}
	\mrm{cl}^{\ell+2}_<(c'_{\ell+1})&=\{ z_1,\dots, z_{n-2} \}\cup \mrm{cl}^{\ell+1}_<(d^0_{\ell})\cup \mrm{cl}^{\ell+1}_<(d^1_{\ell}),\\
	\mrm{cl}^{\ell+1}_<(c'_{\ell+1})&= \{ z_1,\dots, z_{n-2} \}\cup \mrm{cl}^{\ell}_<(d^0_{\ell})\cup \mrm{cl}^{\ell}_<(d^1_{\ell}).
	\end{align*}
	Whence, from the induction hypothesis and the display above  we immediately obtain that $d(A, c'_{\ell+1}/A\mrm{cl}^{\ell+2}_<(c'_{\ell+1}))> (\ell+1)+n$. Similarly, it is clear from the choice of  $c_{\ell+1}'$ that we also have $Ac \cong_A Ac'_{\ell+1} $. Finally, to see that $Ac'_{\ell+1} \oleq A\mrm{cl}_<^{\ell+1}(c'_{\ell+1})$ we show that $Ac'_{\ell+1} \hleq A\mrm{cl}_<^{\ell+1}(c'_{\ell+1})$. We start with $Ac'_{\ell+1}$ and define a $\mrm{HF}$-order $<'$ by adding $d^{0}_\ell <'z_{2}<'z_{3}\dots <' z_{n-2}<' d^1_\ell$. Crucially, each new element that we add has only one incidence with the previous ones, thus $<'$ is a $\mrm{HF}$-order. Finally, notice that $c'_{\ell+1}$ is only incident  to $z_{2}$ and $d^0_\ell$ in $ \mrm{cl}^{\ell+2}_<(c'_{\ell+1}) $, and also that no element in $\mrm{cl}^{\ell}_<(d^0_{\ell})$ is incident to any element in $\mrm{cl}^{\ell}_<(d^1_{\ell})$ by the choice of $d^1_\ell$. So using the inductive hypothesis it is easy to continue to construct the $\mrm{HF}$-order $ <' $. This finally shows that also $Ac'_{\ell+1} \oleq A\mrm{cl}_<^{\ell+1}(c'_{\ell+1})$.
\end{proof}	

The proof of the technical assumption \ref{technical_assumptions}(\hyperref[minimality_condition]{D2}) is similar to the case of Steiner systems.  In fact, as in the setting of Steiner systems, the key ingredient of the next proof is Lemma \ref{ngons:algebraic_lemma}. 

\begin{proposition}[{\ref{technical_assumptions}(\hyperref[minimality_condition]{D2})}]\label{ngons:minimality_condition}
	If $M \models T^+$ is $\aleph_1$-saturated, $A \oleq M$ is countable, and for every trivial extension $A\oleq Ab$ there is $b' \in M$ s.t. $Ab \cong_A Ab'\oleq M$, then for every minimal extension $A\oleq A\bar{c}$ there is $\bar{c}' \in M$ s.t. $A\bar{c} \cong_A A\bar{c}'\oleq M$.
\end{proposition}
\begin{proof}
	Let $M \models T^+$ and $A \oleq M$. Suppose $A\oleq A\bar{c}$ is a minimal strong extension. If $\bar{c}$ is a single element with no incidence with elements in $A$ then we are immediately done by the inductive assumption, thus we distinguish the following two cases.
	
	\smallskip
	\noindent \underline{Case 1}. $\bar{c}=c$ is an element incident with one element in $A$.
	\newline\noindent 
	Suppose w.l.o.g. that $c$ is a point incident with a line $d\in A$.  By assumption, we can find an element $d'\in M$ such that $Ad'\oleq M$ and $d'$ has no incidence with $A$. Additionally, we can choose $d'$ so that it has an appropriate sort, i.e., it is a point if $n$ is even and it is a line if $n$ is odd. Then, we have that $d(d,d'/M)\geq n+1$, for otherwise $M\setminus Ad'$ contains a short chain contradicting $Ad'\oleq M$. It follows that  $\{d,d',b_1,\dots,b_{n-2}  \}$, where $\bar{b}=(b_1,\dots,b_{n-2})$ is a clean arc between $d$ and $d'$ of the appropriate length, amalgamates with every finite subset from $M$. Since $M\models T^+$ there must be a clean arc $\bar{z}\in M^{n-2}$ with endpoints $d$ and $d'$.  Let $z_1$ be the first element from $\bar{z}$. Clearly $Ac \cong_A Az_1$. By \ref{ngons:algebraic_lemma} we have that $Ad'\bar{z}\oleq M$. Moreover, we have that $Az_1\oleq Ad'\bar{z}$, since we can obtain $Ad'\bar{z}$ by a sequence of one-element extensions from $Az_1$. It follows by transitivity that $Az_1\oleq M$. 
	
	\smallskip
	\noindent \underline{Case 2}. $\bar{c}$ is a clean arc with endpoints in $A$.
	\newline\noindent Let $a,b\in A$ be the endpoints of $\bar{c}$.  Since $A\oleq A\bar{c}$ we must have $d_{A}(a,b)\geq n+1$. Crucially, $A\oleq M$, thus $d_{M}(a,b)$ cannot be strictly less than $n$. Therefore, since $M\models T^+$, it follows that there is  a clean arc $\bar{d}\in (M\setminus A)^{<\omega}$ with endpoints $a$ and $b$. Clearly $A\bar{c}\cong_A A\bar{d}$ and, by \ref{ngons:algebraic_lemma}, it follows that $A\bar{d}\oleq M$.
\end{proof}

From the previous propositions together with Proposition~\ref{the_K_saturated_lemma} and Theorem \ref{prop_completeness} it follows that $T^+$ is complete. We next exhibit how the free completion process from \ref{free_algebraic_completion} behaves in the setting of generalised $n$-gons. The following construction is slightly different from the original one from Tits in \cite{tits}, but it is equivalent to it for connected partial $n$-gons.

\begin{definition}\label{ngons:free_completion2}
	Let $A$ be a partial $n$-gon. We let $A_0=A$ and, for every $i<\omega$, we let $A_{i+1}$ be obtained by adding a clean arc $a\, \vert\,z_1\, \vert\,z_2\, \vert\,\dots\, \vert\,z_{n-2}\, \vert\,b$ between any two elements $a,b\in A_i$ such that
	\begin{enumerate}[(a)]
		\item either $d(a,b/A_i)=n+1$,
		\item or $d(a,b/A_i)=\infty$ and $a,b$ are of the appropriate sort (i.e., $a,b$ have the same sort if $n$ is odd and have different sorts if $n$ is even).
	\end{enumerate}
	The structure $F(A)\coloneq\bigcup_{i < \omega}A_i$ is called the \emph{free completion} of $A$ and we say that $F(A)$ is freely generated over $A$. We say that $A$ is \emph{non-degenerate} if $F(A)$ is infinite, and \emph{degenerate} otherwise. 
\end{definition}

\begin{remark}
	The free completion of generalised $n$-gons  was introduced by Tits in \cite{tits} and differs from the one above in that it applies only to connected partial $n$-gons and thus does not include our clause (b) (see also \cite[Def.~2.4]{tent} and \cite[p.~755]{funk2}). Without (b), one has that $A$ generates an infinite model if it is connected and, either it contains a cycle of length $2n+2$, or has diameter $\geq n+2$ (cf. \cite[p.~755]{funk2}).  Given our slightly different definition, we can replace the requirement of connectedness with one saying that $|A|\geq 3$ and that $A$ contains elements of both sorts. If this is the case, clause (b) makes sure that in $A_1$ we introduce at least one clean arc, and thus $A_2$ is a connected partial $n$-gon. In particular, it follows that if $A$ contains at least $m_\mathcal{K}=2n+2$ elements of both sorts, then  $F(A)$ is infinite.
\end{remark}

The next proposition  simultaneously verifies Condition \ref{technical_assumptions}(\hyperref[extension]{D3}) and Assumption \ref{assumptions:no_prime}(\hyperref[F=C]{C5}). It follows that the theory $T^+$ of open generalised $n$-gons is stable and with$\ind=\ind^\otimes$.

\begin{proposition}[{\ref{technical_assumptions}(\hyperref[extension]{D3})}, {\ref{assumptions:no_prime}(\hyperref[F=C]{C5})}]\label{ngons:free_theory}
	Suppose $A\models T_\forall$, and $A$ is non-degenerate, then $F(A)\models T^+$, $A\oleq F(A)$, $|F(A)|=|A|+\aleph_0$ and:
	\begin{enumerate}[(a)]
		\item $F(A)$ is unique up to $A$-isomorphisms;
		\item $F(A)$ is $\oleq$-prime over $A$.
	\end{enumerate}
	In particular, it follows that if $A\cong B\onleq M\models T_n^+$ and $M$ is $\aleph_1$-saturated, then $F(A)\cong F(B)\preccurlyeq M$.
\end{proposition}
\begin{proof}
	As in \ref{free_theory_steiner}, it suffices to verify that $F(A)\models T^+$, provided that $A$ is non-degenerate.  Let $A'\subseteq_\omega F(A)$ and suppose that $A'\oleq A'\bar{c}$ is a minimal algebraic extension, i.e., $\bar{c}=(c_1,\dots,c_{n-2})$ is a clean arc with endpoints in $A$. Suppose that $A'\leq_{n-2}F(A)$, then it follows in particular that $d(a,b/F(A))\geq n+1$. Since $A$ is non-degenerate, there is a minimal stage $i<\omega$ such that  $\{a,b\}\subseteq A_i$ and   $d(a,b/A_i)= n+1$. By the definition of free completion, there is a clean arc $\bar{d}\in A^{<\omega}_{i+1}$ such that $A'\bar{c}\cong_A A'\bar{d}$. This shows that $F(A)$ satisfies the second family of axioms from \ref{the_theory}. Then, by reasoning as in the case of Steiner systems, one can also verify that $F(A)$ satisfies axioms \ref{the_theory}(1) and \ref{the_theory}(3), thus showing $F(A)\models T^+$.
\end{proof}

\begin{remark}\label{ngons:prime_super}
	In this section we have verified the conditions from \ref{context} and the assumptions from \ref{main_th}(\hyperref[completeness_axiom]{C1})-(\hyperref[the_K_homogeneous_lemma]{C2}) and \ref{ass:hf_closure}(\hyperref[the_hf_axiom]{C3}). This entails that $T^+$ is complete, stable and with $\ind=\ind^{\otimes}$. We do not provide full details for  the fact that $T^+$ is not superstable and it does not have a prime model. A notion of predimension for the generalised $n$-gons can be found in \cite{funk2} and it was used in \cite{tent} to prove that $T^+$ does not have a prime model. Similarly, a configuration satisfying Assumption \ref{assumption:no-superstability} can be essentially found in \cite[Const.~9.1]{paolini&hyttinen} for the case of projective planes. This can be adapted to all $n\geq 3$ by replacing each step in the construction by clean arcs. Essentially, such configurations can be also extracted from the configuration described in \cite[\S5.4]{HPQ} to verify that the class of free generalised $n$-gons satisfies the Construction Principle $\mrm{CP}(\mathbf{K},\ast)$ (cf. also Remark~\ref{remark:construction_principle} above).
\end{remark}

\subsection{$k$-nets}
We consider in this section $k$-nets, which are a third example of incidence structures of rank $2$ admitting a free algebraic completion. As we make explicit later, we axiomatise $k$-nets as structures without any parallelism relation, although in other presentations one has a notion of parallelism between lines. This indicates that, in a certain sense, $k$-nets are a hybrid form of incidence structures that lies between those with parallelism relations and those without. We refer the reader to \cite[p.~20]{pasini} and \cite{funk2} for more information on $k$-nets.  We proceed as in the previous cases and we verify all the conditions and assumptions from \ref{context}, \ref{main_th}(\hyperref[completeness_axiom]{C1})-(\hyperref[the_K_homogeneous_lemma]{C2}), \ref{ass:hf_closure}(\hyperref[the_hf_axiom]{C3}),  \ref{assumption:no-superstability_2}(\hyperref[CP]{C4}) and \ref{assumptions:no_prime}(\hyperref[F=C]{C5})-(\hyperref[delta-rank]{C7}). As in the former classes of examples, we then obtains the following theorem from the results in Section \ref{sec_general}.

\begin{theorem}\label{corollary:nets}
	Let $(\mathcal{K},\oleq)$ be the class of finite, open, partial  $k$-nets and let $T^+$ be defined as in \ref{the_theory}. Then $T^+$ is complete, stable, and with $\ind=\ind^{\otimes}$. Moreover, $T^+$ is not superstable, it does not have a prime model, it is not model complete and it does not eliminate quantifiers.
\end{theorem}

We recall the axiomatisation of $k$-nets, which we take from \cite[pp.~745-746]{funk2} with only small modifications.

\begin{definition}
	For any $k\geq 3$ we fix the two-sorted language $L$ with one sort for points, one sort for lines, one incidence symbol $\, \vert\,$ between points and lines, and $k$ many predicates $(P_i)_{i<k}$ for lines. We denote points by  $p_0,p_1,\dots$ and lines by $\ell_0,\ell_1,\dots$.  The theory of \emph{$k$-nets} is axiomatised by the following axioms:
	\begin{enumerate}[(1)]
		\item every line $\ell$ belongs to some unique $P_i$, for $i<k$;
		\item for all points $p$ and predicates $P_i$, with $i<k$, there is a unique line $\ell$  such that $p\, \vert\,\ell$ and $\ell\in P_i$;
		\item for any $i<j<k$ and lines $\ell\in P_i$ and $\ell'\in P_j$, there is a unique point $p$ such that $p\, \vert\,\ell$ and $p\, \vert\,\ell'$.
	\end{enumerate}
	The universal fragment of the theory of $k$-nets is axiomatised by axiom (1) together with the uniqueness requirement from axioms (2) and (3). We call a model of the universal theory of $k$-nets a \emph{partial $k$-net}.	
\end{definition}

We often refer to the predicates $(P_i)_{i<k}$ as the \emph{parallelism type} of lines. We notice that this is a slight abuse of notation, as this notion of parallelism crucially differs from the relation of parallelism from e.g., affine planes or Laguerre planes, as it does not play a role in the definition of hyperfree elements and we do not treat it as a local equivalence relation. As a matter of fact, in other treatments of nets in the literature such predicates are internalised as added sorts, e.g., in \cite{funk2}. We decided to distinguish these predicates from the two sorts of points and lines to emphasise the fact that nets are also an instance of rank-2 geometries. We refer the interested reader also to  \cite[Ch.~1.2]{pasini} for a similar presentation.

\begin{definition}\label{nets:def_hyperfree}
	For finite partial $k$-nets $A\subseteq B$, we say that an element $a\in A$ is \emph{hyperfree} in $B$ if it satisfies one of the following conditions:
		\begin{enumerate}[(i)]
			\item $a$ is a point incident with at most two lines from $B$; 
			\item $a$ is a line incident with at most one point from $B$.
		\end{enumerate}		
		We write $A \oleq B$ if every non-empty $C \subseteq B \setminus A$ contains an element which is hyperfree in $AC$ and we adopt the notions of \emph{open}, \emph{closed}, \emph{confined}, etc. specialising Definition~\ref{def_extensions}. We let $\mathcal{K}$ be the set of open partial finite $k$-nets  and we let $T_\forall$ be the universal theory of $\mathcal{K}$.  The previous definitions extend to infinite models of $T_\forall$ as in \ref{infinite_strong extensions}. The notion of hyperfree element in a $k$-net is exactly as in \cite[p.~753]{funk2}.
\end{definition}

The following remark makes it explicit what algebraic strong extensions are in the setting of nets. The subsequent Lemma \ref{Nets:algebraic_lemma}  essentially establishes the part \ref{condition:amalgam}(2) of the algebraic amalgamation property.

\begin{remark}\label{nets:algebraci.extensions}\label{nets:algebraic.extensions}
	In the case of $k$-nets, a minimal extensions $A\oleq Ac$ is \emph{algebraic} if $c$ is a point incident to two lines in $A$ or it is a line incident to some point from $A$.
\end{remark}

\begin{lemma}\label{Nets:algebraic_lemma}
	Let $A\oleq C\models T_\forall$ and let $c\in C$ be such that $A\oleq Ac$ is an algebraic strong extension. Then $Ac\oleq C$. 
\end{lemma}
\begin{proof}
	This is proved by reasoning exactly as in \ref{Steiner:algebraic_lemma}.
\end{proof}

Consider first the conditions from \ref{context}. As in the previous cases, conditions (\ref{condition:A}), (\ref{condition:B}), (\ref{condition:D}), (\ref{condition:E}), (\ref{condition:F}), (\ref{condition:K}) and (\ref{condition:L}) follow from our definition of the relation $\oleq$ for finite, open, partial $k$-nets. Condition \ref{context}(\ref{condition:C}) follows as in the case of Steiner systems. Condition (\ref{condition:I})  follow immediately by letting $\mathbf{X}_\mathcal{K}=\{1\}$, as it is easy to verify that the the minimal extensions for $k$-nets are exactly the one-element extensions, and Condition (\ref{condition:J}) follows by letting $\mathbf{n}_\mathcal{K}=2$. Conditions (\ref{condition:G}) and (\ref{condition:H}) are non-trivial, but we omit their proof as they are essentially as in the case of Steiner systems.  We consider in details Assumptions \ref{main_th}(\hyperref[completeness_axiom]{C1})-(\hyperref[the_K_homogeneous_lemma]{C2}) and \ref{ass:hf_closure}(\hyperref[the_hf_axiom]{C3}). Since we view $k$-nets in a signature without local equivalence relations, Condition \ref{ass:hf_closure}(\hyperref[the_hf_axiom]{C3}) can be verified as in the previous cases. In particular, also in this case the operator $\mrm{cl}_<$ is defined purely in terms of the incidence Gaifman closure.

\begin{definition}\label{nets:closure_operator} Let $A\subseteq B \models T_\forall^n$ and suppose $<$ is a $\mrm{HF}$-order of $B$ over $A$.  For all $C\subseteq B\setminus A$ we define $\mrm{cl}_<(C)\coloneqq \bigcup_{n<\omega }\mrm{cl}_<^{n}(C)$ inductively by letting $\mrm{cl}_<^0(C) \coloneqq C$ and $\mrm{cl}_<^{n+1}(C) \coloneqq \mrm{gcl}^{\mrm{i}}_{A\mrm{cl}_<^{n}(C)^\downarrow}(\mrm{cl}_<^{n}(C))$ for all $n<\omega$.
\end{definition}

\begin{proposition}[{\ref{ass:hf_closure}(\hyperref[the_hf_axiom]{C3})}]\label{nets:HF_orders}
	Every model $A\subseteq B\models T^n_\forall$ with an associated $\mrm{HF}$-order $<$ of $B$ over $A$ has a $\mrm{HF}$-closure operator $\mrm{cl}_<\coloneqq\bigcup_{n<\omega}\mrm{cl}^n_<$.
\end{proposition}
\begin{proof}
	This follows exactly as in Proposition~\ref{psedoplane:HF_orders}.
\end{proof}

To verify Assumption \ref{main_th}(\hyperref[completeness_axiom]{C1}) we proceed as in the previous cases and prove  \ref{technical_assumptions}(\hyperref[trivial_condition]{D1}) and \ref{technical_assumptions}(\hyperref[minimality_condition]{D2}). We add a third condition (c) to \ref{technical_assumptions}(\hyperref[trivial_condition]{D1}) to ease our inductive argument.

\begin{proposition}[{\ref{technical_assumptions}(\hyperref[trivial_condition]{D1})}]\label{Nets:trivial_condition}
	If $M \models T^+$, $A \subseteq M$ is finite, $<$ is a $\mrm{HF}$-order of $M$ and $A \oleq Ac$ is a trivial extension, then for all $n < \omega$ there is $ c'=c'_n \in M$ s.t.:
	\begin{enumerate}[(a)]
		\item $Ac \cong_A Ac'$;
		\item $Ac' \oleq A\mrm{cl}_<^n(c')$;
		\item $A < \mrm{cl}_<^{n+1}(c')$.
	\end{enumerate}
\end{proposition}
\begin{proof} By induction on $n < \omega$ we prove that for every finite $A \subseteq M$ and $A \oleq Ac$ we can find $c'_n$ satisfying the conditions (a)-(c). If $n = 0$, then it is easy to see that we can find find $c'_1$ as wanted. So suppose the inductive hypothesis and let $c$ be such that $A \oleq Ac$ is a trivial extension, we want to define $c'_{n+1}$ so that (a)-(c) are satisfied for it. 
	
	\smallskip	
	\noindent \underline{Case 1}. $c$ is a line (of parallelism type $i<k$).
	\newline Firstly, since $k\geq 3$, we notice there are two indices $i_0,i_1$ both different from $i$. By inductive hypothesis we can find a line $d^0_n\in P_{i_0}$ such that $A < \mrm{cl}_<^{n+1}(d^0_n)$ and  $Ad^0_n \oleq A\mrm{cl}_<^n(d^0_n)$. Also, we can find a further line $d^1_n \in P_{i_1}$ such that $A\mrm{cl}_<^{n+1}(d^0_n) < \mrm{cl}_<^{n+1}(d^1_n)$ and  $A\mrm{cl}_<^{n+1}(d^0_n)d^1_n \oleq A\mrm{cl}_<^{n+1}(d^0_n)\mrm{cl}_<^n(d^1_n)$. Consider the unique point $p$ incident to both $d^0_n$ and $d^1_n$, then since $A\mrm{cl}_<^{n+1}(d^0_n) < \mrm{cl}_<^{n+1}(d^1_n)$ it follows that $p>d^1_n>d^0_n$. 
	\newline Then, since  $<$ is a $\mrm{HF}$-order and $M\models T^+$, there is a line $c'_{n+1}\in P_i$ such that $c'_{n+1}>p$ and $c'_{n+1}\, \vert\,p$. Then we have that:
	\begin{align*}
	\mrm{cl}_<^{n+1}(c'_{n+1})&=\mrm{cl}_<^{n}(d^0_n)\cup\mrm{cl}_<^{n}(d^1_n)\cup \{p, c'_{n+1}\}; \\
	\mrm{cl}_<^{n+2}(c'_{n+1})&=\mrm{cl}_<^{n+1}(d^0_n)\cup\mrm{cl}_<^{n+1}(d^1_n)\cup \{p, c'_{n+1}\}.
	\end{align*}
	It follows by construction that $A<\mrm{cl}_<^{n+2}(c_{n+1})$ and $Ac \cong_A Ac'_{n+1}$. We next define a $\mrm{HF}$-order  $<'$ witnessing $Ac'_{n+1}\oleq A\mrm{cl}_<^{n+1}(c_{n+1})$. It suffices to start from $Ac'_{n+1}$ and let $c'_{n+1}<'p<'d^1_n<'d^0_n$ (by the choice of our elements this is a $\mrm{HF}$-order). Notice then that $ c'_{n+1} $ is incident only to $p$ in $\mrm{cl}_<^{n+1}(c'_{n+1})$, $p$ is incident only to $d^0_n,d^1_n$ in $\mrm{cl}_<^{n+1}(c'_{n+1})$, and also by our choice of $d^1_n$ no element in $\mrm{cl}_<^{n}(d^1_n)$ is incident to elements in $ \mrm{cl}_<^{n}(d^0_n) $. Thus we can use the inductive hypothesis to continue the construction of a $\mrm{HF}$-order and obtain that $Ac'_{n+1}\oleq A\mrm{cl}_<^{n+1}(c'_{n+1})$.
	
	\smallskip 
	\noindent \underline{Case 2}. $c$ is a point.
	\newline We proceed analogously to the previous case. By inductive hypothesis we can find a line $d^0_n\in P_{0}$ such that $A < \mrm{cl}_<^{n+1}(d^0_n)$ and  $Ad^0_n \oleq A\mrm{cl}_<^n(d^0_n)$. Also, we can find a further line $d^1_n \in P_{1}$ such that $A\mrm{cl}_<^{n+1}(d^0_n) < \mrm{cl}_<^{n+1}(d^1_n)$ and  $A\mrm{cl}_<^{n+1}(d^0_n)d^1_n \oleq A\mrm{cl}_<^{n+1}(d^0_n)\mrm{cl}_<^n(d^1_n)$. Let  $c'_{n+1}$ be the unique point incident to both $d^0_n$ and $d^1_n$, then since $A\mrm{cl}_<^{n+1}(d^0_n) < \mrm{cl}_<^{n+1}(d^1_n)$ it follows that $c'_{n+1}>d^1_n>d^0_n$. Clearly we have that $\mrm{cl}_<^{n+1}(c'_{n+1})=\mrm{cl}_<^{n}(d^0_{n})\cup\mrm{cl}_<^{n}(d^1_{n})\cup\{c'_{n+1}\}$. From this fact and the induction hypothesis, we can  reason as in Case 1 and see that  $Ac'_{n+1}\oleq A\mrm{cl}_<^n(c_{n+1})$. The facts that $A<\mrm{cl}_<^{n+2}(c_{n+1})$ and $Ac \cong_A Ac'_{n+1}$ are immediately true by construction and the inductive hypothesis.
\end{proof}

\begin{proposition}[{\ref{technical_assumptions}(\hyperref[minimality_condition]{D2})}]\label{Nets:minimality_condition}
	If $M \models T^+$ is $\aleph_1$-saturated, $A \oleq M$ is countable, and for every trivial extension $A\oleq Ab$ there is $b' \in M$ s.t. $Ab \cong_A Ab'\oleq M$, then for every minimal extension $A\oleq Ac$ there is $c' \in M$ s.t. $Ac \cong_A Ac'\oleq M$.
\end{proposition}
\begin{proof}
	This follows by reasoning exactly as in the case of Steiner systems \ref{Steiner:minimality_condition}.
\end{proof}

By Proposition~\ref{the_K_saturated_lemma} and Theorem~\ref{prop_completeness} it already follows from the previous propositions that the theory $T^+$ of infinite, open $k$-nets is complete. To show that also \ref{main_th}(\hyperref[the_K_homogeneous_lemma]{C2}) holds, and thus establish the stability of $T^+$, we verify \ref{technical_assumptions}(\hyperref[extension]{D3}) by considering the free completion process for the setting of $k$-nets. In turn, this also verifies Assumption \ref{assumptions:no_prime}(\hyperref[F=C]{C5}). It is straightforward to verify that the construction given by \ref{general_free_amalgam} is equivalent to the following one.

\begin{definition}\label{nets:free_completions}
	Let $A$ be a partial $k$-net. We let $A_0=A$ and  for every $n<\omega$ we define $A_{2n+1}, A_{2n+2}$ as follows:
	\begin{enumerate}[(a)]
		\item for any two lines $\ell_0\in P_i$ and $\ell_1\in P_j$ in $A_n$ with $i< j<k$ and not incident to any common point, we add a new point $p$ in $A_{2n+1}$ incident only to $\ell_0,\ell_1$;
		\item for every point $p\in A_{2n+1}$ and every $i<k$ such that $p$ is not incident to any line of parallelism type $i$ in $A_{2n+1}$, we add a new line $\ell\in A_{2n+2}$ of parallelism type $i$ which is incident only to $p$.
	\end{enumerate}	
	The structure $F(A):=\bigcup_{i < \omega}A_i$ is called the \emph{free completion of $A$} and we say that $F(A)$ is freely generated over $A$. We say that $A$ is \emph{non-degenerate} if $F(A)$ contains infinitely many points and lines, and \emph{degenerate} otherwise. 
\end{definition}

\begin{remark}\label{nets:non-degenerate}
	As remarked in \cite[p.~13]{barlotti2}, the only degenerate partial $k$-nets are the empty set $\emptyset$, the configuration given by a point incident to $k$ many lines, and any set of lines of the same parallelism type. Letting $\mathbf{m}_{\mathcal{K}}=2$, it follows that any $A\models T_\forall$ containing at least $\mathbf{m}_\mathcal{K}$ many elements of each sort is non-degenerate.
\end{remark}

\begin{proposition}[{\ref{technical_assumptions}(\hyperref[extension]{D3})}, {\ref{assumptions:no_prime}(\hyperref[F=C]{C5})}]\label{free_theory_nets}\label{Nets:the_auto_axiom}
	Suppose $A\models T_\forall$, and $A$ is non-degenerate, then $F(A)\models T^+$, $A\oleq F(A)$, $|F(A)|=|A|+\aleph_0$ and:
	\begin{enumerate}[(a)]
		\item $F(A)$ is unique up to $A$-isomorphisms;
		\item $F(A)$ is $\oleq$-prime over $A$.
	\end{enumerate}
	In particular, it follows that if $A\cong B\onleq M\models T_n^+$ and $M$ is $\aleph_1$-saturated, then $F(A)\cong F(B)\preccurlyeq M$.
\end{proposition}
\begin{proof}
	The proof proceeds analogously to the case of Steiner systems.
\end{proof}

So far, we have verified all the conditions from \ref{context}, \ref{main_th}(\hyperref[completeness_axiom]{C1})-(\hyperref[the_K_homogeneous_lemma]{C2}) and \ref{ass:hf_closure}(\hyperref[the_hf_axiom]{C3}). It remains to consider the assumptions from  \ref{assumption:no-superstability_2}(\hyperref[CP]{C4}) and \ref{assumptions:no_prime}(\hyperref[F=C]{C5})-(\hyperref[delta-rank]{C7}). We point out that Assumption \ref{assumptions:no_prime}(\hyperref[F=C]{C5}) follows immediately from \ref{free_theory_nets}, that Assumption \ref{assumptions:no_prime}(\hyperref[hopf]{C6}) can be proved by a similar argument as in the case of Steiner systems, and that Assumption \ref{assumptions:no_prime}(\hyperref[delta-rank]{C7}) can be verified using the definition of a predimension $\delta(A)$ from \cite[Def.~11]{funk2}. We conclude this section by exhibiting a configuration that witnesses that the technical assumptions in \ref{assumption:no-superstability}(\hyperref[technical:D4]{D4}) hold, thus verifying also \ref{assumption:no-superstability_2}(\hyperref[CP]{C4}). In turn, this entails that $T^+$ is not superstable and concludes the proof of Theorem~\ref{corollary:nets}.

\begin{proposition}[{\ref{assumption:no-superstability}(\hyperref[technical:D4]{D4})}]
	There is a configuration $C\in \mathcal{K}$ satisfying the conditions from \ref{assumption:no-superstability}.
\end{proposition}
\begin{proof} 
We define a configuration $C$ as follow. Let $C=\{c_i : 0\leq i \leq 18\}$, where $c_0,c_1,c_8,c_9,c_{10},c_{14},c_{16},c_{18}$ are points, $c_2, c_5, c_{12}$ are lines of parallelism type $P_0$, $c_3,c_6,c_{13},c_{15}$ are lines of parallelism type $P_1$, and $c_4,c_7,c_{11},c_{17}$ are lines of parallelism type $P_2$. Notice that then, since $k\geq 3$,  this construction thus applies to all $k$-nets. Finally, the incidence relations between elements of $C$ are specified by the following table.

\begin{table}[H]
	\begin{tabular}{|c||c|c|c|c|c|c|c|c|c|c|c|}\hline
		& $c_2$ &$c_3$  &$c_4$  & $c_5$ & $c_6$ & $c_7$ & $c_{11}$ & $c_{12}$ & $ c_{13}$ & $c_{15}$ &$c_{17}$ \\ \hline\hline
		$c_0$		& $\times$& $\times$ &$\times$  &  &  &  &  &  &  &					&\\\hline
		$c_1$		&  &  &  & $\times$ &$\times$  &$\times$  &  &  &  &				&	\\\hline
		$c_8$		& $\times$ &  &  &  &$\times$  &  &$\times$  &  &  &					&\\\hline
		$c_9$		&  & $\times$ &  &  &  &$\times$  &  & $\times$ &  &					&\\\hline
		$c_{10}$	&  &  & $\times$ & $\times$ &  &  &  &  &  $\times$&				&\\\hline
		$c_{14}$	&  &  &  &  &  &  &	$\times$ & $\times$ &  & $\times$&\\\hline
		$c_{16}$	& $\times$  &  &  &  &  &  &	 &  &  &$\times$ &$\times$\\\hline
		$c_{18}$	&  &  &  &  &  &  &	 &  & $\times$ & & $\times$ \\\hline
	\end{tabular}
\end{table}
\noindent	Then $C$ is a partial $k$-net and the order $c_0<c_1<\dots<c_{18}$ is a $\mrm{HF}$-order, which means that $C\in \mathcal{K}$. As the other requirements from  \ref{assumption:no-superstability} are easily verified, it follows that $C$ is as desired.	
\end{proof}

\subsection{Affine planes}\label{section:affine_planes}

We apply in this section our general framework to the theory of affine planes, which essentially axiomatise the incidence properties of points and lines in the Euclidean plane (see e.g., \cite{buekenhout_cameron} for a general reference on the matter). Although this case is very similar to the setting of projective planes, affine planes are particularly interesting as they also involve a notion of parallelism between lines, in addition to the relation of incidence between lines (or blocks) and points. Essentially, this is the first case that we are considering where we need the full generality of Definition~\ref{def:interior_closure}, i.e., where we may have that $\mrm{gcl}_M(A)\neq \mrm{gcl}^{\mrm{i}}_M(A)$. We provide details of all proofs where the presence of the relation of parallelism makes a difference. The results from this section then establish the following theorem.

\begin{theorem}\label{corollary:affine}
	Let $(\mathcal{K},\oleq)$ be the class of finite, open, partial  affine planes and let $T^+$ be defined as in \ref{the_theory}. Then $T^+$ is complete, stable, and with $\ind=\ind^{\otimes}$. Moreover, $T^+$ is not superstable, it does not have a prime model, it is not model complete and it does not eliminate quantifiers.
\end{theorem}

We start by recalling the axiomatisation of affine planes from \cite[p.~744]{funk2} and by introducing the class $(\mathcal{K},\oleq)$ in this setting.

\begin{definition}
	Let $L$ be the two-sorted language where the symbols $p_0,p_1,\dots$ denote \emph{points} and the symbols $\ell_0,\ell_1,\dots$ denote \emph{lines}, the symbol $p\, \vert\,\ell$ means that the point $p$ is \emph{incident} with the line $\ell$, and the symbol $\ell_0\parallel \ell_1$ means the two lines $\ell_0$ and $\ell_1$ are \emph{parallel}.  The theory of \emph{affine planes} consists of the following axioms:
	\begin{enumerate}[(1)]
		\item the parallelism relation $\parallel$ is an equivalence relation;
		\item every two distinct points $p_0,p_1$ have a unique common incident line $p_0\lor p_1$;
		\item every two non-parallel lines $\ell_0,\ell_1$ have a unique common incident point $\ell_0\wedge \ell_1$;
		\item given a point $p$ and a line $\ell$, there is a unique line $\ell'$ s.t.  $p\, \vert\,\ell'$ and $\ell'\parallel \ell$.
	\end{enumerate}
	The universal fragment of the theory of affine planes is axiomatised by the axiom (1) saying that  parallelism is an equivalence relation, and by the formulas saying that the line incident to two given points, the point incident to two non-parallel lines, and the line incident to a point and parallel to a line are all unique. We refer to models of this universal theory as \emph{partial affine planes}.
\end{definition}

\begin{definition}\label{affine:def.hyperfree}
	For finite partial affine planes $A\subseteq B$, we say that an element $a\in A$ is \emph{hyperfree} in $B$ if it satisfies one of the following conditions:
		\begin{enumerate}[(i)]
			\item $a$ is a point incident with at most two lines from $B$; 
			\item $a$ is a line incident with at most one point from $B$;
			\item $a$ is a line incident with exactly two points from $B$ and  not parallel to any other line in $B$.
		\end{enumerate}
	We then write $A \oleq B$ if every non-empty $C \subseteq B \setminus A$ contains an element which is hyperfree in $AC$. We then adopt the notions of \emph{open}, \emph{closed}, \emph{confined}, etc. specialising Definition~\ref{def_extensions}. We let $\mathcal{K}$ be the set of open partial finite affine planes and we let $T_\forall$ be the universal theory of $\mathcal{K}$.  The previous definitions extend to infinite models of $T_\forall$ as in \ref{infinite_strong extensions}. The notion of hyperfree element in partial affine planes is from \cite[p.~752]{funk2}.
\end{definition}

As in the previous cases, we start by recalling what algebraic strong extensions are in the setting of affine planes, and by proving the crucial Lemma \ref{Affine:algebraic_lemma}. We provide details of its proof as this is the first case where parallelism relations play a role. This also establishes part \ref{condition:amalgam}(2) of the algebraic amalgamation property.

\begin{remark}\label{Affine:algebraic_extension}
	In partial affine planes an extension $A\oleq Ac$ is algebraic if:
	\begin{enumerate}[(i)]
		\item $c$ is a point incident with two lines from $A$;
		\item $c$ is a line incident with two points from $A$ and parallel to no other line in $A$;
		\item $c$ is a line incident with a point in $A$ and parallel to some line in $A$.
	\end{enumerate}
\end{remark}

\begin{lemma}\label{Affine:algebraic_lemma}
	Let $A\oleq C\models T_\forall$ and let $c\in C$ be such that $A\oleq Ac$ is an algebraic strong extension. Then $Ac\oleq C$. 
\end{lemma}
\begin{proof}
	Consider a finite subset $D\subseteq C$ such that $Ac\subsetneq D$, we need to show that $D\setminus Ac$ contains some hyperfree element. Now, since $A\oleq C$ and $D\subseteq C$, there is an element $d\in D\setminus A$ hyperfree in $D$.	If $c\neq d$ then $d$ also witnesses that $D\setminus Ac$ contains a hyperfree element and we are done. Thus suppose that $c$ is the only element in $D\setminus A$ to be hyperfree in $D$. Since $c$ is algebraic over $A$, it forms one of the configuration from Remark \ref{Affine:algebraic_extension}.
	
	\smallskip
	\noindent \underline{Case 1}. $c$ is a point incident with two lines from $A$.
	\newline Since $c$ is hyperfree in $D$, it is not incident to lines in $D\setminus A$. It follows that an element $d'\in D\setminus Ac$ is hyperfree in $D$ if and only if it is hyperfree in $D\setminus \{c \}$. Since by assumption $c$ is the unique element  in $D\setminus A$ to be hyperfree in $D$, it follows that  $D\setminus Ac$ contains no element hyperfree in $D$, contradicting  $A\oleq C$.
	
	\smallskip
	\noindent \underline{Case 2}. $c$ is a line incident with two points from $A$. 
	\newline Then $c$ is neither incident to points in $D\setminus A$, nor parallel to any other line from $D$. Our claim then follows exactly as in Case 1.
	
	\smallskip
	\noindent \underline{Case 3}. $c$ is a line incident to some point in $A$ and parallel to some line $\ell\in A$. 
	\newline Since $c$ is hyperfree in $D$, it is not incident to any point from  $D\setminus A$. Moreover, if $d'\in D\setminus A$ is such that $d'\parallel c$ then $d'\parallel l$.  As in the previous cases, it follows that an element $d'\in D\setminus Ac$ is hyperfree in $D$ if and only if it is hyperfree in $D\setminus \{c \}$. Since by assumption $c$ is the unique element  in $D\setminus A$ to be hyperfree in $D$, it follows that  $D\setminus Ac$ contains no element hyperfree in $D$, contradicting  $A\oleq C$.
\end{proof}

We next consider the conditions from \ref{context}. As in the previous cases, conditions (\ref{condition:A}), (\ref{condition:B}), (\ref{condition:D}), (\ref{condition:E}) (\ref{condition:F}) and (\ref{condition:L}) follow immediately from the definition of the relation $\oleq$ for finite, open, partial affine planes. Condition (\ref{condition:C}) follows exactly as in the case of Steiner systems. Condition (\ref{condition:I})  follows immediately for $\mathbf{X}_\mathcal{K}=\{1\}$ and Condition~(\ref{condition:J}) follows by letting $\mathbf{n}_\mathcal{K}=2$. Condition~(\ref{condition:K}) follows by the definition of hyperfree elements in affine planes together with the definition of Gaifman closure \ref{def:interior_closure}. It remains to verify Conditions~(\ref{condition:G}) and (\ref{condition:H}). We first consider the algebraic amalgamation property.

\begin{definition}\label{parallelsim_classes}
	Let $A$ be a partial affine plane and $\ell\in A$ be a line. We say that $\ell$ forms a \emph{trivial parallelism class in $A$} if $\ell\parallel \ell'$, $\ell'\in A$ entail that $\ell=\ell'$. Also, we denote by $P(A)$ any subset of $A$ containing exactly one line for every parallelism class, i.e., for every $\ell\in A$ there is exactly one $\ell'\in P(A)$ such that $\ell\parallel \ell'$.
\end{definition}

\begin{proposition}[\ref{context}(\ref{condition:G})]\label{Affine:conditionG}
	The class $(\mathcal{K},\oleq)$ has the algebraic amalgamation property.
\end{proposition}
\begin{proof}
	Consider $A,B,C\in \mathcal{K}$ with $A\oleq B$ minimal, $A\leq_{|B\setminus A|} C$ and $B\cap C=A$. We distinguish the following cases. By Remark~\ref{Affine:algebraic_extension} the following arguments establish the algebraic amalgamation property.	
	
	\smallskip
	\noindent \underline{Case 1}. $B=A\cup \{\ell\}$ where $\ell$ is a line incident with at most one  point from $A$. 
	\newline If $\ell$ is not incident with any point from $A$, or if it forms a trivial parallelism class in $B$, then $B\otimes_A C\in \mathcal{K}$. So suppose that $p\, \vert\,\ell$ and $\ell\parallel \ell_0$ for some $p,\ell_0\in A$.
	
	\noindent \underline{Case 1.1}. $C$ does not contain a line parallel with $\ell_0$ and incident with $p$.
	\newline Then we immediately have $B\otimes_A C\in \mathcal{K}$.		
	
	\noindent \underline{Case 1.2}. $C$ contains a line $\ell'$ parallel with $\ell_0$ and incident to $p$.
	\newline Then since $A\leq_{|B\setminus A|} C$, it follows in particular that $\ell'$ is incident only with $p$ and so the  map $f$ defined by letting $f\restriction A=\mrm{id}_A$ and $f(\ell)=\ell'$ is an embedding. Moreover, if we additionally have that $A\oleq C$, then since $A\oleq A\ell'$ is algebraic, it follows from \ref{Affine:algebraic_lemma} that $f(B)\oleq C$.
	
	\smallskip
	\noindent \underline{Case 2}. $B=A\cup \{\ell\}$ where $\ell$ is a line incident with exactly two points $p_0,p_1$ from $A$ and parallel to no line in $A$.
	
	\noindent \underline{Case 2.1}.  $C$ does not contain a line incident with $p_0$ and $p_1$.
	\newline In this case we immediately have $B\otimes_A C\in \mathcal{K}$.	
	
	\noindent \underline{Case 2.2}. $C$ contains a line $\ell'$ incident with $p_0$ and $p_1$.
	\newline Since $A\leq_{|B\setminus A|} C$ it follows that $\ell'$ is incident only with $p_0,p_1$ and parallel with no line in $A$. It follows that the  map $f$ such that $f\restriction A=\mrm{id}_A$ and $f(\ell)=\ell'$ is an embedding. Moreover, if we additionally have that $A\oleq C$, then since $\ell'$ is algebraic over $A$, it follows by \ref{Affine:algebraic_lemma} that $f(B)\oleq C$.
	
	\smallskip
	\noindent \underline{Case 3}. $B=A\cup \{p\}$ where $p$ is a point incident with at most two lines  from $A$.
	\newline If $p$ is incident with at most one line, then $B\otimes_A C\in \mathcal{K}$, so we suppose $p$ is incident to exactly two lines $\ell_0,\ell_1\in A$.
	
	\noindent \underline{Case 3.1}.   $C$ does not contain a point incident with $\ell_0$ and $\ell_1$.
	\newline In this case we immediately have $B\otimes_A C\in \mathcal{K}$.	
	
	\noindent \underline{Case 3.2}. $C$ contains a point $p'$ incident with $\ell_0$ and $\ell_1$
	\newline Since $A\leq_{|B\setminus A|} C$ it follows  that $p'$ is incident only to such lines in $A$. Then the  map $f$ such that $f\restriction A=\mrm{id}_A$ and $f(p)=p'$ is an embedding of $B$ into $C$. If we have that $A\oleq C$, then since $A\oleq Ap'$ is algebraic, it follows from \ref{Affine:algebraic_lemma} that $f(B)\oleq C$.
\end{proof} 

Next, we give details of the proof of \ref{context}(\ref{condition:H}), since this requires to take parallelism relations into account.

\begin{proposition}[\ref{context}(\ref{condition:H})]\label{Affine:conditionH}
	If $A\subseteq B \in \mathcal{K}$ and $A \noleq B$ is confined, then there is $n < \omega$ such that every $C \in \mathcal{K}$ contains at most $n$ many disjoint copies of $B$ over $A$.
\end{proposition}
\begin{proof}
	Suppose $A \noleq B$ is confined, then every point $p\in B\setminus A$ is incident with at least three lines from $B$ and every line $\ell\in B\setminus A$ is incident with at least three points from $B$, or it is incident with two points from $B$ while being parallel to some other line from  $B$.  For every line $\ell\in B$ we let $P_\ell$ be the set of points in $A$ that are incident with $\ell$, and for every point $p\in B$ we let $L_p$ be the set of lines in $A$ that are incident with $p$. Suppose $C\in \mathcal{K}$ contains at least three copies $(B_i)_{i<3}$ of $B$ isomorphic over $A$, then we claim that the set 
	\[ D= \bigcup_{p\in B_0}L_p\cup \bigcup_{\ell\in B_0}P_l \cup \bigcup_{i<3}(B_i\setminus A) \]
	witnesses that $C$ is not open. We need to consider the following elements.
	
	\smallskip
	\noindent  \underline{Case 1}. $\ell\in L_{p_0}$ for some $p_0\in B_0$.
	\newline  By definition $p_0\, \vert\,\ell$ where $p_0\in B_0$. Since $B_0,B_1,B_2$ are all isomorphic over $A$, there are $p_1\in B_1$ and $p_2\in B_2$ such that $p_1\, \vert\,\ell$ and $p_2\, \vert\,\ell$. Thus $\ell$ has three incidences in $D$ and it is not hyperfree.
	
		\smallskip
	\noindent  \underline{Case 2}. $p\in P_{\ell_0}$ for some $\ell_0\in B_0$.
	\newline  By reasoning exactly as in Case 1 it follows that $p$ is incident to three lines in $D$ and thus it is not hyperfree.
	
		\smallskip
	\noindent  \underline{Case 3}.  $p\in B_i\setminus A$ is a point, for $i<3$.
	\newline  By assumption $p$ is incident to at least three lines in $B_i$, and thus in  $L_p\cup (B_i\setminus A)\subseteq D$. Thus $p$ is not hyperfree in $D$.
	
		\smallskip
	\noindent  \underline{Case 4}.  $\ell\in B_i\setminus A$ is a line, for $i<3$.
	\newline Since $\ell$ is not hyperfree in $B_i$ we have three possible subcases.
	
		\smallskip
	\noindent  \underline{Case 4.1}. $\ell$ is incident to three points in $B_i$.
	\newline Then, by definition of the set $P_\ell$, it follows that $\ell$ is incident to three elements in $P_\ell \cup (B_i\setminus A)\subseteq D$. It follows that $\ell$ is not hyperfree in $D$.
	
		\smallskip
	\noindent  \underline{Case 4.2}. $\ell$ is incident to two points in $B_i$, and parallel to another line $\ell'\in B_i\setminus A$.
	\newline Reasoning as in case 4.1, it follows that $\ell$ is incident to two points also in $P_\ell \cup (B_i\setminus A)\subseteq D$. Then, since $\ell'\in (B_i\setminus A)\subseteq D$, we obtain that $\ell$ is incident to two points in $D$ while being parallel to some line $\ell'\in D$, showing that $\ell$ is not hyperfree in $D$.
	
	\smallskip
	\noindent  \underline{Case 4.3}. $\ell$ is incident to two points in $B_i$ while being parallel to some line $\ell'\in A$.
	\newline Assume w.l.o.g. that $i=0$. Reasoning as in case 4.1, it follows that $\ell$ is incident to two points also in $P_\ell \cup (B_0\setminus A)\subseteq D$. Now, since $B_0$ and $B_1$ are isomorphic over $A$, it follows that there is a line $\ell''$ such that $\ell\parallel \ell'$ and $\ell''\parallel \ell'$. By transitivity of the parallelism relation it follows that $\ell\parallel \ell''\in (B_1\setminus A)\subseteq D$. It follows that $\ell$ is not hyperfree in $D$.  This completes our proof.
\end{proof}

We next consider Assumptions \ref{main_th}(\hyperref[completeness_axiom]{C1})-(\hyperref[the_K_homogeneous_lemma]{C2}) and \ref{ass:hf_closure}(\hyperref[the_hf_axiom]{C3}). Since affine planes also comprise the relation of parallelism $\parallel$ in their language, we provide details for the definition of the $\mrm{HF}$-closure operator $\mrm{cl}_<\coloneqq \bigcup_{n<\omega}\mrm{cl}^n_<$.

\begin{definition}
	Let $A\subseteq B \models T_\forall^n$ and suppose $<$ is a $\mrm{HF}$-order of $B$ over $A$. For every line $\ell\in B$ we let $\mrm{H}_{<}(\ell)$ be a family of lines $(\ell_i)_{i\in I}$ all parallel to $\ell$ and s.t.:
	\begin{enumerate}[(1)]
		\item $\mrm{H}_<(\ell')=\mrm{H}_<(\ell)$ for all lines $\ell'\parallel \ell$;
		\item  for all non-empty $J\subseteq I$ the set $\{\ell_j : j\in J\}$ has a maximal element with respect to the ordering $<$;
		\item for all $\ell'\parallel \ell$ with $\ell'<\ell$ there is some $\ell_i\in \mrm{H}_{<}(\ell)$ with $\ell_i=\ell'$ or $\ell_i<\ell'$.
	\end{enumerate}
\end{definition}

\begin{definition}\label{affine:closure_operator} Let $A\subseteq B \models T_\forall^n$ and suppose $<$ is a $\mrm{HF}$-order of $B$ over $A$. For all $C\subseteq B\setminus A$ we define $\mrm{cl}_<(C)\coloneqq \bigcup_{n<\omega }\mrm{cl}_<^{n}(C)$ by letting $\mrm{cl}_<^0(C) \coloneqq C$ and, for all $n<\omega$, we define:
\begin{enumerate}[(1)]
	\item $\mrm{gcl}^{\mrm{i}}_{A\mrm{cl}_<^{n}(C)^\downarrow}(\mrm{cl}_<^{n}(C)) \subseteq \mrm{cl}_<^{n+1}(C)$;
	\item if $\ell\in \mrm{cl}_<^{n}(C)$ and $\mrm{H}_<(\ell)$ has a least element $\ell'$, then $\ell'\in \mrm{cl}_<^{n+1}(C)$;
	\item if $\ell\in \mrm{cl}_<^{n}(C)$, there is $\ell_0\parallel \ell$ with $\ell_0<\ell$, and $\ell_1$ is the greatest element in $\mrm{H}_<(\ell)$ which does not belong to  $\mrm{cl}_<^{n}(C)$, then $\ell_1\in \mrm{cl}_<^{n+1}(C)$;
	\item no other element belongs to $\mrm{cl}_<^{n+1}(C)$.
\end{enumerate}
\end{definition}

\begin{proposition}[{\ref{ass:hf_closure}(\hyperref[the_hf_axiom]{C3})}]\label{affine:HF_orders}
	Every model $A\subseteq B\models T^n_\forall$ with an associated $\mrm{HF}$-order $<$ of $B$ over $A$ has a $\mrm{HF}$-closure operator $\mrm{cl}_<\coloneqq\bigcup_{n<\omega}\mrm{cl}^n_<$.
\end{proposition}
\begin{proof}
	Let $<$ be a $\mrm{HF}$-order of $B$ over $A$ and let $\mrm{cl}_<(C)\coloneqq \bigcup_{n<\omega }\mrm{cl}_<^{n}(C)$ be the operator defined in \ref{affine:closure_operator}. It is straightforward to verify that  $\mrm{cl}_<(C)\coloneqq \bigcup_{n<\omega }\mrm{cl}_<^{n}(C)$ satisfies conditions (1)-(2) from Definition~\ref{def_closure_operator}, whence it suffices to verify that 
	\[A\mrm{cl}_<(CD)=A\mrm{cl}_<(C)\otimes_{A\mrm{cl}_<(C)\cap A\mrm{cl}_<(D)} A\mrm{cl}_<(D)\oleq B\]
	holds for all $C\subseteq B\setminus A$.
	
	\medskip
	\noindent First, notice that if there is a relation of incidence between some $c\in \mrm{cl}_<(C)$ and $d\in \mrm{cl}_<(D)$, then since $<$ is a linear order it follows from Definition~\ref{affine:closure_operator} that either $c\in \mrm{cl}_<(D)$ or $d\in \mrm{cl}_<(C)$. Similarly, suppose there is a relation of parallelism between two lines $\ell_0\in \mrm{cl}_<(C)$ and $\ell_1\in \mrm{cl}_<(D)$ and assume without loss of generality that $\ell_0<\ell_1$. Then from Definition~\ref{affine:closure_operator} it follows that either $\ell_0\in \mrm{cl}_<(D)$ or there is some $\ell_k\in \mrm{H}_<(\ell_0)=\mrm{H}_<(\ell_1)$ such that $\ell_k\in \mrm{cl}_<(C)\cap \mrm{cl}_<(D)$. It follows that  $A\mrm{cl}_<(CD)=A\mrm{cl}_<(C)\otimes_{A\mrm{cl}_<(C)\cap A\mrm{cl}_<(D)} A\mrm{cl}_<(D)$.

	\medskip
	\noindent It remains to show that $A\mrm{cl}_<(CD)\oleq B$. By Corollary \ref{equivalence_HFo.ordering} it suffices to prove that $<\restriction B\setminus(A\mrm{cl}_<(CD)) $ is a $\mrm{HF}$-order of $B$ over $A\mrm{cl}_<(CD)$. Condition \ref{def_HF_order}(\hyperref[HF1]{H1}) is immediate to verify. We show that  Condition \ref{def_HF_order}(\hyperref[HF2]{H2}) holds as well. 
	
	\smallskip
	\noindent Let $b\in B\setminus A\mrm{cl}_<(CD)$ and suppose towards contradiction that $A\mrm{cl}_<(CD)E\noleq A\mrm{cl}_<(CD)Eb$ for some finite $E\subseteq b^\downarrow \setminus \{b\}$. By Definition \ref{infinite_strong extensions}, there is a finite subset $C_0\subseteq \mrm{cl}_<(CD)$ such that $AC_0E\noleq AC_0Eb$. Additionally, we can also assume that $C_0$ is minimal with this property, i.e., for all $C_1\subsetneq C_0$ it holds that $AC_1E\oleq AC_1Eb$. Recall the notion of distance from \ref{gaifman_graph}, we then distinguish the following cases. The first three cases are exactly as in Proposition~\ref{psedoplane:HF_orders}.
	
	\smallskip 
	\noindent \underline{Case 1}. $c<b$ for all $c\in C_0$.
	\newline Then $AC_0E\noleq AC_0Eb$ contradicts the fact that $<$ is an $\mrm{HF}$-order of $B$ over $A$.
	
	\smallskip 
	\noindent \underline{Case 2}. $b<c$ for some $c\in C_0$ and  $d(c,b/A\mrm{cl}_<(CD)b^\downarrow)>1$.
	\newline Let $C_1=C_0\setminus \{c \}$. Since there is no relation between $c $ and $b $, it follows that $ b$ has the same Gaifman closure in $AC_0E$ and $AC_1E$, thus by  Condition \ref{context}(\ref{condition:K}) it follows that $AC_1E\noleq AC_1Eb$, contradicting the minimality of $C_0$.
	
	\smallskip 
	\noindent \underline{Case 3}. $b<c$ for some $c\in C_0$ and  $d_{\mrm{i}}(c,b/A\mrm{cl}_<(CD)b^\downarrow)= 1$.
	\newline Then $b\in \mrm{gcl}_{A\mrm{cl}_<(CD)}(c)$ and so, by the definition of $\mrm{HF}$-closure, it follows that $ b\in \mrm{cl}_<(CD)$, contradicting our choice of $b$.
	
	\smallskip 
	\noindent \underline{Case 4}. $b<c$ for some $c\in C_0$ and  $d_{\mrm{p}}(c,b/A\mrm{cl}_<(CD)b^\downarrow)= 1$.
	\newline Suppose also that $d_{\mrm{i}}(c,x/A\mrm{cl}_<(CD)b^\downarrow)= 0$ for all $x\in C_0$.	Then either $b$ is the least line in the ordering $<$ which is parallel to $c$, in which case we have $ b\in \mrm{cl}_<(CD)$ and we proceed exactly as in Case 3 above, or there are a line $\ell\parallel b \parallel c$ such that $\ell\in \mrm{cl}_<(CD)$ and a line $\ell'\parallel \ell$ such that $\ell'<b$. Let $C_1=C_0\setminus\{c\}$, then by the choice of $C_0$ we have that $AC_1E\oleq AC_1E b$. Since we can assume without loss of generality that $c$ is the unique element in $C_0$ parallel to $b$, and we have that there is no incidence between $b$ and elements of $C_0$, by Condition~(\ref{condition:K}) we can also assume without loss of generality that $C_1\subseteq b^\downarrow$. Then, since $<$ is a $\mrm{HF}$-order, we also obtain that  $AC_1E\ell'\oleq AC_0E\ell' b$ and, since $\ell'\parallel c\parallel b$, we get by Condition~(\ref{condition:K}) that $AC_1Ec\oleq AC_0Ec b$, contradicting $AC_0E\noleq AC_0Eb$.	
	
	\smallskip 
	\noindent We thus conclude that $<\restriction B\setminus(A\mrm{cl}_<(CD)) $ satisfies also Condition \ref{def_HF_order}(\hyperref[HF2]{H2}) and thus it is a $\mrm{HF}$-order of $B$ over $A\mrm{cl}_<(CD)$, completing our proof.
\end{proof}

 We next consider \ref{main_th}(\hyperref[completeness_axiom]{C1})-(\hyperref[the_K_homogeneous_lemma]{C2}). As in the previous cases we verify \ref{main_th}(\hyperref[completeness_axiom]{C1}) by showing  \ref{technical_assumptions}(\hyperref[trivial_condition]{D1}) and \ref{technical_assumptions}(\hyperref[minimality_condition]{D2}). We add a condition to the statement of \ref{technical_assumptions}(\hyperref[trivial_condition]{D1}) to ease our inductive argument. 

\begin{proposition}[{\ref{technical_assumptions}(\hyperref[trivial_condition]{D1})}]\label{Affine:trivial_condition}
	If $M \models T^+$, $A \subseteq M$ is finite, $<$ is a $\mrm{HF}$-order of $M$ and $A \oleq Ac \in \mathcal{K}$ is a trivial extension, then for all $n < \omega$ there is $ c'_n \in M$ s.t.:
	\begin{enumerate}[(a)]
		\item $Ac \cong_A Ac'$;
		\item $Ac' \oleq A\mrm{cl}_<^n(c')$;
		\item $A < \mrm{cl}_<^{n+1}(c')$.
	\end{enumerate}
\end{proposition}
\begin{proof} By induction on $n < \omega$ we prove that for every $A \oleq M$ and $A \oleq Ac \in \mathcal{K}$ we can find $c'_n$ as in (a)-(c). If $n = 0$, then it is easy to see that we can find find $c'_0$ as wanted. So suppose the inductive hypothesis and let $c$ be such that $A \oleq Ac \in \mathcal{K}$ is a trivial extension, we want to define $c'_{n+1}$ so that (a)-(c) hold. We assume without loss of generality that $c$ is a line, since the case for points is analogous.
	
	\smallskip	
	\noindent By inductive hypothesis we can a point $d^0_n$ which satisfies (a)-(c) with respect to $A$. Additionally, we can apply the induction hypothesis again to the set $A\mrm{cl}_<^n(d^0_n)$ and obtain a second point $d^1_n$ satisfying (a)-(c)  with respect to the set $A\mrm{cl}_<^{n}(d^0_n)$, whence in particular $A\mrm{cl}_<^{n}(d^0_n) < \mrm{cl}_<^{n+1}(d^1_n)$ and $A\mrm{cl}_<^{n}(d^0_n)d^1_n \oleq A\mrm{cl}_<^{n}(d^0_n)\mrm{cl}_<^n(d^1_n)$. We let $c'_{n+1}\coloneq d^0_n\lor d^1_n$ and we claim it satisfies conditions (a)-(c) for the value $n+1$. 
	
	\smallskip	
	\noindent  Firstly, notice that since $A\mrm{cl}_<^{n}(d^0_n) < \mrm{cl}_<^{n+1}(d^1_n)$ it follows immediately that $d^0_n<d^1_n<c'_{n+1}$. Then, it follows from the fact that $<$ is a $\mrm{HF}$-order of $M$ that the only elements incident to $c'_{n+1}$ occurring before it in the ordering $<$ are $d_n^0$ and $d^1_n$. By construction it follows immediately that $Ac \cong_A Ac'_{n+1}$. Additionally, we have:
	\begin{align*}
	\mrm{cl}_<^{n+2}(c'_{n+1})&= \{ c'_{n+1}\}\cup \mrm{cl}_<^{n+1}(d^0_n)\cup\mrm{cl}_<^{n+1}(d^1_n), \\
	\mrm{cl}_<^{n+1}(c'_{n+1})&= \{ c'_{n+1}\}\cup \mrm{cl}_<^{n}(d^0_n)\cup\mrm{cl}_<^{n}(d^1_n).
	\end{align*}
	Then, it follows immediately from the induction hypothesis and the choice of $d^0_n,d^1_n$ that $A<\mrm{cl}_<^{n+2}(c'_{n+1})$. Finally, we define a $\mrm{HF}$-order  $<'$ witnessing $Ac'_{n+1}\oleq A\mrm{cl}_<^n(c'_{n+1})$. It suffices to start from $Ac'_{n+1}$ and let $c'_{n+1}<'d_n^1<'d^0_n$ (by the choice of our elements this is a $\mrm{HF}$-order). As remarked above, $ c'_{n+1} $ is incident only to $d^0_n$ and $d^1_n$ in $\mrm{cl}_<^{n+1}(c'_{n+1})$, and moreover by our choice of $d^1_n$ no element in $\mrm{cl}_<^{n}(d^1_n)$ is incident to elements in $ \mrm{cl}_<^{n}(d^0_n) $. Thus, we can use the inductive hypothesis that $Ad^0_n \oleq A\mrm{cl}_<^n(d^0_n)$ and $A\mrm{cl}_<^{n}(d^0_n)d^1_n \oleq \mrm{cl}_<^{n}(d^0_n)\mrm{cl}_<^n(d^1_n)$ to continue the construction of a $\mrm{HF}$-order and obtain that $Ac'_{n+1}\oleq A\mrm{cl}_<^{n+1}(c'_{n+1})$. This completes our proof.
\end{proof}	

\begin{proposition}[{\ref{technical_assumptions}(\hyperref[minimality_condition]{D2})}]\label{Affine:minimality_condition}
	If $M \models T^+$ is $\aleph_1$-saturated, $A \oleq M$ is countable, and for every trivial extension $A\oleq Ab$ there is an element $b' \in M$ such that $Ab \cong_A Ab'\oleq M$, then for every minimal extension $A\oleq Ac$ there is an element $c' \in M$ such that $Ac \cong_A Ac'\oleq M$.
\end{proposition}
\begin{proof}
	The proof is analogous to the case of Steiner systems (Proposition \ref{Steiner:minimality_condition}).
\end{proof}

We next verify Assumption \ref{main_th}(\hyperref[the_K_homogeneous_lemma]{C2}), which we do as in the previous cases by showing \ref{technical_assumptions}(\hyperref[extension]{D3}) and  simultaneously proving \ref{assumptions:no_prime}(\hyperref[F=C]{C5}). We recall the free completion process for affine planes from \cite[p.~754]{funk2}, and we point out that this is essentially equivalent to the one given by Definition \ref{general_free_amalgam}. Recall that, after \ref{parallelsim_classes}, given a partial affine plane $A$ we denote by $P(A)$ any set containing exactly one representative for every parallelism class.

\begin{definition}[Free completion]	
	Let $A$ be a partial affine plane. We let $A_0=A$ and for every $n<\omega$ we define $A_{3n+1}, A_{3n+2}$ and $A_{3n+3}$ as follows.	
	\begin{enumerate}[(a)]
		\item First, for every pair of distinct points $p_0,p_1\in A_{3n}$ not incident to a common line in $A_{3n}$, we add a new line $\ell$ incident only to $p_0,p_1$ and not parallel to any other line from $A_{3n}$. We denote the resulting partial plane by $A_{3n+1}$.		
		\item Secondly, for every pair of lines $\ell_0,\ell_1\in A_{3n+1} $ which are not incident to any common point in $A_{3n+1}$ and which are not parallel, we add a new point $p$ incident only to $\ell_0$ and $\ell_1$. We denote the resulting partial plane by $A_{3n+2}$.
		\item Thirdly, for every line $\ell\in P(A_{3n+2})$ and every point $p\in A_{3n+2}$ such that $A_{3n+2}$ contains no line parallel to $\ell$ and incident to $p$, we add a new line satisfying these two conditions. We denote the resulting partial plane by $A_{3n+3}$.
	\end{enumerate}	
	The structure $F(A):=\bigcup_{n < \omega}A_n$ is called the \emph{free completion} of $A$ and we say that $F(A)$ is freely generated over $A$. We say that $A$ is \emph{non-degenerate} if $F(A)$ contains infinitely many elements of each sort, and \emph{degenerate} otherwise. 
\end{definition}

\begin{remark}\label{affine:non-degenerate}
	We notice that, as in the projective case, it is enough that $A$ contains four non-collinear points in order for it to be non-degenerate. Let $\mathbf{m}_\mathcal{K}=5$, then  in particular any partial affine plane $A$ which contains at least $\mathbf{m}_\mathcal{K}$ many elements of both sorts is non-degenerate. Thus the following proposition establishes \ref{main_th}(\hyperref[the_K_homogeneous_lemma]{C2}). We refer the reader to  \cite[p.~29]{schleiermacher_strambach_1} for a precise characterisation of non-degenerate partial affine planes.
\end{remark}

\begin{proposition}[{\ref{technical_assumptions}(\hyperref[extension]{D3})}, {\ref{assumptions:no_prime}(\hyperref[F=C]{C5})}]\label{free_theory_affine}
	Suppose $A\models T_\forall$, and $A$ is non-degenerate, then $F(A)\models T^+$, $A\oleq F(A)$, $|F(A)|=|A|+\aleph_0$ and:
	\begin{enumerate}[(a)]
			\item $F(A)$ is unique up to $A$-isomorphisms;
			\item $F(A)$ is $\oleq$-prime over $A$.
		\end{enumerate}
		In particular, it follows that if $A\cong B\onleq M\models T_n^+$ and $M$ is $\aleph_1$-saturated, then $F(A)\cong F(B)\preccurlyeq M$.
\end{proposition}
\begin{proof}
	The proof proceeds exactly as in the case of Steiner systems.
\end{proof}

	We have thus verified all the conditions from \ref{context}, \ref{main_th}(\hyperref[completeness_axiom]{C1})-(\hyperref[the_K_homogeneous_lemma]{C2}), \ref{ass:hf_closure}(\hyperref[the_hf_axiom]{C3}) and \ref{assumptions:no_prime}(\hyperref[F=C]{C5}). For what concerns the assumptions from \ref{assumption:no-superstability_2}(\hyperref[CP]{C4}) and \ref{assumptions:no_prime}(\hyperref[hopf]{C6})-(\hyperref[delta-rank]{C7}) we simply notice the following. The configuration from \cite[Cons.~9.1]{paolini&hyttinen} (or, after Remarks~\ref{remark:construction_principle}-\ref{ngons:prime_super}  the one from \cite[\S5.4]{HPQ}) does work (with minor modifications) also in the present setting of affine planes to establish \ref{assumption:no-superstability}(\hyperref[technical:D4]{D4}), and thus we do not provide the full details of the construction. Similarly, Assumption  \ref{assumptions:no_prime}(\hyperref[hopf]{C6}) can be proved by reasoning as in the case of Steiner systems, and Assumption \ref{assumptions:no_prime}(\hyperref[delta-rank]{C7}) can be verified by using the definition of $\delta(A)$ from \cite[Def.~11]{funk2}.

\subsection{Projective M\"obius, Laguerre and Minkowski planes}\label{last_section}

We conclude this article  by providing a last application of our framework to the setting of Benz planes. Differently from the previous examples, Benz planes are not a single instance of incidence geometry but rather a family of theories closely related to each other: M\"obius planes, Laguerre planes and Minkowski planes. These were originally introduced by Walter Benz in \cite{benz} and they axiomatise the structure induced by a 3-dimensional projective space on one of its quadratic sets (cf.~\cite[\S~5]{delandtsheer} and \cite[p.~21]{pasini}). As in the previous cases, we will study Benz planes from the synthetic point of view, essentially as a multigraph with points and blocks. From our current perspective, the main interest of this family of theories is that, additionally to the relation of incidence (and parallelism) they also involve a notion of tangency between blocks, which is a local equivalence relation parametrised by points. Crucially, it is important to keep in mind that there are in the literature different versions of Benz planes, and it is essential in our context that we consider \emph{projective} Benz planes (we elaborate this further in Remark \ref{projective_remark}). Additionally, we stress that for  brevity and simplicity we restrict our treatment only to the case of M\"obius planes. The key difference with Laguerre and Minkowski planes is simply that they add respectively one and two parallelism relations between points, and it is then straightforward to adapt our arguments to these settings. We refer the reader to \cite{benz} and \cite{funk2} for more details on Laguerre and Minkowski planes. As in the previous cases, our results in this section establish the following theorem.

\begin{theorem}\label{corollary:benz}
	Let $(\mathcal{K},\oleq)$ be the class of finite, open, partial  M\"obius planes (resp.~of Laguerre and Minkowski planes), and let $T^+$ be defined as in \ref{the_theory}. Then $T^+$ is complete, stable, and with $\ind=\ind^{\otimes}$. Moreover, $T^+$ is not superstable, it does not have a prime model, it is not model complete and it does not eliminate quantifiers.
\end{theorem}

We recall the axiomatisation of projective M\"obius planes from \cite[p.~748]{funk2},  and we introduce the associated class $(\mathcal{K},\oleq)$ of finite, open, partial M\"obius planes.

\begin{definition}
	Let $L$ be the two-sorted language where the symbols $p_0,p_1,\dots$ denote \emph{points} and the symbols $b_0,b_1,\dots$ denote \emph{blocks}, the symbol $p\,\vert \,b$ means that the point $p$ is incident to the block $b$, and the ternary tangency symbols $\rho(b_0,b_1,p)$ means that the two blocks $b_0,b_1$ are tangent in $p$. The theory $T$ of \emph{projective M\"obius planes} consists of the following axioms:
	\begin{enumerate}[(1)]
		\item  every three different points are incident to a unique block;
		\item  any two blocks tangent to the same point are incident to it, i.e.,
		\[\forall p\forall b_0 \forall b_1 \; (\rho(b_0,b_1,p_0) \to (p_0\, \vert\,b_0\land p_0\, \vert\,b_1));    \]
		\item given two points $p_1,p_2$ and a block $b_1$ such that $p_1$ is incident to $b_1$ but $p_2$ is not incident to $b_1$, there is a unique block $b_2$ incident to $p_2$ and s.t. $\rho(b_1,b_2,p_1)$ (we say this is the unique block incident to $p_2$ that ``touches'' $b_1$ in $p_1$);
		\item  for every point $c$, the relation induced by $\rho(x,y,c)$ is an equivalence relation;
		\item every two different blocks are either tangent in one common point, or they are incident to two common points.
	\end{enumerate}
	The universal fragment of the theory of projective M\"obius planes is axiomatised exactly by the axioms (2) and (4), by the axiom stating that two different blocks are incident with at most two common points, and by the uniqueness requirements from axioms (1) and (3) stating that the block incident to three given points and the touching block are unique. We refer to models of this universal theory as \emph{partial projective M\"obius planes}.
\end{definition}

Before defining what open M\"obius planes are, we introduce the following notion of valency in M\"obius planes. This should not be confused with the predimension function $\delta$, although, as one can then see from \cite[Def.~11]{funk2}, they are closely related.

\begin{definition}\label{moebius:valency} 
	Let $B\in \mathcal{K}$. Given a block $b\in B$, we let $T(b/B)$ be a maximal set of blocks $b'\in B$ such that $B\models \rho(b,b',p)$ for some $p\in B$ and every $b',b''\in T(b/B)$ are not tangent to each other.	Given an element $c\in B$ we write $I(c/B)$ for the set of elements incident to $c$ in $B$.  We then define the \emph{valency} of an element $c\in B$ as the number $\vartheta(c/B)$ given by 
		\[  \vartheta(c/B) \coloneqq |I(c/B)| + |T(c/B)|. \]
	Notice that, by Definition \ref{def:interior_closure}, we immediately have that $\vartheta(c/B)=| \mrm{gcl}_B(c) |$.
\end{definition}

\begin{definition}\label{Moebius:hyperfree}
	For finite partial projective M\"obius planes  $A\subseteq B$, we say that an element $a\in A$ is \emph{hyperfree} in $B$ if it satisfies one of the following conditions:
		\begin{enumerate}[(i)]
			\item  $a$ is a point with valency $\vartheta(a/B)\leq 2$;
			\item  $a$ is a block with valency $\vartheta(a/B)\leq 3$.
		\end{enumerate}
	We then write $A \oleq B$ if every non-empty $C \subseteq B \setminus A$ contains an element which is hyperfree in $AC$. We then adopt the notions of \emph{open}, \emph{closed}, \emph{confined}, etc. specialising Definition~\ref{def_extensions}.  We let $\mathcal{K}$ be the set of open partial projective M\"obius planes and we let $T_\forall$ be the universal theory of $\mathcal{K}$. The previous definitions extend to infinite models of $T_\forall$ as in \ref{infinite_strong extensions}. The notion of hyperfree elements in partial projective M\"obius planes is from \cite[p.~753]{funk2}.
\end{definition}

\begin{remark}\label{moebius:tangency_remark} \label{Moebius:algebraic_extension}
	We notice that, in particular, an extension $A\oleq Ac$ is algebraic if $c$ is a point with valency 2, or if $c$ is a block with valency $3$, i.e., if
	\begin{enumerate}[(i)]
		\item $c$ is a point incident with two blocks from $A$;	
		\item $c$ is a block incident to three points from $A$;	
		\item $c$ is a block incident to two points from $A$ and tangent to a block from $A$.
	\end{enumerate}
\end{remark}

\noindent  We proceed as in the previous cases  and we first consider the following key lemma concerning strong  algebraic extensions. 

\begin{lemma}\label{Moebius:algebraic_lemma}
	Let $A\oleq C\models T_\forall $ and let $c\in C$ be such that $A\oleq Ac$ is an algebraic strong extension. Then $Ac\oleq C$. 
\end{lemma}
\begin{proof}
	The proof follows analogously to the case of affine planes.
\end{proof}

We now consider the conditions from \ref{context}.  As before, conditions (\ref{condition:A}), (\ref{condition:B}), (\ref{condition:D}), (\ref{condition:E}), (\ref{condition:F}),  (\ref{condition:K}) and (\ref{condition:L}) follow immediately from our definition of the relation $\oleq$ for finite, open, partial M\"obius planes. Condition (\ref{condition:C}) follows exactly as in the case of Steiner systems. Condition (\ref{condition:I})  follows by letting $\mathbf{X}_\mathcal{K}=\{1\}$ and Condition (\ref{condition:J}) follows by letting $\mathbf{n}_\mathcal{K}=3$. Then, the algebraic amalgamation property from Condition (\ref{condition:G}) can be proved similarly to the case of affine planes. To showcase how to work with the relation of tangency, and also make it explicit that it behaves essentially as the parallelism relation of projective planes, we include the proof of Condition (\ref{condition:H}).

\begin{proposition}[\ref{context}(\ref{condition:H})]\label{Moebius:conditionH}
	If $A\subseteq B \in \mathcal{K}$ and $A \noleq B$ is confined, then there is $n < \omega$ such that every $C \in \mathcal{K}$ contains at most $n$ many disjoint copies of $B$ over $A$.
\end{proposition}
\begin{proof}
	Suppose $A \noleq B$ is confined, then every element in $B\setminus A$ is not hyperfree. In particular, for every point $p\in B$ we have $\vartheta(p/B)\geq 3$ and for every block in $b\in B$ we have that $\vartheta(b/B)\geq 4$.  For every element $c\in B\setminus A$ we write $I(c/A)$ for the set of elements in $A$ that are incident with $c$. Suppose  $C\in \mathcal{K}$ contains at least four copies $(B_i)_{i<4}$ of $B$ over $A$, then we claim that the set 
	\[ D= \bigcup_{c\in B_0\setminus A}I(c/A) \cup  \bigcup_{i<4}(B_i\setminus A)  \]
	witnesses that $C$ is not open. We consider the following elements.
	
	\noindent \underline{Case 1}. $d\in I(c/A)$ for some $c\in B\setminus A$.
	\newline Then, from the fact that $c$ is incident to $d$ and that $B_0,B_1,B_2,B_3$ are all isomorphic over $A$, it follows that $\vartheta(d/D)\geq 4$, showing that $d$ is not hyperfree in $D$.
		
	\noindent\underline{Case 2}. $c\in B_i\setminus A$ is a point, for $i\leq 3$.
	\newline By assumption $\vartheta(c/B_i)\geq 3$. Then, by Definition \ref{moebius:valency} and the definition of $D$, it follows immediately that $\vartheta(c/D)\geq 3$, whence $c$ is not hyperfree in $D$.
	
	\noindent \underline{Case 3}. $c\in B_i\setminus A$ is a block, for $i\leq 3$.
	\newline We let w.l.o.g. $i=0$, then by assumption $\vartheta(c/B_0)\geq 4$. We claim that $\vartheta(c/D)\geq 4$. Firstly, notice that every point in $B_0$ incident to $c$ belongs to $I(c/B_0)\cup(B_0\setminus A)\subseteq D$, thus $|I(c/B_0)|=|I(c/D)|$. Moreover, consider a block $d\in B_0$ such that $B_0\models \rho(c,d,p)$ for some point $p\in B_0$. If $p\in B_0\setminus A$ then we have that $d\in D$, since clearly $p$ is incident with $d$. If $p\in A$, then there is a block $c'\in B_1\setminus A$ such that $B_1\models \rho(c,c',p)$ and $B_1\models \rho(c',d,p)$. Therefore, $D$ contains a representative of all tangency classes parametrised by points in $B_0$. By Definition \ref{moebius:valency} it follows that  $\vartheta(c/D)\geq 4$ and so $c$ is not hyperfree in $D$. This completes our proof.
\end{proof}

We next consider Assumptions \ref{main_th}(\hyperref[completeness_axiom]{C1})-(\hyperref[the_K_homogeneous_lemma]{C2}) and Assumption \ref{ass:hf_closure}(\hyperref[the_hf_axiom]{C3}). We first verify \ref{ass:hf_closure}(\hyperref[the_hf_axiom]{C3}) and define the $\mrm{HF}$-closure operator $\mrm{cl}_<\coloneqq \bigcup_{n<\omega}\mrm{cl}^n_<$ in the setting of projective M\"obius planes.

\begin{definition}
	Let $A\subseteq B \models T_\forall^n$ and suppose $<$ is a $\mrm{HF}$-order of $B$ over $A$. For every block $b\in B$ and point $p\in B$ such that $p\,\vert \,b$ we let $\mrm{H}_{<}(b,p)$ be a family of blocks $(b_i)_{i\in I}$ all tangent to $b$  in $p$ and s.t.:
	\begin{enumerate}[(1)]
		\item $\mrm{H}_<(b,p)=\mrm{H}_<(b',p)$ for all blocks $b'\in B$ tangent to $b$ in $p$;
		\item  for all non-empty $J\subseteq I$ the set $\{b_j : j\in J\}$ has a maximal element with respect to the ordering $<$;
		\item for all $b'$ with $B\models \rho(b,b',p)$ there is some $b_i\in \mrm{H}_{<}(b,p)$ with $b_i=b'$ or $b_i<b'$.
	\end{enumerate}
\end{definition}

\begin{definition}\label{benz:closure_operator} Let $A\subseteq B \models T_\forall^n$ and suppose $<$ is a $\mrm{HF}$-order of $B$ over $A$. For all $C\subseteq B\setminus A$ we define $\mrm{cl}_<\coloneqq \bigcup_{n<\omega }\mrm{cl}_<^{n}(C)$ by letting $\mrm{cl}_<^0(C) \coloneqq C$ and, for all $n<\omega$, we define:
	\begin{enumerate}[(1)]
		\item $\mrm{gcl}^{\mrm{i}}_{A\mrm{cl}_<^{n}(C)^\downarrow}(\mrm{cl}_<^{n}(C)) \subseteq \mrm{cl}_<^{n+1}(C)$;
		\item if $b,p\in \mrm{cl}_<^{n}(C)$ and $\mrm{H}_<(b,p)$ has a least element $b'$, then $b'\in \mrm{cl}_<^{n+1}(C)$;
		\item if $b,p\in \mrm{cl}_<^{n}(C)$, there is $b_0<b$ with $\rho(b,b_0,p)$, and $b_1$ is the greatest element in $\mrm{H}_<(b,p)$ which does not belong to  $\mrm{cl}_<^{n}(C)$, then $b_1\in \mrm{cl}_<^{n+1}(C)$;		
		\item no other element belongs to $\mrm{cl}_<^{n+1}(C)$.
	\end{enumerate}
\end{definition}

\begin{proposition}[{\ref{ass:hf_closure}(\hyperref[the_hf_axiom]{C3})}]\label{benz:HF_orders}
	Every model $A\subseteq B\models T^n_\forall$ with an associated $\mrm{HF}$-order $<$ of $B$ over $A$ has a $\mrm{HF}$-closure operator $\mrm{cl}_<\coloneqq\bigcup_{n<\omega}\mrm{cl}^n_<$.
\end{proposition}
\begin{proof}
	This can be proved by reasoning similarly to how we proceed in Proposition~\ref{benz:HF_orders} for the case of affine planes.
\end{proof}

Then, to verify Assumption \ref{main_th}(\hyperref[completeness_axiom]{C1}) we reason as in the previous cases and we prove the technical conditions from \ref{technical_assumptions}(\hyperref[trivial_condition]{D1}) and \ref{technical_assumptions}(\hyperref[minimality_condition]{D2}). We add an extra condition to the statement of the next proposition to ease our inductive argument. 

\begin{proposition}[{\ref{technical_assumptions}(\hyperref[trivial_condition]{D1})}]\label{Moebius:trivial_condition}
	If $M \models T^+$, $A \subseteq M$ is finite, $<$ is a $\mrm{HF}$-order of $M$ and $A \oleq Ac \in \mathcal{K}$ is a trivial extension, then for $n < \omega$ there is $ c'_n \in M$ s.t.:
	\begin{enumerate}[(a)]
		\item $Ac \cong_A Ac'$;
		\item $Ac' \oleq A\mrm{cl}_<^n(c')$;
		\item $A < \mrm{cl}_<^{n+1}(c')$. 
	\end{enumerate}
\end{proposition}
\begin{proof} By induction on $n < \omega$ we prove that for every $A \oleq M$ and trivial extension $A \oleq Ac$ we can find $c'_n$ as in (a)-(c). If $n = 0$, then it is easy to see that we can find find $c'_0$ as wanted. So suppose the inductive hypothesis holds and let $c$ be such that $A \oleq Ac$ is a trivial extension, we find $c'_{n+1}$ so that (a)-(c) hold.
	
	\smallskip 
	\noindent \underline{Case 1}. $c$ is a block.
	\newline By inductive hypothesis we can find three points $p^0_n,p^1_n,p^2_n$ such that (a)-(c) hold respectively for $A$, $A\mrm{cl}^n_<(p^0_{n})$ and $A\mrm{cl}^n_<(p^0_{n})\mrm{cl}^n_<(p^1_{n})$. Let $c'_{n+1}$ be the unique block incident to $p^0_n,p^1_n,p^2_n$. By the assumption (c) and the choice of the elements it follows in particular that $c'_{n+1}>p^2_n>p^1_n>p^0_n$. We have:
	\begin{align*}
	\mrm{cl}^{n+2}_<(c'_{n+1})&=\{ c'_{n+1} \}\cup \mrm{cl}^{n+1}_<(p^0_{n})  \cup \mrm{cl}^{n+1}_<(p^1_{n}) \cup \mrm{cl}^{n+1}_<(p^2_{n}), \\
	\mrm{cl}^{n+1}_<(c'_{n+1})&=\{ c'_{n+1} \}\cup \mrm{cl}^{n}_<(p^0_{n})  \cup \mrm{cl}^{n}_<(p^1_{n}) \cup \mrm{cl}^{n}_<(p^2_{n}).
	\end{align*}
	Hence we immediately have that $A < \mrm{cl}_<^{n+2}(c'_{n+1})$ and also $Ac \cong_A Ac'_{n+1}$. We define a $\mrm{HF}$-order $<'$ to show that $Ac'_{n+1} \oleq A\mrm{cl}_<^{n+1}(c'_{n+1})$. We start by letting $c'_{n+1}<'p^2_n<'p^1_n<'p^0_n$. By the condition (b) and the choice of $p^0_n,p^1_n,p^2_n$ we have that $Ap^0_n \oleq A\mrm{cl}_<^n(p^0_n)$, $A\mrm{cl}_<^n(p^0_n)p^1_n\oleq A\mrm{cl}_<^n(p^0_n)\mrm{cl}_<^n(p^1_n) $ and $A\mrm{cl}^n_<(p^0_{n})\mrm{cl}^n_<(p^1_{n})p^2_n\oleq A\mrm{cl}^n_<(p^0_{n})\mrm{cl}^n_<(p^1_{n})\mrm{cl}_<^n(p^2_n)$. By condition (c) it follows that there is no incidence between elements of $\mrm{cl}^{n}_<(p^0_{n})$, $\mrm{cl}^{n}_<(p^1_{n})$ and $ \mrm{cl}^{n}_<(p^2_{n})$. Also, since $<$ is a $\mrm{HF}$-order, we have that  	$p^0_n,p^1_n,p^2_n$ are the only elements in $\mrm{cl}^{n+1}_<(c'_{n+1})$ to which $c'_{n+1}$ is incident.  It follows that we can extend the ordering $<'$ using the  induction hypothesis and show that $Ac'_{n+1} \oleq A\mrm{cl}_<^n(c'_{n+1})$.
	
	\smallskip	
	\noindent \underline{Case 2}. $c$ is a point.
	\newline  Similarly to Case~1, we use the inductive hypothesis to find two blocks $b^0_n$ and $b^1_n$ such that they satisfy conditions (a)-(c) with respect to $A$ and $A\mrm{cl}^n_<(b^0_{n})$. Then we let $c'_{n+1}$ be the unique block incident to $b^0_n$ and $b^1_n$. Then by reasoning as in Case~1 one can verify that conditions (a)-(c)  are satisfied.
\end{proof}		

\begin{proposition}[{\ref{technical_assumptions}(\hyperref[minimality_condition]{D2})}]\label{Moebius:minimality_condition}
	If $M \models T^+$ is $\aleph_1$-saturated, $A \oleq M$ is countable, and for every trivial extension $A\oleq Ab$ there is $b' \in M$ s.t. $Ab \cong_A Ab'\oleq M$, then for every minimal extension $A\oleq Ac$ there is $c' \in M$ s.t. $Ac \cong_A Ac'\oleq M$.
\end{proposition}
\begin{proof}
	This is proved by Lemma \ref{Moebius:algebraic_lemma} reasoning as in the previous cases.
\end{proof}

We next verify Assumption \ref{main_th}(\hyperref[the_K_homogeneous_lemma]{C2}), which we prove by establishing \ref{technical_assumptions}(\hyperref[extension]{D3}) and  simultaneously showing also \ref{assumptions:no_prime}(\hyperref[F=C]{C5}). We recall the free completion process for projective M\"obius planes from \cite[pp.~755-756]{funk2} and we leave to the reader to verify that this is equivalent to the construction given by \ref{free_algebraic_completion}.

\begin{definition}	\label{Moebius:free completion}	
	Let $A$ be a partial M\"obius planes. We let $A_0=A$ and for every $n<\omega$ we define $A_{4n+1}, A_{4n+2},A_{4n+3},A_{4n+4}$ as follows.	
	\begin{enumerate}[(a)]
		\item First, for every distinct points $p_0,p_1,p_2\in A_{4n}$ which are not incident to any common block, we add a new block $b$ incident only to $p_0,p_1,p_2$ and not tangent any other block from $A_{4n}$. We denote the resulting partial plane by $A_{4n+1}$.		
		\item Secondly, for every block $b\in A_{4n+1}$, every point $p$ incident to $b$ and every point $q$ not incident to it, we add a new block $b'$ incident only only to $p,q$ in $A_{4n+1}$, and tangent to $b$ exactly in $p$. We denote the resulting partial plane by $A_{4n+2}$.
		\item Thirdly, for every two non-tangent blocks $b_0,b_1\in A_{4n+2}$ incident to only one common point $p_0\in A_{4n+2}$, we add a new point $p_1$ incident to both $b_0,b_1$. We denote the resulting partial plane by $A_{4n+3}$.
		\item Finally, for every two non-tangent blocks $b_0,b_1\in A_{4n+2}$ incident to no common point in $A_{4n+3}$, we add two new points $p_0,p_1$ incident to both $b_0,b_1$. We denote the resulting partial plane by $A_{4n+4}$.
	\end{enumerate}	
	The structure $F(A):=\bigcup_{n < \omega}A_n$ is called the free completion of $A$ and we say that $F(A)$ is freely generated over $A$. We say that $A$ is \emph{non-degenerate} if $F(A)$ contains infinitely many elements of all sorts, and \emph{degenerate} otherwise. 
\end{definition}

\begin{remark}
	A partial projective M\"obius plane is non-degenerate if it contains at least one block, two points incident to it, and one point not incident to it  \cite[p.~756]{funk2}. It follows in particular that any set $A$ containing at least $\mathbf{m}_\mathcal{K}=4$ elements of every sort is non-degenerate. Thus the following proposition also establishes \ref{main_th}(\hyperref[the_K_homogeneous_lemma]{C2}).
\end{remark}

\begin{proposition}[{\ref{technical_assumptions}(\hyperref[extension]{D3})}, {\ref{assumptions:no_prime}(\hyperref[F=C]{C5})}]\label{free_theory_Moebius}
Suppose $A\models T_\forall$, and $A$ is non-degenerate, then $F(A)\models T^+$, $A\oleq F(A)$, $|F(A)|=|A|+\aleph_0$ and:
\begin{enumerate}[(a)]
		\item $F(A)$ is unique up to $A$-isomorphisms;
		\item $F(A)$ is $\oleq$-prime over $A$.
	\end{enumerate}
\end{proposition}
\begin{proof}
The proof proceeds exactly as in the case of Steiner systems.
\end{proof}

\begin{remark}\label{projective_remark}
	The definition of free completion process in M\"obius plane makes it clear why the projectivity requirement is essential in the context of Benz planes. In fact, if we drop Clause (d) from Definition \ref{Moebius:free completion}, what we obtain is the free completion process for arbitrary M\"obius planes (cf.~\cite[p.~755-756]{funk2}). Now, let $A$ be a partial M\"obius plane containing two blocks $b_0,b_1$ and three points $p_0,p_1,p_2$ such that $p_0,p_1$ are incident only to $b_0$, and $p_2$ is not incident with neither $b_0$ nor $b_1$. Let $F_{\mrm{aff}}(A)$ be the free \emph{affine} completion of $A$, i.e., the model obtained by closing $A$ under all the clauses from Definition \ref{Moebius:free completion} with the exception of Clause (d). Then  $F_{\mrm{aff}}(A)$ satisfies all the axioms of the theory of infinite, open M\"obius with the exception of the projectivity requirement, since the blocks $b_0$ and $b_1$ have no common incidence point in $F_{\mrm{aff}}(A)$. It follows that $F_{\mrm{aff}}(A)$ and $F(A)$ are two models of the theory of infinite open M\"obius planes without the projectivity axiom which are \emph{not} elementarily equivalent. This shows that the projectivity requirement  is essential to obtain a complete theory $T^+$. Notice that this is reflected in the proof of Assumption \ref{technical_assumptions}(\hyperref[trivial_condition]{D1}), since in Case~2 we crucially used that $M$ is projective.
\end{remark}

We thus have verified all the conditions from \ref{context}, and the assumptions from \ref{main_th}(\hyperref[completeness_axiom]{C1})-(\hyperref[the_K_homogeneous_lemma]{C2}), \ref{ass:hf_closure}(\hyperref[the_hf_axiom]{C3}), \ref{assumption:no-superstability_2}(\hyperref[CP]{C4}) and \ref{assumptions:no_prime}(\hyperref[F=C]{C5}).  It remains to consider \ref{assumption:no-superstability_2}(\hyperref[CP]{C4}) and \ref{assumptions:no_prime}(\hyperref[hopf]{C6})-(\hyperref[delta-rank]{C7}). We notice that \ref{assumptions:no_prime}(\hyperref[hopf]{C6}) is verified by essentially the same argument as in the case of Steiner systems, and  \ref{assumptions:no_prime}(\hyperref[delta-rank]{C7}) can be verified by using the definition of $\delta(A)$ from \cite[Def.~11]{funk2}. The next proposition exhibits a configuration which satisfies the technical assumption \ref{assumption:no-superstability}(\hyperref[technical:D4]{D4}), and so by Proposition~\ref{prop:technical_superstability} it also establishes  \ref{assumption:no-superstability_2}(\hyperref[CP]{C4}). Together with the previous results, the following proposition finally completes the proof of Theorem \ref{corollary:benz}.

\begin{proposition}[{\ref{assumption:no-superstability}(\hyperref[technical:D4]{D4})}]
	There is a configuration $C\in \mathcal{K}$ satisfying the conditions from \ref{assumption:no-superstability}.
\end{proposition}
\begin{proof} 
We define a configuration $C$ as follow. Let $C=\{c_i : 0\leq i \leq 16 \}$, where $c_0,c_1,c_2,c_3,c_8,c_{9},c_{10},c_{12},c_{14},c_{16}$ are points, $c_4, c_5, c_{6}, c_7,c_{11},c_{13},c_{15}$ are blocks, and  the incidence relations between elements of $C$ is specified by the following table. 

\begin{table}[H]
	\begin{tabular}{|c||c|c|c|c|c|c|c|} \hline
		& $c_4$ & $c_5$ & $c_6$ & $c_7$ & $c_{11}$ & $c_{13}$ & $c_{15}$ \\ \hline \hline
		$c_0$  & $\times$ &   & $\times$ & $\times$ &    &    &    \\ \hline
		$c_1$  & $\times$ & $\times$ &   & $\times$ &    &    &    \\ \hline
		$c_{2}$  & $\times$ & $\times$ & $\times$ &   &    &    &    \\ \hline
		$c_3$  &   & $\times$ & $\times$ & $\times$ &    &    &    \\ \hline
		$c_8$  & $\times$ &   &   &   & $\times$ &    & $\times$  \\ \hline
		$c_9$  &   & $\times$ &   &   & $\times$  &    & $\times$  \\ \hline
		$c_{10}$ &   &   & $\times$ &   & $\times$  & $\times$  &    \\ \hline
		$c_{12}$ &   &   &   & $\times$ & $\times$  & $\times$  &    \\ \hline
		$c_{14}$ & $\times$ &   &   &   &    & $\times$  & $\times$  \\ \hline
		$c_{16}$ &   &   &   &   &    & $\times$  & $\times$ \\ \hline
	\end{tabular}
\end{table}

\noindent Crucially, recall that we are working in the projective setting, thus we can always find \emph{two} incidence points between two non-tangent blocks. Then, it follows that $C$ is a partial M\"obius plane and the order $c_0<c_1<\dots<c_{16}$ is a $\mrm{HF}$-order, which means that $C\in \mathcal{K}$. As the other requirements from  \ref{assumption:no-superstability} are easily verified, it follows that $C$ witnesses that \ref{assumption:no-superstability} holds.
\end{proof}

\end{document}